\documentclass[11pt]{amsart}  % Specifies the document style.
\usepackage{amsthm, amsmath, amscd, amssymb, latexsym, stmaryrd, color}%txfonts 
\usepackage[top=2.5cm, bottom=2.5cm, left=2cm, right=2cm]{geometry}
\usepackage[colorinlistoftodos]{todonotes}
\theoremstyle{plain}

\newtheorem*{theorem*}{Theorem}
\newtheorem{theorem}{Theorem}[section]
\newtheorem{lem}[theorem]{Lemma}
\newtheorem{prop}[theorem]{Proposition}

\newtheorem{cor}[theorem]{Corollary}
\theoremstyle{definition}
\newtheorem{convention}[theorem]{Convention}
\newtheorem{defn}[theorem]{Definition}

\newtheorem{ex}[theorem]{Example}

\newtheorem{rmk}[theorem]{Remark}

\newtheorem{hypothesis}{Hypothesis}

\theoremstyle{remark}

\numberwithin{equation}{section}
\usepackage[mathscr]{eucal}
\usepackage{stackrel,amssymb}
\usepackage{graphics, graphpap}
\usepackage{array, tabularx, longtable}
\usepackage{color}

\usepackage{url}
\usepackage[T1]{fontenc}
\usepackage{array,epsfig}
\usepackage{amsmath}
\usepackage{old-arrows}
\usepackage{amsfonts}
\usepackage{amssymb}
\usepackage{amsxtra}
\usepackage{latexsym}
\usepackage{dsfont}
\usepackage[thinlines]{easytable}
\usepackage{makecell}
\usepackage{tabularray}
\usepackage{mathrsfs}
\usepackage{color}
\usepackage[all]{xy}
\usepackage{}
\usepackage{xfrac}
\usepackage{tikz}
\usepackage{tikz-cd}
\usepackage{graphicx}
\usepackage{mathtools}
\usepackage{caption}
\makeatletter
\newcommand\mathcircled[1]{%
  \mathpalette\@mathcircled{#1}%
}
\newcommand\@mathcircled[2]{%
  \tikz[baseline=(math.base)] \node[draw,circle,inner sep=1pt,color=red] (math) {$\m@th#1#2$};%
}
\makeatother

\usepackage[pagebackref=true]{hyperref}

\usepackage[OT2,T1]{fontenc}
\usepackage[page,toc,titletoc,title]{appendix}

\hypersetup{
 colorlinks,
 linkcolor=blue, citecolor=blue
 }
 
\newcommand{\ascr}{\mathscr{A}}

\newcommand{\gscr}{\mathscr{G}}

\newcommand{\mscr}{\mathscr{M}}

\newcommand{\pscr}{\mathscr{P}}

\newcommand{\tscr}{\mathscr{T}}
\newcommand{\uscr}{\mathscr{U}}

\newcommand{\acal}{\mathcal{A}}
\newcommand{\bcal}{\mathcal{B}}
\newcommand{\ccal}{\mathcal{C}}

\newcommand{\fcal}{\mathcal{F}}

\newcommand{\jcal}{\mathcal{J}}

\newcommand{\scal}{\mathcal{S}}
\newcommand{\tcal}{\mathcal{T}}

\newcommand{\rfrak}{\mathfrak{R}}

\newcommand{\xfrak}{\mathfrak{X}}

\newcommand{\hbf}{\mathbf{H}}
\newcommand{\hnor}{\textnormal{H}}

\newcommand{\inor}{\textnormal{I}}

\newcommand{\lnor}{\textnormal{L}}

\newcommand{\et}{\textnormal{\'et}}

\newcommand{\loccit}{\textit{loc.cit.}}

\newcommand{\an}{\textnormal{an}}

\newcommand{\mot}{\textnormal{mot}}

\newcommand{\ct}{\textnormal{ct}}
\newcommand{\ext}{\textnormal{ext}}
\newcommand{\aff}{\textnormal{aff}}
\newcommand{\pure}{\textnormal{pure}}
\newcommand{\univ}{\textnormal{univ}}

\newcommand{\id}{\operatorname{id}}

\newcommand{\Gal}{\operatorname{Gal}}
\newcommand{\Ker}{\operatorname{Ker}}

\newcommand{\pr}{\operatorname{pr}}
\newcommand{\Hom}{\operatorname{Hom}}
\newcommand{\Ext}{\operatorname{Ext}}

\newcommand{\affsch}{\operatorname{AffSch}}

\newcommand{\Alg}{\operatorname{Alg}}

\newcommand{\res}{\operatorname{Res}}
\newcommand{\infl}{\operatorname{Infl}}
\newcommand{\av}{\operatorname{Av}}

\newcommand{\daet}{\mathbf{DA}^{\et}}
\newcommand{\daetct}{\mathbf{DA}^{\et}_{\textnormal{ct}}}
\newcommand{\derivedcat}{\mathbf{D}}
\newcommand{\derivedcatct}{\mathbf{D}_{\textnormal{ct}}}
\newcommand{\derivednori}{\mathbf{DN}}

\newcommand{\betti}{\operatorname{Bti}}

\renewcommand{\pr}{\operatorname{pr}}
\newcommand{\heart}{\ensuremath\heartsuit}
\newcommand{\gm}{\mathbb{G}}
\newcommand{\phnor}{{}^p\mathrm{H}}
\newcommand{\cthnor}{{}^{\ct}\mathrm{H}}
\newcommand{\rat}{\operatorname{rat}}

\newcommand{\var}{\operatorname{Var}}

\newcommand{\shv}{\mathrm{Sh}}

\newcommand{\perv}{\operatorname{Perv}}
\newcommand{\inters}{\mathbf{IC}}
\newcommand{\localsystem}{\operatorname{Loc}}
\newcommand{\sets}{\mathbf{Sets}}
\newcommand{\op}{\operatorname{op}}

\newcommand{\colim}{\operatorname{colim}}

\newcommand{\Spec}{\operatorname{Spec}}

\renewcommand{\Im}{\operatorname{Im}}

\DeclareSymbolFont{cyrletters}{OT2}{wncyr}{m}{n}
\DeclareMathSymbol{\Sha}{\mathalpha}{cyrletters}{"58}
\DeclareMathSymbol{\Be}{\mathalpha}{cyrletters}{"42}

\newcommand{\flag}{\mathcal{F}\ell}

\title[The Motivic Satake Equivalence using Perverse Nori Motives]{The Motivic Satake Equivalence using Perverse Nori Motives}  % Declares the document's title.

\author[Khoa Bang. P]{Khoa Bang Pham}
\address{The Hong Kong Universtiy of Science and Technology, Clear Water Bay, Hong Kong}
\email{phamkb@ust.hk}
\thanks{}


\begin{document}           % End of preamble and beginning of text.
\begin{abstract}
  In this article, we develop the theory of stratified perverse Nori motives to prove a refinement of the geometric Satake equivalence of Mirkov\'ic-Vilonen, for which we call the Nori motivic Satake equivalence, in contrast to the "Tate motivic" Satake equivalence of Richarz-Scholbach.
\end{abstract}
\maketitle                 % Produces the title.

%************************

\section{Introduction}

\subsection{The (geometric) Satake equivalence}

In geometric representation theory and the geometric Langlands program, the philosophy of Langlands duality suggests that topological data associated with a split reductive group $G$ is encoded in the algebraic data of its Langlands dual group $G^{\vee}$, and vice versa. The Satake isomorphism \cite{satake-1963} (see also \cite{gross-1998}) establishes a foundational bridge, providing an isomorphism between the spherical Hecke algebra and the representation ring of the dual group. Through the sheaf-function correspondence, it becomes natural to expect a categorical enhancement of the Satake isomorphism, a result now known as the geometric Satake equivalence. The first complete proof of this equivalence was achieved in the work of Mirkovi\'c and Vilonen \cite{mirkovic+vilonen-2007}, using the Tannakian formalism and the theory of perverse sheaves. Their work is built upon ideas from several earlier sources, including \cite{lusztig-1983} \cite{ginzburg-2000}\cite{beilinson+drinfeld-unpublished}. Following this pioneering work, the geometric Satake equivalence has been intensively studied and generalized. Subsequent developments include the works \cite{richarz-2012} and \cite{zhu-2017} for $\ell$-adic sheaves (for further references, see also \cite{baumann+riche-2018}\cite{zhu-2016}), as well as the recent proof by Fargues and Scholze \cite{fargues+scholze-2024} using the theory of spatial diamonds. 

Currently, Richarz and Scholbach in \cite{richarz+scholbach-2021} provide a motivic refinement of the geometric Satake equivalence, called the motivic Satake equivalence. More precisely, (take base scheme to be a field $k$ for simplicity) the motivic Satake equivalence reads
\begin{equation*}
    \mathbf{MTM}(\lnor^+G \textbackslash \lnor G / \lnor^+G) \simeq \operatorname{Rep}_{\mathbb{Q}}( G^{\vee}_1 \rtimes \uscr^{\textnormal{Tate}}_{\mathbb{Q}})
\end{equation*}
where $\mathbf{MTM}$ denotes the category of mixed Tate motives, $\uscr_{\mathbb{Q}}^{\textnormal{Tate}}$ is the pro-algebraic unipotent group arising from extensions of mixed Tate motives, $G^{\vee}_1$ is the modified Langlands dual group (see for instance, \cite{zhu-2016} and references therein). In comparison to the geometric Satake equivalence of Mirkovi\'c and Vilonen, the work of Richarz and Scholbach takes mixed Tate motives as a model for perverse sheaves and hence the appearance of the factor $\uscr_{\mathbb{Q}}^{\textnormal{Tate}}$ is understandable. More precisely, their motivic Satake equivalence is built upon their previous work \cite{richarz+scholbach-2021} on intersection motives; again, using mixed Tate motives as a model for perverse sheaves. Let us first discuss Richarz and Scholbach's work in more depth, as it inspires us to write this manuscript. 

Intersection cohomology complexes play a central role in the theory of perverse sheaves, as they provide fundamental building blocks of the category. A motivic refinement of these complexes for moduli stacks of shtukas was introduced and studied in \cite{richarz+scholbach-2020}, independently of the standard conjectures. In their subsequent work \cite{richarz+scholbach-2021}, Richarz and Scholbach use these intersection motives to establish a motivic Satake equivalence as described above. Their approach models perverse sheaves using Tate motives (the subcategory generated by shifts of motives of the form $\mathds{1}(m)$), and requires both the Beilinson–Soulé conjecture and $\ell$-adic realizations. We identify several aspects of this Tate motivic framework that can be improved:
\begin{enumerate} 
\item (Restrictive geometric assumptions) The use of Tate motives requires stratifying (ind)-schemes into products of affine spaces, a condition that is highly restrictive even for ordinary schemes. While this suffices for certain applications, such as the motivic Satake equivalence, it limits the generality of the theory.
\item (A limited category of motives) Tate motives form a small part of the full category of motives. Accordingly, the motivic Satake equivalence in \cite{richarz+scholbach-2021} involves the motivic Galois group associated with Tate motives. It is natural to expect that a more comprehensive theory of motives would lead to a larger and more refined Galois group.
\item (Exclusion of non-trivial local systems) Since Tate motives are generated by the Tate objects $\mathds{1}(m)$, their simple objects are precisely these twists. These behave analogously to trivial local systems, thereby excluding non-trivial local systems from the framework. This leads to several issues: (1) it again results in the appearance of the Tate motivic Galois group rather than a full motivic Galois group; (2) it produces certain unnatural phenomena, such as smooth descent equivalences (see \cite[Lemma 2.10]{richarz+scholbach-2021}).
\item (Galois descent) The motivic Satake category defined in \cite[Section 6]{richarz+scholbach-2021} is independent of the base field. However, allowing non-trivial local systems should relate Satake categories via Galois actions associated with base field extensions, a phenomenon expected at the level of a full motivic theory but absent in the Tate motivic setting.
\end{enumerate} 

Furthermore, in \cite{richarz+scholbach-2021}, there is only the $\ell$-adic realization taken into account but not the Betti realization (see the new work \cite{cass-2025}). At a motivic level, we should expect both $\ell$-adic realizations (with $\ell$ varying) and Betti realization and these realizations are compatible. More importantly, in the context of the Satake equivalence, one should be able to "create" (in the language of \cite{richarz+scholbach-2020}) the motivic $t$-structure using either of realizations (though the resulting dual group can be different; just as topological fundamental groups and $\ell$-adic fundamental groups). 

In this article, we address these limitations using perverse Nori motives developed in \cite{florian+morel-2019}. This approach allows us to develop a theory of intersection motives that closely mirrors the classical theory of perverse sheaves, while remaining independent of the standard conjectures. Importantly, this framework contains Tate motives as a subcategory. As our main application, we construct an enhanced version of the motivic Satake equivalence from \cite{richarz+scholbach-2021}, which we call the Nori-motivic Satake equivalence. This provides a more general and geometrically natural framework than the Tate-motivic version. Since we work with all Nori motives, the resulting Satake category is expected to be as large as possible (at least within the category of étale motives $\daet(-,\mathbb{Q})$). Upon restriction to suitable subcategories, we recover corresponding variants of the Satake equivalence, including the Tate–Satake equivalence.

\subsection{Why perverse Nori motives?}

The theory of motives is a grand program envisioned by Grothendieck to encapsulate, within a single framework, the essential features shared by various cohomology theories developed by his school for smooth projective varieties over a field $k$, which are nowadays called Weil cohomology theories. Typical examples include $\ell$-adic cohomology, algebraic de Rham cohomology, and Betti cohomology. The notion of pure motives was introduced by Grothendieck, along with the expectation that there should exist a universal Weil cohomology theory reproducing all known properties of the existing ones. A natural candidate for the category of pure motives is the category of Chow motives introduced by Grothendieck. However, this approach immediately leads to the notorious standard conjectures, which remain unproven to this day. It is also natural to imagine that one can define motives for smooth but possibly non-projective varieties, thereby obtaining the notion of mixed motives $\mathrm{MM}(k)$, and then recover pure motives through other tools such as resolutions of singularities or semisimplifications. The existence of such a category was conjectured by Beilinson. Furthermore, one may expect the existence of a relative version $\mathrm{MM}(X)$, where $X$ varies over varieties over $k$. Instead of directly searching for such a category, Deligne proposes to first construct its derived version $\mathbf{D}^b(\mathrm{MM}(X))$ and then recover $\mathrm{MM}(X)$ via a motivic t-structure. This idea is made precise in \cite{beilinson-2012}. 

%On the other hand, the collection of $\derivedcat^b(\mathrm{MM}(X))$ should enhance the ordinary six operations (complex or $\ell$-adic)
%\begin{equation*}
  %  \begin{tikzcd}[sep=large]
   %    \derivedcat^b(\mathrm{MM}(X)) \arrow[rr,"f_!", bend left = 20,swap,""{name=U,inner sep=1pt,below}] \arrow[rr, "f_*", bend right = 20, ""{name=D,inner sep=1pt}]&   & \derivedcat^b(\mathrm{MM}(Y))  \arrow[ll,"f^*",bend left = 30] \arrow[ll, "f^!", bend right = 30,swap ] \arrow[Rightarrow,  from=U, to=D, shorten >=10pt, shorten <=10pt]
  %  \end{tikzcd} \ \ \ \ \ \ \ \ \begin{tikzcd}[sep=large]
   %    \derivedcat^b(\mathrm{MM}(X))  \arrow[rr,"\otimes", bend left = 20] &   & \derivedcat^b(\mathrm{MM}(X))  \arrow[ll,"\underline{\operatorname{Hom}}",bend left = 20]
  %  \end{tikzcd}
%\end{equation*}
A suitable candidate for the derived category of mixed motives is the category $\mathbf{DM}(k)$ constructed by Voevodsky in \cite{voevodsky-2000}. Along with the construction of $\mathbf{DM}(k)$, Voevodsky and Morel develop the so-called $\mathbb{A}^1$-homotopy theory of schemes (inspired by techniques from algebraic topology) in \cite{voevodsky+morel-1999}, and build the motivic stable homotopy category $\mathbf{SH}(k)$, in which motivic cohomology, algebraic K-theory, and algebraic cobordism are representable. In the thesis of Ayoub \cite{ayoub-thesis-1}\cite{ayoub-thesis-2}, it is shown that the constructions of Voevodsky and Morel can be unified into a general construction $\mathbf{SH}_{\mathfrak{M}}^{\tau}(X)$, where $\mathfrak{M}$ is a sufficiently good model category and $\tau$ is a topology on smooth varieties over $k$. If $\mathfrak{M} = \mathbf{Ch}(\mathbb{Q})$ is the category of chain complexes of $\mathbb{Q}$-vector spaces and $\tau = \textnormal{\'et}$ is the \'etale topology, then the constructible part $\mathbf{SH}_{\mathfrak{M}}^{\tau}(X) = \mathbf{DA}^{\textnormal{\'et}}_{\textnormal{ct}}(X,\mathbb{Q})$ is expected to be the derived category $\mathbf{D}^b(\mathrm{MM}(X))$. Meanwhile, we note that the six operations (as well as vanishing cycles and nearby cycles) for $\daet(X,\mathbb{Q})$ are realized in \cite{ayoub-thesis-1}\cite{ayoub-thesis-2}\cite{cisinski+deglise-2019}. Nowadays, after numerous work, one knows how to define the category $\operatorname{MM}(X)$. In \cite{florian+morel-2019}, Ivorra and Morel defines the category $\mscr\perv(X)$ of \textit{motivic perverse sheaves} (also called \textit{perverse Nori motives} as $\mscr\perv(k)$ recovers the category of Nori motives in \cite{nori-2011}). In \cite{tubach-2025}, Tubach proves that under the standard conjectures for fields of characteristic zero, $\mscr\perv(X)$ is exactly the desired category of mixed motives $\operatorname{MM}(X)$ and there is an equivalence of categories
\begin{equation*}
    \daetct(X,\mathbb{Q}) \simeq \derivedcat^b(\mscr\perv(X))
\end{equation*}
and they carry two $t$-structures, the constructible $t$-structure and the motivic (or perverse motivic) $t$-structure. In other words, philosophically, relying on standard conjectures, one can still choose one of two models -  $\daetct(X,\mathbb{Q})$ or $\derivedcat^b(\mscr\perv(X))$ - to study motives. The advantage of the approach taken by Ivorra-Morel is that it automatically provide us the heart which evidently behaves like the category of perverse sheaves and hence is appropriate for many geometrical constructions already done in classical settings. In this manuscript, we follow the Ivorra, Morel's approach to study a form of the geometric Satake equivalence.

\subsection{Formulation of main results} The first main result is inspired by stratified mixed Tate motives in the sense of \cite{richarz+scholbach-2020}. Let $k$ be a field together with a complex embedding $\sigma \colon k \longhookrightarrow \mathbb{C}$. For a $k$-variety $X$, we denote by $\mscr\perv(X)$ the category of motivic perverse sheaves on $X$ constructed in \cite{florian+morel-2019} and we write $\derivednori^b(X) = \derivedcat^b(\mscr\perv(X))$ its bounded derived category. The extension of $\derivednori^b(X)$ to ind-schemes is reviewed in the first section. 

Let $G$ be a split reductive group over $k$ with a choice of a Borel subgroup and a maximal torus $T \leq B$. By an abuse of notation, we implicitly choose an integral model (a Chevalley group with root datum of $G$) of $G$ and still denote by $G$ this group. Let $\mathbf{f} \subset  X_*(T) \otimes \mathbb{R}$ be a facet in the standard apartment. Let $\pscr_{\mathbf{f}},\pscr_{\mathbf{f}'}$ and $\flag_{G,\mathbf{f}},\flag_{G,\mathbf{f}'}$ be associated parahoric subgroups (in the sense of Bruhat-Tits \cite{bruhat+tits-1984}\cite{bruhat+tits-1987}) and partial affine flag varieties, respectively. Let $\pscr_{\mathbf{f}'}$ act on $\flag_{G,\mathbf{f}}$ canonically. The following result is analogous to \cite[\textbf{Theorem C}]{richarz+scholbach-2020}. 
\begin{theorem} 
Assume that $\mathbf{f},\mathbf{f}'$ are contained in the closure of the standard alcove defined by $B$. the (affine) Bruhat decomposition 
\begin{equation*}
    \flag_{G,\mathbf{f}} = \coprod_{w \in W_{\mathbf{f}'}\setminus W_{\ext} / W_{\mathbf{f}}} \flag_G^{w}(\mathbf{f}',\mathbf{f})
\end{equation*}
is a Whitney-Nori stratification and this yields:
\begin{enumerate}
\item (Non-equivariant category) There is a well-defined category of stratified Nori motives $\derivednori^b(\flag_{G,\mathbf{f}},\pscr_{\mathbf{f}'})^{\sigma}$ (the second variable $\pscr_{\mathbf{f}'}$ indicates the stratification into $\pscr_{\mathbf{f}'}$-orbits and the superscript indices the dependence on the complex embedding $\sigma \colon k \longhookrightarrow \mathbb{C}$) with a motivic perverse $t$-structure whose heart is stratified perverse Nori motives $\mscr\perv(\flag_{G,\mathbf{f}},\pscr_{\mathbf{f}'})^{\sigma}$ generated by intersection motives of the form $\mathbf{IC}_w(L)$ with $w \in W_{\mathbf{f}'} \setminus W_{\ext} / W_{\mathbf{f}}$ and $L \in \mscr\localsystem(\flag_G^{w}(\mathbf{f}',\mathbf{f}))^{\sigma}$ a motivic local system on the Schubert variety.
\item (Equivariant category) There is a well-defined category of equivariant, stratified Nori motives $\derivednori^b_{\pscr_{\mathbf{f}'}}(\flag_{G,\mathbf{f}},\pscr_{\mathbf{f}'})^{\sigma}$ with a motivic perverse $t$-structure whose heart is equivariant, stratified perverse Nori motives $\mscr\perv_{\pscr_{\mathbf{f}'}}(\flag_{G,\mathbf{f}},\pscr_{\mathbf{f}'})^{\sigma}$ generated by equivariant intersection motives of the form $\mathbf{IC}_w^{\pscr_{\mathbf{f}'}}(L)$ with $w \in W_{\mathbf{f}'} \setminus W_{\ext} / W_{\mathbf{f}}$ and $L \in \mscr\localsystem_{\pscr_{\mathbf{f}'}}(\flag_G^{w}(\mathbf{f}',\mathbf{f}))^{\sigma}$ an equivariant motivic local system on the Schubert variety. Moreover, the forgetful functor
\begin{equation*}
    \mscr\perv_{\pscr_{\mathbf{f}'}}(\flag_{G,\mathbf{f}},\pscr_{\mathbf{f}'})^{\sigma} \longrightarrow \mscr\perv(\flag_{G,\mathbf{f}},\pscr_{\mathbf{f}'})^{\sigma}
\end{equation*}
is fully faithful and the essentially image is given by 
\begin{equation*}
    \left \{M \in \mscr\perv(\flag_{G,\mathbf{f}},\pscr_{\mathbf{f}'})^{\sigma} \mid a^*(M) \simeq p^*(M) \right \}, 
\end{equation*}
where $a,p \colon \pscr_{\mathbf{f}'} \times \flag_{G,\mathbf{f}} \longrightarrow \flag_{G,\mathbf{f}}$ denote the action and the projection, respectively. 
\item (Galois descent) Let $F/k$ be a Galois extension, there are canonical equivalences of categories 
\begin{equation*}
\begin{split} 
    \left(\mscr\perv(\flag_{G,\mathbf{f}} \otimes F, \pscr_{\mathbf{f}'} \otimes F)^{\sigma} \right)^{\Gal(F/k)} & \simeq \mscr\perv(\flag_{G,\mathbf{f}}, \pscr_{\mathbf{f}'})^{\sigma} \\ 
      \left(\mscr\perv_{\pscr_{\mathbf{f}'} \otimes F}(\flag_{G,\mathbf{f}} \otimes F, \pscr_{\mathbf{f}'} \otimes F)^{\sigma}\right)^{\Gal(F/k)} & \simeq \mscr\perv_{\pscr_{\mathbf{f}'}}(\flag_{G,\mathbf{f}}, \pscr_{\mathbf{f}'})^{\sigma}.
    \end{split} 
\end{equation*}
\item (Realizations and Complex Embeddings) Let $\hbf$ be either the complex derived category $\mathbf{D}^b_{\ct}((-)^{\an},\mathbb{Q})$ or the $\ell$-adic derived category $\mathbf{D}^b_{\ct}((-)_{\et},\mathbb{Q}_{\ell})$ (for any $\ell$ prime), there are Betti and $\ell$-adic realizations
\begin{equation*}
    \begin{split}
        \derivednori^b_{\pscr_{\mathbf{f}'}}(\flag_{G,\mathbf{f}},\pscr_{\mathbf{f}'})^{\sigma} \longrightarrow \hbf_{\pscr_{\mathbf{f}'}}(\flag_{G,\mathbf{f}},\pscr_{\mathbf{f}'})
    \end{split}
\end{equation*}
and either of them "creates" the perverse $t$-structure on the left. Moreover, if $\sigma' \colon k \longhookrightarrow \mathbb{C}$ is another complex embedding, there is an equivalence of categories 
\begin{equation*}
    \begin{split}
        \mscr\perv(\flag_{G,\mathbf{f}},\pscr_{\mathbf{f}'})^{\sigma} & \simeq    \mscr\perv(\flag_{G,\mathbf{f}},\pscr_{\mathbf{f}'})^{\sigma'} \\ 
           \mscr\perv_{\pscr_{\mathbf{f}'}}(\flag_{G,\mathbf{f}},\pscr_{\mathbf{f}'})^{\sigma}  & \simeq  \mscr\perv_{\pscr_{\mathbf{f}'}}(\flag_{G,\mathbf{f}},\pscr_{\mathbf{f}'})^{\sigma'} 
    \end{split}
\end{equation*}
In other words, the categories $\mscr\perv(\flag_{G,\mathbf{f}}, \pscr_{\mathbf{f}}),\mscr\perv_{\pscr_{\mathbf{f}'}}(\flag_{G,\mathbf{f}}, \pscr_{\mathbf{f}})$ are motivic. 
\end{enumerate} 
\end{theorem}
The case $\mathrm{Gr}_G = \flag_{G,0}$ the affine Grassmannian and $\lnor^+G  = \pscr_0$ the positive loop group is of our main interest. The following result summarizes the structure of the categories $\mscr\perv(\mathrm{Gr}_G,\lnor^+G)$ and $\mscr\perv_{\lnor^+G}(\mathrm{Gr}_G,\lnor^+G)$ (besides what are described above) and provides an enhancement of the Tate-motivic Satake equivalence. 
\begin{theorem}
Let $G$ be a connected, split reductive group over a field $k$. Let $\mathrm{Gr}_G$ be its associated affine Grassmannian equipped with the natural action of the positive loop group $\lnor^+G$. The following hold true:
\begin{enumerate} 
\item Let $n \in \mathbb{Z}$, then the category $\mscr\perv(\mathrm{Gr}_G,\lnor^+G,n)$ of motives pure of weight $n$ is semisimple. In particular, the category 
\begin{equation*} 
\mscr\perv(\mathrm{Gr}_G,\lnor^+G,\pure) = \bigoplus_{n \in \mathbb{Z}}\mscr\perv(\mathrm{Gr}_G,\lnor^+G,n)
\end{equation*} 
is semisimple. The same applied to equivariant motives, one has a decomposition 
\begin{equation*} 
\mscr\perv_{\lnor^+G}(\mathrm{Gr}_G,\lnor^+G,\pure) = \bigoplus_{n \in \mathbb{Z}}\mscr\perv_{\lnor^+G}(\mathrm{Gr}_G,\lnor^+G,n)
\end{equation*} 
into semisimple categories.
\item Let $\scal \subset \mscr\localsystem(k)$ be a Tannakian subcategory of motivic local systems, there is a category of $\scal$-motives $\mscr\perv_{\lnor^+G}^{\scal}(\mathrm{Gr}_G,\lnor^+G) \subset \mscr\perv_{\lnor^+G}(\mathrm{Gr}_G,\lnor^+G)$ (here $\scal$ stands for "small") generated by intersection motives of "small" motivic local systems (for definition, see {\color{blue}section 7.1}). There is a canonical convolution product
\begin{equation*}
    (- ) \star (-) \colon \mscr\perv_{\lnor^+G}^{\scal}(\mathrm{Gr}_G,\lnor^+G) \times \mscr\perv^{\scal}_{\lnor^+G}(\mathrm{Gr}_G,\lnor^+G) \longrightarrow \mscr\perv^{\scal}_{\lnor^+G}(\mathrm{Gr}_G,\lnor^+G)
\end{equation*}
and a fiber functor
\begin{equation*}
    \omega \colon \mscr\perv^{\scal}_{\lnor^+G}(\mathrm{Gr}_G,\lnor^+G) \longrightarrow \operatorname{Vect}_{\mathbb{Q}}^{\textnormal{fd}}
\end{equation*}
making $\mscr\perv_{\lnor^+G}^{\scal}(\mathrm{Gr}_G,\lnor^+G)$ a neutral Tannakian category whose dual group is $G^{\vee}_{\mathbb{Q}} \times \gscr^{\mot}_{\scal}(k)$ with $G^{\vee}_{\mathbb{Q}}$ the Langlands dual group and $\gscr^{\mot}_{\scal}(k)$ the dual group of $\scal$, i.e., there is an equivalence of Tannakian categories
\begin{equation*}
    \mscr\perv^{\scal}_{\lnor^+G}(\mathrm{Gr}_G,\lnor^+G) \simeq \operatorname{Rep}^{\textnormal{fd}}_{\mathbb{Q}}(G^{\vee}_{\mathbb{Q}} \times \gscr^{\mot}_{\scal}(k)). 
\end{equation*}
The construction is functorial in the following sense: 
\begin{itemize}
    \item If $\scal_1 \subset \scal_2 \subset \mscr\localsystem(k)$ are Tannakian subcategories, then there is a canonical unital monoidal functor
    \begin{equation*}
        \mscr\perv_{\lnor^+G}^{\scal_1}(\mathrm{Gr}_G,\lnor^+G) \longrightarrow \mscr\perv_{\lnor^+G}^{\scal_2}(\mathrm{Gr}_G,\lnor^+G). 
    \end{equation*}
    \item If $e \colon \Spec(F) \longrightarrow \Spec(k)$ is a field extension of subfields of $\mathbb{C}$, then there is a canonical unital monoidal functor 
\begin{equation*}
     \mscr\perv^{\scal}_{\lnor^+G_k}(\mathrm{Gr}_{G_k},\lnor^+G_k)  \longrightarrow \mscr\perv^{e^*(\scal)}_{\lnor^+G_F}(\mathrm{Gr}_{G_F},\lnor^+G_F).
\end{equation*}
If $e$ is an \'etale morphism, then there is a canonical equivalence of categories
\begin{equation*}
    \mscr\perv^{\scal}_{\lnor^+G_k}(\mathrm{Gr}_{G_k},\lnor^+G_k) \simeq \mscr\perv^{e^*(\scal)}_{\lnor^+G_F}(\mathrm{Gr}_{G_F},\lnor^+G_F)^{\Gal(F/k)}.
\end{equation*}
\end{itemize}
\end{enumerate} 
\end{theorem}
The category $\mscr\perv_{\lnor^+G}^{\scal}(\mathrm{Gr}_G,\lnor^+G)$ is the largest possible choice (despite of its name "small motives") in the following sense: any reasonable sub-theory of motives inside Nori motives $\mathrm{M}(-) \subset \mscr\localsystem(k)$ yields a corresponding theory of geometric Satake equivalence. For instance, by restricting to:
\begin{enumerate}
    \item If $\scal = \mscr\localsystem(k)$, one obtains the dual group as $G^{\vee}_{\mathbb{Q}} \times \gscr^{\mot}(k)$. 
    \item If $\scal$ is the category of pure Tate motives, one obtains the dual group as $G^{\vee}_{\mathbb{Q}} \times \mathbb{G}_{m,\mathbb{Q}}$. 
\end{enumerate}
In particular, one can safely say that the ordinary geometric Satake equivalence (using either analytic or $\ell$-adic perverse sheaves) is the "Artin part" (i.e., the $0$-dimensional part) of the full motivic Satake equivalence. 

\subsection{Structure of the manuscript}

In {\color{blue}section 2}, we recall fundamental properties of perverse Nori motives and prove some results that are well-known for ordinary perverse sheaves but are not written yet for perverse Nori motives. In {\color{blue}section 3}, we study (stratified) perverse Nori motives in depths. In {\color{blue}section 4}, we define the (derived) equivariant categories of perverse Nori motives for schemes and then extend the framework to ind-schemes. In {\color{blue}section 5}, we study the category of perverse Nori motives on partial affine flag varieties associated with a split reductive group. In {\color{blue}section 6}, we concentrate to the case of affine Grassmannians. In {\color{blue}section 7}, we study various Satake categories and determine the dual groups.

\subsection*{Acknowledgements} The author gratefully acknowledges the support of the Hong Kong RGC GRF grants 16304923 and 16301324. He would like to express his deepest gratitude to Sophie Morel for pointing out a serious mistake in an early version of the paper. He is also thankful to Jakob Scholbach and Timo Richarz for their valuable feedback on earlier drafts. Furthermore, he thanks Joseph Ayoub and Luca Terenzi for answering his naive questions on (motivic) local systems during the preparation of this work. Special thanks go to Le Nhat Hoang and Pham Quang Toan for inspiring him and sharing their thoughts on geometric representation theory long before this project began. He is also grateful to Thiago Landim and Swann Tubach for helpful discussions on Tate motives and Nori motives for algebraic stacks.

Finally, the author thanks Ho Phu Quoc for his constant support as a postdoctoral supervisor, as well as colleagues Yuguo Zhang, Wenwei Liu, and Yang Hu for their encouragement.

\section{Perverse Nori Motives}

In an unpublished work (see for instance \cite{nori-2011} and see also \cite{arapura-2013}\cite{florian-2017}\cite{huber-book}\cite{ayoub+luca-2015}\cite{handbookofktheory}), Nori defines a candidate for the category of mixed motives as envisioned by Grothendieck. This category is now called the category of \textit{Nori motives} and it underlies a Tannakian structure, leading to a mysterious pro-algebraic group, called \textit{motivic Galois group}. Among the subsequent work, Ivorra's work \cite{florian-2017} is based on perverse sheaves and in \cite{florian+morel-2019}, Ivorra and Morel study the derived category of perverse Nori motives and prove that they acquire a formalism of four operations $(f^*,f_*,f_!,f^!)$. In \cite{terenzi-2025}, Terenzi completes the picture by constructing two remaining operations $(\otimes,\underline{\Hom})$. We take their results for granted throughout this work. 
\subsection{Motivic Perverse Sheaves (d'apres Ivorra-Morel-Terenzi)} We recall the definition of perverse motivic sheaves. We assume that the reader is familiar with the construction of the category of \'etale motives $\daet(X,\mathbb{Q})$ (see for instance \cite{ayoub-2014}) and the construction of the Betti and $\ell$-adic (also called \'etale) realizations (see \cite{ayoub-2010}\cite{ayoub-2014}). 
\begin{defn}
Let $\sigma \colon k \longhookrightarrow \mathbb{C}$ be a field embedded in $\mathbb{C}$. Let $X$ be a $k$-variety. We define the \textit{category of perverse motivic sheaves} $\mscr\perv(X)^{\sigma}$ as the universal abelian factorization
\begin{equation*}
    \begin{tikzcd}[sep=large]
              \daetct(X,\Lambda) \arrow[r,"\betti^*_X"] \arrow[d,"\phnor^0_{\univ}",swap] & \derivedcat_{\ct}^b(X^{\an},\mathbb{Q}) \arrow[r,"\phnor^0"] & \perv(X^{\an}) \\
              \mscr\perv(X)^{\sigma} \arrow[rru,"\rat_X",swap] &  & 
    \end{tikzcd}
\end{equation*}
Alternatively, by \cite[Proposition 6.11]{florian+morel-2019}, one can use the $\ell$-adic realization to obtain $\mscr\perv(X)$. The resulting category (regardless of the realization functor used) is independent of the choice of the prime $\ell$ and the embedding $\sigma$ (see \cite[Proposition 6.11]{florian+morel-2019}) so we may safely denote by $\mscr\perv(X)$ this category. The category
\begin{equation*}
    \derivednori^b(X) \coloneqq \derivedcat^b(\mscr\perv(X))
\end{equation*}
is called the \textit{category of derived Nori motives} (or \textit{mixed Nori motives}). 
\end{defn}
We recollect here some theoretically important results (see \cite{florian+morel-2019}\cite{terenzi-2024}\cite{tubach-2025}) that we cite frequently thoughout this manuscript. 
\begin{theorem} Let $\sigma \colon k \longrightarrow \mathbb{C}$ be a field embedded in $\mathbb{C}$. The following statements hold true:
\begin{enumerate}
    \item \textnormal{(Ivorra and Morel)} The collection $\derivedcat^b(\mscr\perv(X))$ with $X$ quasi-projective $k$-varieties forms a stable homotopical $2$-functor in the sense of \cite{ayoub-thesis-1}. In particular, there exists a formalism of four operations $(f^*,f_*,f_!,f^!)$.
    \item \textnormal{(Terenzi)} The collection $\derivedcat^b(\mscr\perv(X))$ with $X$ quasi-projective $k$-varieties has a six-functors formalism $(f^*,f_*,f_!,f^!,\mathds{1},\otimes,\underline{\Hom})$ compatible with the Betti realization and $\ell$-adic realization
\begin{equation*}
\begin{split} 
    \derivednori^b(X) & \longrightarrow \derivedcatct^b(X^{\an},\mathbb{Q}) \\
    \derivednori^b(X) & \longrightarrow \derivedcatct^b(X_{\et},\mathbb{Q}_{\ell})
    \end{split}
\end{equation*}
at the level of derived categories.
    \item \textnormal{(Tubach)} The categories $\derivedcat^b(\mscr\perv(X))$ and its six operations admit $(\infty,1)$-categorical enhancements and $\derivedcat^b(\mscr\perv(X))$ is a $h$-hypersheaf. Moreover, there is a second $t$-structure on $\derivedcat^b(\mscr\perv(X))$, called the constructible $t$-structure and the canonical functor
    \begin{equation*}
        \derivedcat^b(\mscr\perv(X)) \longrightarrow \derivedcatct^b(X,\mathbb{Q})
    \end{equation*}
    is constructible $t$-exact, which induces a faithful, exact functor on hearts
    \begin{equation*}
        \mscr\shv_{\ct}(X) \longrightarrow \shv_{\ct}(X,\mathbb{Q}),
    \end{equation*}
    where $ \mscr\shv_{\ct}(X) = \derivedcat^b(\mscr\perv(X))^{\heart}$ is called the category of motivic constructible sheaves and $\shv_{\ct}(X)$ is the category of ordinary constructible sheaves. There is also a motivic Beilinson equivalence
    \begin{equation*}
       \derivedcat^b(\mscr\perv(X)) \simeq \derivedcat^b(\mscr\shv_{\ct}(X)).
    \end{equation*}
\end{enumerate}
For technical reasons, we sometimes have to work with the ind-completion.
\begin{enumerate}
    \item [(4)] \textnormal{(Tubach)} Let $\derivednori(-) = \operatorname{Ind}(\mathbf{D}^b(\mscr\perv(-)))$ be the ind-completion. There exists a pair of Nori realizations
    \begin{equation*}
        (\operatorname{Nri}^* \dashv \operatorname{Nri}_*) \colon \daet(X,\mathbb{Q}) \longrightarrow \derivednori(X) 
    \end{equation*}
    and there exists an equivalence of $\infty$-categories
    \begin{equation*}
        \derivednori(X) = \operatorname{Mod}_{\operatorname{Nri}}(\daet(X,\mathbb{Q}))
    \end{equation*}
    where $\operatorname{Nri}$ is the Nori algebra defined in \cite{tubach-2025}. Moreover, $\mathbf{DN}(-)$ is also a $h$-hypersheaf (on $k$-varieties) and acquires a formalism of six operations (see \cite[Proposition 3.1.2]{ruimy-2025}) and the Nori realization is an equivalence if standard conjectures for fields of characteristic zero hold.
\end{enumerate}
\end{theorem}

\begin{rmk} \label{rmk: t-structures on ind-completions}
The categories $\mathbf{DN}(X)$ also admit canonical motivic perverse and constructible t-structures. Indeed, by \cite[Proposition 2.13]{antieau-2019}, for any stable $\infty$-category $\ccal$ with a t-structure, the ind-completion $\operatorname{Ind}(\ccal)$ has a t-structure $(\ccal^{\geq 0},\ccal^{\leq 0})$ given by $(\operatorname{Ind}(\ccal^{\geq 0}),\operatorname{Ind}(\ccal^{\leq 0}))$. The canonical functor $\derivednori^b(X) \longrightarrow \derivednori(X)$ is then t-exact for both perverse and constructible t-structures. The category $\derivednori^b(X)$ admits other intepretations inside $\mathbf{DN}(X)$ as follows (see \cite{ruimy-2025}):
\begin{enumerate}
    \item It is the full subcategory of $\derivednori(X)$ consisting of motives $M$ such that for any open subscheme $U \subset X$, there exists a finite stratification into locally closed subschemes $U = \coprod_{i \in I} U_i$ such that $M_{\mid U_i}$ is dualisable. 
    \item It is the full thick subcategory of $\derivednori(X)$ generated by motives of the form $f_!(\mathds{1}_Y)$ with $f \colon Y \longrightarrow X$ smooth. 
\end{enumerate}
\end{rmk}
\begin{rmk} \label{rmk: extensions to prestacks}
We sometimes have to deal with objects and functorialities that are not of finite type over $k$. The problem can be solved as follows. The categories $\derivednori^b(-)$ is a h-hypersheaf (see \cite[Proposition 3.1.5]{ruimy-2025}) and take values in $\infty\operatorname{Cat}^{\textnormal{st}}_{\textnormal{cont}}$, the $\infty$-category of stable, $\mathbb{Q}$-linear, $\infty$-categories with colimit preserving functors. Let us fix a sufficiently large regular cardinal $\kappa$ (in fact, we only need the case $\kappa = \mathbb{N}$), we have the following embeddings 
\begin{equation*}
    \var_k \longhookrightarrow \var_k^{\kappa} \longhookrightarrow \operatorname{Func}( \var_k^{\kappa}, \infty\operatorname{Grpd})
\end{equation*}
where the middle $\var_k^{\kappa}  = \operatorname{Pro}_{\kappa}( \var_k)$ is the category of $\kappa$-pro-objects and the right is the category of  $\infty$-presheaves $\operatorname{Func}(\var_k^{\kappa},\infty\operatorname{Grpd})$ for which we call $\kappa$-\textit{prestacks}, denoted $\operatorname{PreStk}_k^{\kappa}$. These categories are small and since $\infty\operatorname{Cat}^{\textnormal{st}}_{\textnormal{cont}}$ is bicomplete (see for instance \cite[Chapter I, Corollary 5.3.4]{gaitsgory-2017}) then by \cite[Theorem 5.1.3.6]{lurie-2009}, there exists a right Kan extension
\begin{equation*}
    \derivednori^b \colon  \operatorname{PreStk}_k^{\kappa} \longrightarrow \infty\operatorname{Cat}^{\textnormal{st}}_{\textnormal{cont}}.
\end{equation*}
By definition, for any morphism $f \colon X \longrightarrow Y$ in $\operatorname{PreStk}_k^{\kappa}$, there is a pullback $f^* \colon \derivednori^b(Y) \longrightarrow \derivednori^b(X)$ on compact objects. We will be particularly interested in the case of ind-varieties, pro-affine varieties and quotient stacks.
\end{rmk}
By these results, we have two collections of cohomological functors
\begin{equation*}
    \phnor^i \colon \derivednori^b(X) \longrightarrow \mscr\perv(X)  \ \ \ \ \text{and} \ \ \ \  \cthnor^i \colon \derivednori^b(X) \longrightarrow \mscr\shv_{\ct}(X) . 
\end{equation*}
The reader should not be confused when we call the first, the standard $t$-structure (according to the way it is defined) $\derivedcat^b(\mscr\perv(X))$ the \textit{motivic perverse $t$-structure} (or simply motivic $t$-structure) and the second $t$-structure the \textit{motivic constructible $t$-structure}. For the reader's convenience, we record here standard properties of perverse Nori motives. Most of the usual properties known for ordinary perverse sheaves still hold for perverse Nori motives and what is powerful in the Nori setting is that the Tannakian structure is richer.
\begin{enumerate}
    \item The category $\mscr\perv(X)$ is $\mathbb{Q}$-linear abelian category. \item If $f \colon X \longrightarrow Y$ is affine, quasi-affine, then there exist exact functors 
    \begin{equation*}
        f_!,f_* \colon \mscr\perv(X) \longrightarrow \mscr\perv(Y).
    \end{equation*}
    \item If $f \colon X \longrightarrow Y$ is smooth, then there exists an exact functor $f^{\dagger} = f^*[d_f] \colon \mscr\perv(Y) \longrightarrow \mscr\perv(X)$.
    \item There is a Verdier duality $\mathbb{D} \colon \mscr\perv(X) \longrightarrow \mscr\perv(X)^{\op}$ commuting with $f^{\dagger}$ for any smooth morphism $f$. 
    \item There exists a unique weight structure on each $\mscr\perv(X)$. Moreover, let $\mscr\perv(X,n)$ be the full subcategory of $\mscr\perv(X)$ containing motives pure of weight $n$, then $\mscr\perv(X,n)$ is semi-simple. In particular, the full subcategory of pure motives
    \begin{equation*}
        \mscr\perv(X,\pure) = \left \{A \in \mscr\perv(X) \mid A \ \textnormal{pure of some weight} \right \}
    \end{equation*}
    is a semi-simple category and it is the maximal semi-simple subcategory of $\mscr\perv(X)$ since any object admits a weight filtration.
    \item There is a canonical equivalence $\mscr\perv(\Spec(k)) \simeq \mathbf{HM}(k)$, with $\mathbf{HM}(k)$ the category of Nori motives and hence underlies a pro-algebraic group, namely, the motivic Galois group $\gscr^{\mot}(k)$. The category $\mscr\perv(\Spec(k),\pure)$ is a semi-simple category and equivalent to the category of Andr\'e motives (see for instance \cite{huber-book}). The dual group  $\gscr_{\pure}^{\mot}(k)$ is pro-reductive and there exists an exact sequence
    \begin{equation*}
        1 \longrightarrow \uscr^{\textnormal{mot}}(k) \longrightarrow \gscr^{\mot}(k) \longrightarrow \gscr_{\pure}^{\mot}(k) \longrightarrow 1
    \end{equation*}
    with $\uscr^{\textnormal{mot}}(k)$ pro-unipotent. By the weak Tannakian formalism of Ayoub (see \cite{ayoub-hopf1}\cite{ayoub-hopf2}), there is another motivic Galois group and it is isomorphic to $\gscr^{\mot}(k)$ by the work of Choudhury and Gallauer \cite{choudhury+gallauer-2017}. 
    \item As explained in \cite{florian+morel-2019}\cite{terenzi-2024}, there exists a theory of weights on perverse Nori motives and they share the same formal properties with ordinary $\ell$-adic perverse sheaves (see \cite{bbd}). Let $f \colon X \longrightarrow Y$ be a $k$-morphism. 
    \begin{itemize}
    \item The functors $f^*,f_!$ send $\derivednori^b(X,\leq w)$ to $\derivednori^b(Y,\leq w)$. 
    \item The functors $f_*,f^!$ send $\derivednori^b(Y,\geq w)$ to $\derivednori^b(X,\geq w)$. 
    \item The functor $(-) \otimes (-)$ sends $\derivednori^b(X,\leq w) \times \derivednori^b(X,\leq v)$ to $\derivednori^b(X,\leq w+v)$. \item The functor $\underline{\Hom}$ sends $\derivednori^b(X,\leq w) \times \derivednori^b(X,\leq v)$ to $\derivednori^b(X,\leq v-w)$. 
    \item Let $X,Y$ be $k$-varieties, then the box product
\begin{equation*}
    (-) \boxtimes (-) \colon \derivednori^b(X) \times \derivednori^b(Y) \longrightarrow \derivednori^b(X\times_k Y)
\end{equation*}
is weight-exact.
    \end{itemize}
   This weight structure is transversal to the canonical $t$-structure in the sense of \cite{bondarko-2012} and if $A, B \in \derivedcat^b(\mscr\perv(X))$ with $K$ is of weight $\leq w$ and $L$ is of weight $> w$, then 
    \begin{equation*}
        \Hom_{\derivedcat^b(\mscr\perv(X))}(A,B) = 0.
    \end{equation*}
    In particular, 
    \begin{equation*}
        \Ext^r_{\mscr\perv(X)}(A,B) = 0
    \end{equation*}
    if $A,B$ are pure of weight $i,j$ and $i < j +r$. 
    \item Let $j \colon U \longrightarrow X$ be a quasi-finite morphism between $k$-varieties. The intermediate extension functor $j_{!*} \colon \mscr\perv(U) \longrightarrow \mscr\perv(X)$ is defined by 
    \begin{equation*}
    j_{!*} \coloneqq \Im\left(\phnor^0(j_!) \longrightarrow \phnor^0(j_*) \right).
    \end{equation*}
    From \cite{florian+morel-2019}, if $j$ is an open immersion then $j_{!*}$ preserves weights and hence defines an exact functor $j_{!*} \colon \mscr\perv(U,\pure) \longrightarrow \mscr\perv(X,\pure)$. Moreover, (see for instance, \cite{kiehl+weissauer}) 
    \begin{equation*}
        \Hom_{\mscr\perv(U)}(A,B) = \Hom_{\mscr\perv(X)}(j_{!*}(A),j_{!*}(B))
    \end{equation*}
    for $A,B \in \mscr\perv(U)$. 
\end{enumerate}

%\subsection{Motivic Perverse Sheaves with general coefficients} In \cite{florian+morel-2019}, Ivorra and Morel defines the category $\mscr\perv(X)$ as the universal abelian factorization of 
%\begin{equation*}
%    \daetct(X,\mathbb{Q}) \longrightarrow \derivedcatct^b(X,\mathbb{Q}) \longrightarrow \perv(X,\mathbb{Q}).
%\end{equation*}
%Thus in a sense, the category $\mscr\perv(X)$ can be thought as the category of motivic perverse sheaves $\mscr\perv(X,\mathbb{Q})$ with $\mathbb{Q}$-coefficients. Follow the same strategy, we can define define category of motivic perverse sheaves with $\Lambda$-coefficients, where $\Lambda$ is a reasonable ring of coefficients. Remind that in \cite{ayoub-2010}, Ayoub defines the Betti realization
%\begin{equation*}
%    \daetct(X,\Lambda) \longrightarrow \derivedcatct^b(X,\Lambda) 
%\end{equation*}
%for any ring $\Lambda$. Let $\Lambda$ be a ring of finite global dimension,   and hence$\perv(X,\Lambda)$ is defined. We define the category $\mscr\perv(X,\Lambda)$ to be the universal abelian factorization of 
%\begin{equation*}
%    \daetct(X,\Lambda) \longrightarrow \derivedcatct^b(X,\Lambda) \longrightarrow \perv(X,\Lambda). 
%\end{equation*}
%\begin{prop}
%Let $E/F$ be a field extension, there is a natural action of $\Gal(E/F)$ on $\mscr\perv(X,E)$. Moreover, if $E/F$ is Galois, there is an equivalence of categories
%\begin{equation*}
 %   \mscr\perv(X,F) \simeq \mscr\perv(X,E)^{\Gal(E/F)}.
%\end{equation*}
%\end{prop}
%\begin{proof}
 %  Note that there are equivalences: $\daetct($
%\end{proof}
\subsection{Motivic Local Systems} Recall from \cite{terenzi-2024} that there is a notion of \textit{motivic local system}: let $\sigma \colon k \longhookrightarrow \mathbb{C}$ be an embedding and let $X$ be a smooth, connected $k$-variety. The category $\mscr\operatorname{Loc}(X)^{\sigma}$ of motivic local systems is defined as the pullback
\begin{equation*}
    \begin{tikzcd}[sep=large]
             \mscr\operatorname{Loc}(X)^{\sigma} \arrow[d] \arrow[r] & \operatorname{Loc}(X)[-\dim(X)] \arrow[d] \\ 
             \mscr\perv(X) \arrow[r] & \perv(X).
    \end{tikzcd}
\end{equation*}
In other words, a motivic local system is a motivic perverse sheave mapped to a shifted local system under the Betti realization functor. The category $\mscr\localsystem(X)^{\sigma}$ does not depend on $\sigma$ thanks to \cite[Lemma 2.22]{jacobsen+terenzi-2025} so we may saftely denote by $\mscr\localsystem(X)$ this category. After a choice of a $k$-point $x \in X(k)$ then in \cite{terenzi-2024}, Terenzi shows that the category $\mscr\localsystem(X)$ is a neutral Tannakian category whose dual group  is the \textit{motivic fundamental group} $\gscr(X,x)$. If $X = \Spec(k)$, then $\gscr(X,x)$ is exactly the Nori motivic Galois group, which is isomorphic to the Ayoub motivic Galois group $\gscr^{\mot}(k)$ (see \cite{choudhury+gallauer-2017}).
\begin{prop} \label{prop: motivic local systems are independent of realizations} 
Let $X$ be a smooth $k$-variety. Let $\mscr\localsystem_{\ell}(X) \subset \mscr\perv(X)$ be the category of motivic local systems arising from the $\ell$-adic realization $\mscr\perv(X) \longrightarrow \perv_{\ell}(X,\mathbb{Q}_{\ell})$. Then $\mscr\localsystem_{\ell}(X) = \mscr\localsystem(X)$ under the equivalence $\mscr\perv_{\ell}(X) = \mscr\perv(X)$ in \cite[Proposition 6.11]{florian+morel-2019}. 
\end{prop}

\begin{proof}
  First we note that an ordinary (both analytic and $\ell$-adic) local system is nothing but a strongly dualisable object. Indeed, for analytic local systems, this is due to \cite[Lemma 1.2.9]{ayoub-2025} (see \cite[Lemme 2.45]{ayoub-hopf2}). Regarding $\ell$-adic local system, one can copies the proof of $\textnormal{(i)} \Leftrightarrow \textnormal{(ii)}$ in \cite[Lemma 1.2.9]{ayoub-2025} for each $\shv(X_{\et},\mathbb{Z}/\ell^n\mathbb{Z})$ and then pass to limit. The Betti realization and \'etale realization commutes with six operations and hence preserve and reflect dual objects, thus the results follow from \cite[Proposition 6.11]{florian+morel-2019}.
\end{proof}

\begin{rmk}
\begin{enumerate} 
\item If $X^{\an}$ is simply connected, hence both the topological and the \'etale fundamental groups are trivial and the result becomes trivial. 
\item We note that although the category of motivic local systems is independent of the choice of the realization functor, the dual groups arising from the $\ell$-adic realization is informally larger than the one arising from the Betti realization. The reason is the category of $\ell$-adic local systems is "larger" then the category of analytic local systems (unless in some special cases like being simply connected). For instance, consider the case $X = \mathbb{G}_{m,k}$, then $\pi_1(X^{\an}) = \mathbb{Z}, \pi_1^{\et}(X)=\hat{\mathbb{Z}}$, a rank one $\ell$-adic local system with $\overline{\mathbb{Q}}_{\ell} = \mathbb{C}$-coefficients amounts to a choice $a \in \mathbb{C}$ and it is induced from an analytic one if and only if $a \in \overline{\mathbb{Z}}_{\ell}^{\times}$, namely, the topological monodromy can be lifted to the \'etale monodromy. 
\end{enumerate} 
\end{rmk}

\begin{prop} \label{prop: perverse t-structures of trivial stratification}
Let $X$ be a smooth, connected $k$-variety. Consider 
\begin{align*}
    \derivednori^b_{\mscr\localsystem(X)}(X) & = \left \{M \in \derivednori^b(X) \mid \cthnor^i(M) \in \mscr\localsystem(X) \right \} \\ 
    & =  \left \{M \in \derivednori^b(X) \mid \phnor^i(M) \in \mscr\localsystem_p(X) \right \}
\end{align*}
then the perverse and constructible t-structures on $\derivednori^b(X)$ restrict to $\derivednori^b_{\mscr\localsystem(X)}(X)$ so that 
\begin{equation*}
    \begin{split}
         {}^p\derivednori^b_{\mscr\localsystem(X)}(X)^{\leq 0} & =  {}^{\ct}\derivednori^b_{\mscr\localsystem(X)}(X)^{\leq -\dim(X)} \\ 
           {}^p\derivednori^b_{\mscr\localsystem(X)}(X)^{\geq 0} & =  {}^{\ct}\derivednori^b_{\mscr\localsystem(X)}(X)^{\geq -\dim(X)}
    \end{split}
\end{equation*}
Their hearts are $\mscr\localsystem_p(X)$ and $\mscr\localsystem(X)$, respectively. Moreover, the category $\derivednori^b_{\mscr\localsystem(X)}(X)$ is stable Verdier duality and Tate twists. 
\end{prop}

\begin{proof}
  The restriction of the motivic perverse $t$-structure on $\derivednori^b(X)$ is a well-defined $t$-structure on $\derivednori^b_{\mscr\localsystem(X)}(X)$ since $\mscr\localsystem_p(X) \subset \mscr\perv(X)$ is a Serre subcategory. Clearly, $\derivednori^b_{\mscr\localsystem(X)}(X)$ is stable under Tate twists and for Verdier duality, we note that $\mathbb{D}(M) \in \derivednori^b_{\mscr\localsystem(X)}(X)$ if and only if $\phnor^i(\mathbb{D}(M)) \simeq \mathbb{D}(\phnor^i(M)) \in \mscr\localsystem_p(X)$ (because $\mathbb{D}$ is perverse $t$-exact) so it suffices to show that the Verdier duality $\mathbb{D}(L)$ of a shifted motivic local system $L$ is still a shifted motivic local system. However, $\rat_X(\mathbb{D}(L)) \simeq \mathbb{D}(\rat_X(L)) \simeq \rat_X(L)^{\vee}(\dim(X))[2\dim(X)]$ thanks to \cite[Remark 2.8.3]{achar-book} and we win. 
\end{proof}

\begin{lem} \label{lem: weight filtrations of motivic local systems}
Let $X$ be a smooth $k$-variety, then any object $M \subset \mscr\localsystem(X)[\dim(X)]$ admits a weight filtration whose graded pieces are in $\mscr\localsystem(X)[\dim(X)] \cap \mscr\perv(X,n)$ for $n \in \mathbb{Z}$.
\end{lem}

\begin{proof}
  Any object $M$ in $\mscr\perv(X)$ admits a weight filtration by \cite[Proposition 6.17]{florian+morel-2019} and if $M \subset \mscr\localsystem(X)[\dim(X)]$ and all terms in the filtration and their graded pieces are also in $\mscr\localsystem(X)[\dim(X)]$ because shifted local systems is a Serre subcategory of $\mscr\perv(X)$ (for instance, this holds for analytic perverse sheaves \cite[Proposition 3.4.1]{achar-book} and the motivic case follows immediately). 
\end{proof}

\begin{defn}
  Let $X$ be a smooth, connected $k$-variety. The \textit{category of Tate motives} $\mscr\operatorname{Tate}(X) \subset \mscr\localsystem(X)$ is the smallest full abelian subcategory containing all $\mathds{1}(n)$ and stable under extensions and direct summands. Alternatively, one can define $\mscr\operatorname{Tate}_p(X) \subset \mscr\perv(X)$ with the same properties.
\end{defn}

\begin{lem} \label{lem: pure Tate motives} 
A motive $M \in \mscr\localsystem(X)$ is a Tate motive if and only if its graded pieces are of the form $\mathds{1}(n)^{\oplus m}$. In particular, pure Tate motives (Tate motives that are pure) are of the form $\mathds{1}(n)^{\oplus m}$ (weight $-2n$). 
\end{lem}

\begin{proof}
  Let $M \in \mscr\localsystem(X)$ be a motive so that its graded pieces are of the form $\mathds{1}(n)^{\oplus m}$ then it belongs to $\mscr\operatorname{Tate}(X)$ for trivial reason. Conversely, if $M \in \mscr\textnormal{Tate}(X)$, then $M$ admits a filtration whose graded pieces are subquotients of $\mathds{1}(n)$ and by modifying the filtration, we see that $M$ can be obtained as a direct summand of iterated extensions of $\mathds{1}(n)$. The functors $\operatorname{gr}^W_n$ are exact so we see that graded pieces are direct summands of those $\mathds{1}_X(n)^{\oplus m}$, hence if we can show that $\mathds{1}(n)$ has no nontrivial direct summand then we win. The object $\mathds{1}_X(n)[\dim(X)]$ is a shifted motivic local systems $$ \Hom_{\mscr\perv(X)}(\mathds{1}_X(n)[\dim(X)],\mathds{1}_X(n)[\dim(X)]) =   \Hom_{\mscr\localsystem(X)}(\mathds{1}_X,\mathds{1}_X) = \mathbb{Q}$$ since $\mscr\localsystem(X)$ is a neutral Tannakian category as showed in \cite{terenzi-2024}. This proves that $\mathds{1}_X(n)$ has no proper direct summand.
\end{proof}

\begin{defn}
  The \textit{category of derived Tate motives} $\mathbf{DT}^b(X)$ is the smallest stable sub-$\infty$-category of $\derivednori^b(X)$ containing all motives of the form $\mathds{1}_X(n)$ and stable under direct summands. We note that there is a fully faithful embedding
  \begin{equation*}
      \mscr\operatorname{Tate}_p(X) \longhookrightarrow \mathbf{DT}^b(X)
  \end{equation*}
  obtained by restricting $\mscr\perv(X) \longhookrightarrow \mathbf{DN}^b(X)$. 
\end{defn}

%\begin{rmk}
%For the honest definition, one should add the condition that $\mathbf{DT}^b(X)$ also stable under direct summands. However, as noted in \cite[Lemma 2.3]{scholbach-2011}, we can ignore this (under the validity of the Beilinson-Soul\'e vanishing conjecture) since the endormorphism ring of $\mathds{1}(n)$ is one-dimensional. 
%\end{rmk}

\begin{lem} \label{lem: Tate cohomological functors on schemes}
The composition 
\begin{equation*}
    \mathbf{DT}^b(X) \longhookrightarrow \derivednori^b(X) \overset{\phnor^n}{\longrightarrow} \mscr\perv(X) 
\end{equation*}
takes values in $\mscr\textnormal{Tate}_p(X)$ for any $n \in \mathbb{Z}$. 
\end{lem}

\begin{proof}
  Note that $\phnor^n(\mathds{1}(n)[\dim(X)]) = \mathds{1}(n)[\dim(X)]$ if $n= 0$ and zero otherwise. Now if $M$ is obtained by successive extensions of $\mathds{1}(n)$, then we see that $\phnor^n(M)$ is also obtained by successive extensions of shifted Tate motives, hence itself a shifted Tate motive. 
\end{proof}

\begin{rmk}
It is however not clear if there exists a $t$-structure on $\mathbf{DT}^b(X)$ whose heart is $\mscr\operatorname{Tate}_p(X)$. We expect that at least over number fields, with the help of the Beilinson-Soul\'e vanishing conjecture, thí is the case. 
\end{rmk}
\subsection{Smooth Descent}
We recall the following result, which is proven in \cite{bbd}, and whose proof for perverse Nori motives is not written yet in \cite{florian+morel-2019}. The materials in this subsection will be useful when we discuss equivariant perverse Nori motives. 
\begin{prop} \label{prop: pullback of smooth morphism is fully faithful}
Let $X,Y$ be $k$-varieties. Let $f \colon X \longrightarrow Y$ be a smooth surjective morphism. The canonical sequence
\begin{equation*}
 0 \longrightarrow \Hom_{\mscr\perv(Y)}(M,N) \longrightarrow \Hom_{\mscr\perv(X)}(f^{\dagger}(M),f^{\dagger}(N))   \longrightarrow \Hom_{\mscr\perv(X \times_Y X)}((f')^{\dagger}(M),(f')^{\dagger}(N))
\end{equation*}
is exact. Moreover, if $f$ has geometrically connected fibers, then $f^{\dagger}$ is fully faithful. 
\end{prop} 

\begin{proof}
The morphism $f_{\dagger} = f_*[-d_f] \colon \derivedcat^b_{\ct}(X,\mathbb{Q}) \longrightarrow \derivedcat^b_{\ct}(Y,\mathbb{Q})$ is left exact since its left adjoint $f^{\dagger} = f^*[d_f]$ is exact. The compatibility of realization functor as well as its conservativity (on triangulated categories) and exactness (on abelian categories) of $\rat$ functors show that $f_*[-d_f] \colon \derivedcat^b(\mscr\perv(X)) \longrightarrow \derivedcat^b(\mscr\perv(Y))$ is left exact as well. They induce left adjoints of $f^{\dagger},f^{\dagger}$
\begin{equation*}
    \begin{split}
      {}^pf_{\dagger} =  \phnor^0 \circ f_{\dagger} \colon \mscr\perv(X) & \longrightarrow \mscr\perv(Y) \\ 
       {}^pf_{\dagger} = \phnor^0 \circ f_{\dagger} \colon \perv(X) & \longrightarrow \perv(Y).
    \end{split}
\end{equation*}
The exactness of the original sequence is equivalent to the exactness of
\begin{equation*}
    0 \longrightarrow N \longrightarrow {}^p f_{ \dagger}f^{\dagger}(N)\longrightarrow {}^p f_{ \dagger}'f^{'\dagger}(N)
\end{equation*}
Since $\rat_Y$ is faithful, exact functor, it reflects exact sequences so one can apply $\rat_Y$ to the sequence above, the result then follows from the corresponding statement for ordinary perverse sheaves. If moreover $f$ has geometrically connected fibers, then $f^{\dagger}$ is fully faithful if and only if $N \longrightarrow {}^pf_{ \dagger}f^{\dagger}(N)$ is an isomorphism for any $N \in \mscr\perv(Y)$. This again can be checked after applying $\rat_Y$ and using its conservativity. 
\end{proof} 

\begin{rmk}
For stratified mixed Tate motives, this is proven in \cite[Lemma 3.2.12]{richarz+scholbach-2020} with the help of the Beilinson-Soul\'e conjecture.
\end{rmk}

Next we need smooth descent for perverse Nori motives. 
\begin{defn} 
Let $f \colon X \longrightarrow Y$ be a smooth surjective morphism. Let $M \in \mscr\perv(X)$ be a perverse Nori motive. A \textit{descent datum} $(M,\phi)$ for $M$ is an isomorphism
\begin{equation*}
    \phi \colon \pr_1^{\dagger}(M) \simeq \pr_2^{\dagger}(M) 
\end{equation*}
in $\mscr\perv(X \times_Y X)$ such that the diagram 
\begin{equation*}
    \begin{tikzcd}
             \pr_1^{\dagger}(M) \arrow[rr] \arrow[rd] & & \pr_3^{\dagger}(M) \arrow[ld] \\ 
             &\pr_2^{\dagger}(M)
    \end{tikzcd}
\end{equation*}
is commutative. We can turn descent data into a category by declaring objects to be descent data and a morphism $(M,\phi) \longrightarrow (N,\phi)$ is a morphism $q \colon M \longrightarrow N$ in $\perv(X)$ making the obvious diagram commutative. 
\end{defn}
\begin{prop} \label{prop: smooth descent}
Let $f \colon X \longrightarrow Y$ be a smooth surjective morphism. For any $N \in \perv(Y)$, the motivic perverse sheaf $f^{\dagger}(N)$ admits a canonical descent datum and 
\begin{equation*} 
f^{\dagger} \colon \mscr\perv(Y) \longrightarrow \operatorname{Desc}_{\mscr\perv}(f)
\end{equation*} 
gives rise to an equivalence of categories.
\end{prop}

\begin{proof}
Thanks to the functoriality of the functors of type $(-)^{\dagger}$, the descent datum of $f^{\dagger}(N)$ is simply the isomorphism 
\begin{equation*}
    \pr_1^{\dagger}f^{\dagger}(N) \simeq (f')^{\dagger}(N) \simeq \pr_2^{\dagger}f^{\dagger}(N).
\end{equation*}
The functoriality again ensures that a morphism $M \longrightarrow N$ in $\mscr\perv(Y)$ produces a morphism in $\operatorname{Desc}_{\mscr\perv}(f)$. Let us first prove that $f^{\dagger} \colon \mscr\perv(Y) \longrightarrow \operatorname{Desc}_{\mscr\perv}(f)$ is fully faithful. There is an obvious commutative diagram 
\begin{equation*}
    \begin{tikzcd}[row sep = large, column sep = scriptsize]
             0 \arrow[r] & \Hom_{\mscr\perv(X)}(M,N) \arrow[r,"f^{\dagger}"]\arrow[d] & \Hom_{\mscr\perv(Y)}(f^{\dagger}(M),f^{\dagger}(N)) \arrow[d,equal]  \arrow[r] & \Hom_{\mscr\perv(X \times_Y X)}((f')^{\dagger}(M),(f')^{\dagger}(N)) \arrow[d,equal] \\ 
             0 \arrow[r] & \Hom_{\operatorname{Desc}_{\mscr\perv}(f)}(f^{\dagger}(M),f^{\dagger}(N)) \arrow[r] &  \Hom_{\mscr\perv(Y)}(f^{\dagger}(M),f^{\dagger}(N)) \arrow[r] & \Hom_{\mscr\perv(X \times_Y X)}(\pr_1^{\dagger}f^{\dagger}(M),\pr_2^{\dagger}f^{\dagger}(N)).
    \end{tikzcd}
\end{equation*}
The top row is exact by \ref{prop: pullback of smooth morphism is fully faithful} while the exactness of the bottom row is simply the definition of descent data. Hence, by the five lemma the most-left vertical arrow is an isomorphism and we win. It remains to prove that $f^{\dagger}$ is essentially surjective, the proof is as same as the proof of \cite[Theorem 3.7.4]{achar-book}. 
\end{proof}

\section{Derived Nori Motives on Ind-Schemes}

\subsection{Reminder on ind-schemes and their stratifications} For an introduction on ind-schemes, we refer to \cite{richarz-2019}. Let us briefly review terminologies used in this article. By an \textit{ind-scheme}, we mean an \'etale sheaves $X \colon (\affsch_k)^{\op} \longrightarrow \sets$ that admits a presentation $X = \colim_{i \in I} X_i$ as a filtered colimit of schemes along closed immersions $\iota_{i \to j} \colon X_i \longrightarrow X_j$ if $i \leq j$. We call such an ind-scheme an \textit{ind-variety} (over a field $k$) if $X_i$'s can be taken to be quasi-projective varieties over $k$. 

The central objects of the section are stratified ind-varieties, which are already in \cite{richarz+scholbach-2020} plus some additional hypotheses. An ind-variety $X$ is \textit{stratified} if there exists a morphism of ind-varieties 
\begin{equation*}
    \iota \colon X^+  = \coprod_{w \in W}X_w \longrightarrow X
\end{equation*}
such that $\iota$ is bijective on the underlying sets, each stratum $X_w$ is a quasi-projective $k$-variety, each restriction $\iota_{\mid X_w}$ is representable by a quasi-compact immersion and each $\overline{\iota(X_w)}$ is union of other strata. If $k \longhookrightarrow \mathbb{C}$, then each stratification of $X$ gives a stratification of the complex variety $X^{\an}$ in the usual sense. A morphism of ind-varieties $(X,X^+),(Y,Y^+)$ is a commutative diagram
\begin{equation*}
    \begin{tikzcd}[sep=large]
             X^+ \arrow[r,"\iota_X"] \arrow[d,"f^+",swap] & X \arrow[d,"f"] \\ 
             Y^+ \arrow[r,"\iota_Y"] & Y
    \end{tikzcd}
\end{equation*}
where $f$ is schematic of finite type, and $f^+$ maps each stratum to a union of strata (this is weaker than requiring mapping each stratum to a stratum as in \cite{richarz+scholbach-2020}). If $X \longrightarrow Y$ is a morphism of ind-varieties and $Y$ is stratified then $X$ can be endowed with the inverse stratification.

\subsection{General properties of derived Nori motives on ind-schemes}
Let $X = \colim_{i \in I}X_i$ be an ind-variety and we view it as an object in $\operatorname{PreStk}_k^{\kappa}$. By remark \ref{rmk: extensions to prestacks}, we have right Kan extensions
\begin{equation*}
    \derivednori^b(X) = \colim_{i \in I} \ \derivednori^b(X_i).
\end{equation*}
Moreover, the six-functors formalism extend to ind-varieties in the following sense.
\begin{prop}[Richarz, Scholbach] \label{prop: six operations of ind-schemes}
Let $X,Y$ be ind-varieties and $f \colon X \longrightarrow Y$ be a morphism of ind-varieties. 
\begin{enumerate}
    \item The functor $f_*$ is well-defined. The functor $f^*$ is well-defined (and left adjoint to $f_*$) if there are presentations $X = \colim_{i \in I} X_i, Y = \colim_{i \in I} Y_i$ and $f$ is a colimit of morphisms $f_i \colon X_i \longrightarrow Y_i$. 
      \item The category $\derivednori^b(X)$ is closed monoidal, stable $\infty$-categories. If there exists a presentation $X = \colim_{i \in I}X_i$ with $X_i$ being quasi-compact, then $a^*(\mathds{1}_S)$ is a monoidal unit for $\derivednori^b(X)$ (for a counterexample in non-quasi-compact cases, see \cite[Example 2.4.3]{richarz+scholbach-2020}). 
     \item If $f$ is schematic smooth, then $f_{\#}$ exists and is left adjoint to $f^*$. 
    \item The adjunction $(f_! \dashv f^!)$ is well-defined. 
    \item Localization sequences are well-defined for schematic immersions.  
    \item Proper and Smooth base change are satisfied whenever the involved operations are well-defined. 
    \end{enumerate}
\end{prop}

\begin{proof}
This is \cite[Theorem 2.4.2]{richarz+scholbach-2020}.
\end{proof}
\begin{lem} \label{lem: motivic perverse and constructible sheaves on ind-schemes} 
Let $X = \colim_{i \in I}X_i$ be an ind-variety. The category $\derivednori^b(X)$ carries a perverse $t$-structure whose heart is the noetherian and artinian category $\mscr\perv(X) =  \colim_{i \in I} \mscr\perv(X_i)$ and a constructible $t$-structure whose heart is the category $\mscr\shv_{\ct}(X) = \colim_{i \in I}\mscr\shv_{\ct}(X_i)$. There are obvious induced functors 
\begin{equation*}
    \begin{split}
          \rat_X \colon \derivednori^b(X) & \longrightarrow \derivedcatct^b(X^{\an},\mathbb{Q}) \\ 
        \rat_X \colon \mscr\perv(X) & \longrightarrow \perv(X^{\an},\mathbb{Q}) \\ 
        \rat_X \colon \mscr\shv_{\ct}(X) & \longrightarrow \shv_{\ct}(X^{\an},\mathbb{Q})
    \end{split}
\end{equation*}
and similarly for $\derivedcatct^b(X_{\et},\mathbb{Q}_{\ell})$, $\perv_{\ell}(X,\mathbb{Q}),\shv_{\ct}(X_{\et},\mathbb{Q}_{\ell})$.
\end{lem}

\begin{proof} Both follows from the fact that an exact triangle $M_i \longrightarrow M_j \longrightarrow M_k \longrightarrow +1$ in $\derivednori^b(X)$ is defined to be a triangle $(\iota_{i \to h})_*M_i \longrightarrow (\iota_{i \to h})_*M_j \longrightarrow (\iota_{i \to h})_*M_k \longrightarrow 1$ for some $h \geq i,j,k$ and pushforwards of proper morphisms are $t$-exact for both structures. 
\end{proof}
\begin{rmk}
In the case of perverse sheaves, there is an alternative description using the universal abelian factorization as in the case of ordinary schemes, this is \cite[Proposition 6.1]{florian+morel-2019} (see also, \cite[Definition 5.2.1]{neeman}).
\end{rmk}

\begin{lem}
Let $X = \colim_{i \in I}X_i$ be an ind-variety. The category $\derivednori^b(X)$ carries a weight structure that is transversal to the perverse $t$-structure. Moreover, six functors preserve weights as in the case of schemes. 
\end{lem}

\begin{proof}
Since $\derivednori^b(X) = \colim_{i \in I}\derivednori^b(X_i)$ where transitions are closed immersions hence weight-exact, one simply sets $M \in \derivednori^b(X)$ of weight $\leq n$ (resp, $\geq n$) if it comes from an object of weight $\leq n$ (resp, $\geq n)$ in some $\derivednori^b(X_i)$. The transversality and weight preservations are induced by the case of schemes.  
\end{proof}

\begin{rmk}
From the lemma above, we see that any object in $\mscr\perv(X)$ carries a weight filtration (induced from a filtration in some $\mscr\perv(X_i)$) but unless $X$ is a scheme, the category of objects pure of weight $n$ is no longer semisimple. The upcoming sections are devoted to showing semisimplicity appears for stratified motives on affine Grassmannians. 
\end{rmk}

\begin{lem}
Let $F/k$ be a Galois extension (not necessarily finite). Let $X$ be an ind-variety over $k$, there are canonical equivalences 
\begin{equation*}
\begin{split} 
    \mscr\perv(X_F)^{\Gal(F/k)} \simeq \mscr\perv(X_k)  \\ 
    \mscr\shv_{\ct}(X_F)^{\Gal(F/k)} \simeq \mscr\shv_{\ct}(X_k).
    \end{split} 
\end{equation*}
\end{lem}

\begin{proof}
   Regarding perverse sheaves, if $X$ is a scheme, then this is \cite[Lemma 2.13]{jacobsen+terenzi-2025}, the case of ind-scheme then follows. 
\end{proof}

\subsection{Stratified derived Nori Motives}
Follow \cite{richarz+scholbach-2020}\cite{richarz+scholbach-2021}, here we study the category of stratified derived Nori motives. The theory of perverse Nori motives is sufficiently good so that we may omit technicalities on Beilinson-Soul\'e conjecture, cellularities and stratifications as in \textit{loc.cit.} Now suppose that $X$ is a stratified ind-variety, we would like to define the category of stratified perverse Nori motives $\mscr\perv(X,X^+)$ as a motivic upgrade of $\perv(X,X^+)$. 
\begin{defn}
Let $\sigma \colon k \longhookrightarrow \mathbb{C}$ be a complex embedding, the category of stratified derived Nori motives $\derivednori^b(X,X^+)$ is defined as the full subcategory of $\derivednori^b(X)$ whose constructible cohomology sheaves are motivic local systems, namely,
  \begin{equation*}
    \derivednori^b(X,X^+)^{\sigma} = \left \{M \in \derivednori^b(X) \mid \cthnor^n(\iota^*_w(M)) \in \mscr\operatorname{Loc}(X_w)^{\sigma} \right \}.
\end{equation*}
Clearly, $\derivednori^b(X,X^+)^{\sigma}$ does not depend on $\sigma$ since both $\derivednori^b(X),\mscr\localsystem(X_w)$ do not depend on $\sigma$.
\end{defn}

\begin{rmk}
Clearly, there are induced realizations
\begin{equation*}
    \begin{split}
        \betti_{X,X^+}^* \colon \derivednori^b(X,X^+) & \longrightarrow \derivedcatct^b(X^{\an},X^{+,\an},\mathbb{Q}) \\ 
        \rfrak^{\et}_{X,X^+} \colon \derivednori^b(X,X^+) & \longrightarrow \derivedcatct^b(X_{\et},X^+_{\et},\mathbb{Q}_{\ell}) 
    \end{split}
\end{equation*}
for all primes $\ell$. Under the standard conjecture, we have an equivalence $\derivednori^b(X) \simeq \daetct(X,\mathbb{Q})$ (see \cite{tubach-2025}) and the restriction of $\derivednori^b(X,X^+)  \longrightarrow \derivedcatct^b(X_{\et},X^+_{\et},\mathbb{Q}_{\ell})$ to the full subcategory generated by Tate motives is precisely the realization functor in \cite[Lemma 3.2.8]{richarz+scholbach-2020}.
\end{rmk}

\begin{rmk} \label{rmk: alternative description of stratified t-structure}
We also have an alternative description, avoiding the constructible $t$-structure, using only the perverse $t$-structure: indeed, in the description above, one may simply replace $\cthnor^n$ everywhere with $\phnor^n[-\dim(X_w)]$; since the strata are smooth, this causes no problem because perverse $t$-structures on smooth varieties are shifts by dimensions of constructible $t$-structures.
\end{rmk}

Let us show that the category $\derivednori^b(X,X^+)$ can be described in terms of generators in the style of \cite{richarz+scholbach-2020}\cite{richarz+scholbach-2021}. However we have to introduce the notion of a \textit{Whitney-Nori stratification} as a replacement of a \textit{Whitney-Tate stratification} in \cite[Definition and Lemma 3.1.11]{richarz+scholbach-2020}.
\begin{defn}
  The stratification $\iota \colon X^+ = \coprod_{w \in W} X_w \longrightarrow X$ is called \textit{Whitney-Nori} if $\iota_w^*\iota_{s,*}(M) \in \derivednori^b_{\mscr\localsystem(X_w)}(X_w)$ for all $w,s \in W,M \in \derivednori^b_{\mscr\localsystem(X_s)}(X_s)$ (by duality, this is equivalent to requiring  $\iota_w^!\iota_{s,!}(M) \in \derivednori^b_{\mscr\localsystem(X_w)}(X_w)$ for all $w,s \in W,M \in \derivednori^b_{\mscr\localsystem(X_s)}(X_s)$).
\end{defn}

\begin{rmk} \label{rmk: testing Whitney-Nori on Betti realization} 
Since the realization $\rat \colon \derivednori^b(X) \longrightarrow \derivedcat^b_{\ct}(X^{\an},\mathbb{Q})$ is exact and commutes with six operations, a stratification $\iota \colon X^+ = \coprod_{w \in W}X_w \longrightarrow X$ is Whitney-Nori if and only if the same holds after replacing $\derivednori^b(-)$ by $\derivedcat^b_{\ct}((-)^{\an},\mathbb{Q})$ everywhere.
\end{rmk}

\begin{prop}  \label{prop: equivalent definitions of stratified motives} 
   Let $\iota \colon X^+ = \coprod_{w \in W} X_w \longrightarrow X$ be a Whitney-Nori stratification. The following categories are equivalent to $\derivednori^b(X,X^+)$:
  \begin{enumerate}
    \item $\tcal_1 = \left <\iota_{w,*}(M_w) \mid w \in W, M_w \in \derivednori^b_{\mscr\localsystem(X_w)}(X_w) \right >$.
    \item $\tcal_2 = \left <\iota_{w,!}(M_w) \mid w \in W, M_w \in \derivednori^b_{\mscr\localsystem(X_w)}(X_w) \right >$. 
    \item $\tcal_3 = \left \{M \in \derivednori^b(X) \mid \iota_w^!(M) \in \derivednori^b_{\mscr\localsystem(X_w)}(X_w) \right \}$. 
    \item $\tcal_4 = \left \{M \in \derivednori^b(X) \mid \iota^*_w(M) \in \derivednori^b_{\mscr\localsystem(X_w)}(X_w) \right \}$.
\end{enumerate}
Moreover, if these categories coincide then $\iota$ is a Whitney-Nori stratification. 
\end{prop} 

\begin{proof}
Let us first consider the case when $X$ is a scheme, hence $W$ is necessarily finite. Let us argue by induction on the number of strata. Clearly, $(4)$ is just a different way to write $\derivednori^b(X,X^+)$. Let $\iota_w \colon X_w \longhookrightarrow X$ be the inclusion of the open stratum and $i \colon Z \longhookrightarrow X$ its complement. Clearly, $Z$ is stratified by $\coprod_{s \in W \setminus \left\{w \right \}} X_w$. There is a localization sequence
\begin{equation*}
    \iota_{w!}\iota_w^*(M) \longrightarrow M \longrightarrow i_!i^*(M) \longrightarrow +1.
\end{equation*}
Clearly, the first term belongs to $\tcal_2$. By definition, $i^*(M)$ is in $\left <\iota_{s!}(M_s) \mid s \neq w, M_s \in \derivednori^b_{\mscr\localsystem(X_s)}(X_s) \right >$ and hence $i_!i^*(M)$ is in $\tcal_2$. Dually, one can show that $\tcal_1 = \tcal_3$. Under the hypothesis of a Whitney-Nori stratification, we can apply Verdier duality to see that $(1)$ is equivalent to $(2)$. Let us come back to the case when $X$ is an ind-variety. We can write $X = \colim_{w \in W} \overline{X}_w$ and $\overline{X}_w$ is stratified by $\coprod_{s \leq w} X_s$. Let $M \in \derivednori^b(X,X^+) \subset \derivednori^b(X) = \colim_{w \in W}\derivednori^b(\overline{X}_w)$ and write $M = \colim_{w \in W}\iota_{\leq w,*}(M_w)$ with $\iota_{\leq w} \colon \overline{X}_w \longhookrightarrow X$ the inclusion and $M_w \in \derivednori^b(\overline{X}_w)$. In fact, we can choose a representative $M  = \iota_{\leq w,*}(M_w)$. It is then necessary that $M_w \in \derivednori^b(\overline{X}_w,\overline{X}_w^+)$. The result then follows from the case of schemes. 
\end{proof}

\begin{ex}   \label{ex: Bruhat decomposition is Whitney-Nori} 
Let $G$ be a connected, reductive group over $k$ with $T \subset B \subset P \subset G$ a choice of maximal torus, Borel, parabolic, respectively. The stratification into $B$-orbits of $G/P$ is then Whitney-Nori. 
Indeed, this is \cite[Proposition 4.10]{soergel+wendt-2018} with some new ingredients at the beginning. For $w \in W = N_G(T)/T$ (the Weyl group), let $\bcal_w = B\dot{w}B/B \simeq \mathbb{A}^{\ell(w)}_k$ be the associated Bruhat cell, where $\ell(w) \in \mathbb{Z}_{>0}$ is the length of $w$ in the Weyl group. By remark \ref{rmk: testing Whitney-Nori on Betti realization}, it suffices to check this on $\derivedcat^b_{\ct}((-)^{\an},\mathbb{Q})$, namely we have to verify that
\begin{equation*}
    \iota_w^*\iota_{s,*}(M) \in \derivedcat^b_{\localsystem(\bcal_w)}(\bcal_w,\mathbb{Q}) \ \forall \ w,s \in W, M \in \derivedcat^b_{\localsystem(\bcal_s)}(\bcal_s,\mathbb{Q}).
\end{equation*}
Since $M$ is bounded and $\localsystem(\bcal_w)[\dim(\bcal_w)]$ is a Serre subcategory of $\perv(\bcal_w)$, we can truncate and assume that $M$ itself is a shifted local system without loss of generality. However, since $\bcal_s \simeq \mathbb{A}_s^{\ell(s)}$ is simply connected, $M$ is a power of the trivial local system, i.e., $M \simeq \mathds{1}^{\oplus m}$ and so we can assume that $M = \mathds{1}$. Now one can proceed as in \cite[Proposition 4.10]{soergel+wendt-2018}. 
\end{ex}

\begin{cor} \label{lem: Verdier duality on stratified motives}
Let $\iota \colon X^+ = \coprod_{w \in W} X_w \longrightarrow X$ be a Whitney-Nori stratification, then the category $\derivednori^b(X,X^+)$ is stable under Verdier duality and Tate twists.
\end{cor}

\begin{proof}
   Being stable under Tate twists is trivial. Regard Verdier duality, by proposition \ref{prop: equivalent definitions of stratified motives}, $\derivednori^b(X,X^+) = \tcal_1 = \tcal_2$ and Verdier duality just exchanges $\tcal_1$ and $\tcal_2$. 
\end{proof}

\subsection{Stratified perverse Nori motives}
Now we want to define the category of stratified, perverse Nori motives $\mscr\perv(X,X^+)$. Ideally, we should define the \textit{category of stratified motivic perverse sheaves} as $\mscr\perv(X,X^+) = \rat_X^{-1}(\perv(X^{\an},X^{+,\an}))$ and the functor $\rat_X$ restricts to a faithful, exact functor
\begin{equation*}
    \rat_X \colon \mscr\perv(X,X^+) \longrightarrow \perv(X^{\an},X^{+,\an}).
\end{equation*}
 However, as in the usual setting $\mscr\perv(X,X^+)$ should be the heart of some $t$-structure on $\derivednori^b(X,X^+)$ but it is not clear that whether one can pullback the $t$-structure on $\derivedcatct^b(X^{\an},X^{+,\an},\mathbb{Q}) $ (or $\derivedcatct^b(X_{\et},X^+_{\et},\mathbb{Q}_{\ell})$) to $\derivednori^b(X,X^+)$ since they are not faithful (and the conservativity conjecture says that we only expect the Betti realization to be conservative). Fortunately, we can process like the ordinary constructible derived categories and give the second definition of $\derivednori^b(X,X^+)$, which are to us easier to work with than the definition above given in terms of generators. 
\begin{prop} \label{prop: cohomological interpretation of stratified motives} 
Let $\iota \colon X^+ = \coprod_{w \in W} X_w \longrightarrow X$ be a Whitney-Nori stratification, then there is a $t$-structure on $\derivednori^b(X,X^+)$ with 
  \begin{align*}
            \derivednori^b(X,X^+)^{\leq 0} & = \left \{M \in \derivednori^b(X,X^+) \mid \cthnor^n(\iota^*_w(M)) = 0 \ \forall \ w \in W, n > -\dim(X_w) \right \}\\ 
            \derivednori^b(X,X^+)^{\geq 0} & = \left \{M \in \derivednori^b(X,X^+)  \mid \cthnor^n(\iota^!_w(M)) = 0 \ \forall \ w \in W, n < -\dim(X_w) \right \}.
    \end{align*}
The heart of this $t$-structure is the category of stratified motivic perverse sheaves $\mscr\perv(X,X^+) = \mscr\perv(X) \cap \derivednori^b(X,X^+)$. The functors
\begin{align*} 
\rat_X \colon \mscr\perv(X) \longrightarrow \perv(X^{\an},\mathbb{Q}) \\ 
\rfrak_X^{\et} \colon \mscr\perv(X) \longrightarrow \perv_{\ell}(X_{\et},\mathbb{Q}_{\ell}) 
\end{align*}
restrict to faithful, exact functors of abelian $\mathbb{Q}$-categories
\begin{align*} 
\rat_{(X,X^+)} \colon \mscr\perv(X,X^+) \longrightarrow \perv(X^{\an},X^{+,\an},\mathbb{Q}) \\ 
\rfrak_{(X,X^+)}^{\et} \colon \mscr\perv(X,X^+) \longrightarrow \perv_{\ell}(X_{\et},X_{\et}^+,\mathbb{Q}_{\ell}). 
\end{align*}
\end{prop}

\begin{proof}
   The existence of the $t$-structure follows from recollectement (see \cite[1.4.10]{bbd}). Concern the last part, we know the analogous description for $\derivedcat^b_{\ct}((-)^{\an},\mathbb{Q})$ and since $\rat$ is conservative, exact and commutes with $f^*,f^!$ (see \cite[Theorem 5.1]{florian+morel-2019}), it reflects the identity $\cthnor^n(\iota^*_w(M)) = 0, \cthnor^n(\iota^!_w(M)) = 0$ and hence an induced functor. 
\end{proof}

\begin{lem} \label{lem: stratified perverse motives inside non-stratified perverse motives} 
The category $\mscr\perv(X,X^+)$ is a $\mathbb{Q}$-linear, abelian, Noetherian and Artinian category and the inclusion $\mscr\perv(X,X^+) \subset \mscr\perv(X)$ is a Serre subcategory. 
\end{lem}

\begin{proof}
  Clearly, $\mscr\perv(X,X^+)$ is $\mathbb{Q}$-linear, abelian. The properties of being Noetherian and Artinian follows from the property of being a Serre subcategory so it suffices to treat this last claim. If $X$ is smooth and $X = X^+$ then $\mscr\perv(X,X^+) = \mscr\localsystem(X)[\dim(X)] \subset \mscr\perv(X)$ because the same holds for ordinary local systems and perverse sheaves (see for instance \cite[Proposition 3.4.1]{achar-book}). Now for general $X$, a short exact sequence
  \begin{equation*}
      0 \longrightarrow M' \longrightarrow M \longrightarrow M'' \longrightarrow 0
  \end{equation*}
 in $\mscr\perv(X)$ induces long exact sequence
      \begin{equation*}
          ... \longrightarrow \cthnor^n(\iota_w^*(M')) \longrightarrow \cthnor^n(\iota_w^*(M)) \longrightarrow \cthnor^n(\iota_w^*(M'')) \longrightarrow \cthnor^{n+1}(\iota_w^*(M')) \longrightarrow ... 
      \end{equation*}
   and hence $\cthnor(\iota_w^*(M)) \in \mscr\localsystem(X_w)$ if and only if $\cthnor(\iota_w^*(M')),\cthnor(\iota_w^*(M'')) \in \mscr\localsystem(X_w)$ for every $w \in W$ (the same for $\iota_w^!$'s operations) thanks to the first observation and hence $M \in \derivednori^b(X,X^+)$ if and only if $M',M'' \in \derivednori^b(X,X^+)$. The rest then follows from the fact that $\mscr\perv(X,X^+) = \mscr\perv(X) \cap \derivednori^b(X,X^+)$. 
\end{proof}

\begin{prop} \label{prop: galois descent of stratified motives} 
Let $X$ be an ind-variety over $k$ with a Whitney-Nori stratification. Let $e \colon \Spec(F) \longrightarrow \Spec(k)$ be a field extension of subfields of $\mathbb{C}$, then there is a canonical $t$-exact functor
\begin{equation*}
    \derivednori^b(X,X^+) \longrightarrow \derivednori^b(X_F,X_F^+). 
\end{equation*}
Hence an exact functor
\begin{equation*}
    \mscr\perv(X,X^+) \longrightarrow \mscr\perv(X_F,X_F^+). 
\end{equation*}
Moreover, if $e$ is a Galois extension, there is an equivalence of categories
\begin{equation*}
    \mscr\perv(X,X^+) \simeq \mscr\perv(X_{k'},X^+_{k'})^{\Gal(k'/k)}.
\end{equation*}
\end{prop}

\begin{proof}
   Indeed, this follows from the corresponding statement for motivic local systems (see \cite{jacobsen+terenzi-2025}) and proposition \ref{thm: simple objects of perverse motives on ind-schemes}.
\end{proof}

\subsection{Functoriality}  Functorialities of perverse Nori motives behave like normal perverse sheaves. The restriction to stratified one requires further assumptions. 
\begin{lem} \label{lem: pullbacks of stratified Nori motives}
Let $(f,f^+) \colon (X,X^+) \longrightarrow (Y,Y^+)$ be a morphism of stratified ind-varieties, then the functor $f^* \colon \derivednori^b(Y) \longrightarrow \derivednori^b(X)$ restricts to a perverse $t$-exact functor
\begin{equation*}
    f^{\dagger} \colon \derivednori^b(Y,Y^+) \longrightarrow \derivednori^b(X,X^+)
\end{equation*}
and hence a fully faithful exact functor
\begin{equation*}
    f^{\dagger} \colon \mscr\perv(Y,Y^+) \longrightarrow \mscr\perv(X,X^+). 
\end{equation*}
provided that $f$ is smooth and $f^+$ is smooth on each stratum.
\end{lem}

\begin{proof}
   We take remark \ref{rmk: alternative description of stratified t-structure} into account. First, $f^{\dagger}$ is well-restricted because $\phnor^n(\iota^!_w(f^{\dagger}(M)) = \phnor^n (f^{+,\dagger}\iota^!_w(M)) = f^{+,\dagger}\phnor^*(\iota_w^*M)$, which belongs to $ f^{+,\dagger}\mscr\operatorname{Loc}(Y_w)[-\dim(X_w)] =  f^{+,*}\mscr\operatorname{Loc}(Y_w)[-\dim(X_w)+\dim(f_w)] \subset \mscr\operatorname{Loc}(Y_w)[-\dim(Y_w)]$ (note that smoothness of $f$ here is necessary for base change and dimensions, the case $\iota_w^*$ is easier). The perverse $t$-exactness is argued similarly. The full faithfulness is a consequence of smooth descent in proposition \ref{prop: pullback of smooth morphism is fully faithful}. 
\end{proof}

\begin{rmk}
This is an enhancement of \cite[Lemma 3.2.12]{richarz+scholbach-2020}\cite[Lemma 2.10]{richarz+scholbach-2021}. In fact, the result in \loccit \ establishes an equivalence on Tate motives by relying on the Beilinson-Soul\'e vanishing conjecture. Meanwhile, here (and for ordinary perverse sheaves) one should only expect a fully faithful functor. The reason behind this is for Tate motives, the functors $f^*,f^{\dagger}$ are automatically surjective since they preserve the unit object and twists. Such a surjectivity already fails for ordinary perverse sheaves by taking $X,Y$ to be schemes (for instance $Y = \Spec(\mathbb{C}),X = \mathbb{A}_{\mathbb{C}}^1$). 
\end{rmk}

\begin{lem} \label{lem: pushforwards of quasi-finite morphisms of stratified Nori motives} 
Let $i \colon Z \longhookrightarrow X$ be a quasi-finite morphism of stratified ind-varieties, then the functors $i^*,i_!$ are right perverse $t$-exact, the functors $i^!,i_*$ are left perverse $t$-exact. 
\end{lem}

\begin{proof}
   This is obvious since these exactness already hold true for schemes.
\end{proof}

We recall the notion of a stratified semismall morphism (see for instance, \cite[Section 6.1]{baumann+riche-2018}): a morphism $f \colon (X,X^+) \longrightarrow (Y,Y^+)$ is semismall if it is proper, $f(X_w)$ is a union of strata, and for each $y \in Y_{w'} \subset f(X_w)$, one has
\begin{equation*}
    \dim(X_y \times X_w) \leq \frac{1}{2}(\dim(X_w) - \dim(Y_{w'}))
\end{equation*}
and we say that $f$ is a locally trivial fibration if for such $X_w,Y_{w'}$, the morphism $X_w \times (Y_{w'} \times_Y X) \longrightarrow Y_{w'}$ is a locally trivial fibration in the Zariski topology.
\begin{prop} \label{prop: pushforwards of stratified motives}
Let $f \colon (X,X^+) \longrightarrow (Y,Y^+)$ be a stratified semismall morphism and locally trivial fibration. Assume that strata of $X^+,Y^+$ are smooth, if $M \in \derivednori^b(X,X^+)$, then $f_*(M) \in \derivednori^b(Y,Y^+)$. Moreover, if $M \in \mscr\perv(X,X^+)$, then $f_*(M) \in \mscr\perv(Y,Y^+)$. 
\end{prop}

\begin{proof}
We are beneficial from the corresponding analytic statements (see for instance \cite{baumann+riche-2018}) and the conservativity of the realization functor. 
\end{proof}
\begin{lem} \label{lem: tensor products and internal homs of stratified motives}
The category $\derivednori^b(X,X^+)$ is stable under $(-) \otimes (-)$ and $\underline{\Hom}(-,-)$ (consequently, by duality). 
\end{lem}

\begin{proof}
Again, we are beneficial from the corresponding analytic statements and the conservativity of the realization functor.  
\end{proof}

\begin{cor} \label{cor: box products of stratified motives} 
Let $X^+ = \coprod_{w \in W}X_w \longrightarrow X, Y^+ = \coprod_{r \in R}Y_r \longrightarrow Y$ be stratified ind-varieties. Let $X^+ \times Y^+ = \coprod_{(w,r) \in W \times R}(X_w \times Y_r) \longrightarrow X \times Y$ be the product stratification. If $M \in \derivednori^b(X,X^+)$ and $N \in \derivednori^b_G(Y,Y^+)$ then $M \boxtimes N \in \derivednori^b(X \times Y, X^+ \times Y^+)$. 
\end{cor}

\begin{proof}
  This is an immediate consequence of lemmas \ref{lem: equivariant, stratified pullbacks} and \ref{lemma: equivariant, stratified tensors} since the projections $(X \times Y, X^+ \times Y^+) \longrightarrow (X,X^+),(Y,Y^+)$ are stratified $G$-equivariant morphisms. 
\end{proof}

\subsection{Intersection Motives}
Let $\iota \colon X^+ = \coprod_{w \in W}X_w \longrightarrow X$ be a stratified ind-scheme. We have factorizations $\iota_w \colon X_w \overset{j_w}{\longrightarrow } \overline{X}_w \overset{i_w}{\longrightarrow} X$ where $j_w$ is an open immersion and $i_w$ is a closed immersion. For any $w \in W$, $L \in \mscr\localsystem(X_w)$, we define the \textit{Nori intersection motive} 
\begin{equation*}
    \mathbf{IC}_{w}(L) = \mathbf{IC}(X_w,L) = (i_w)_*(j_w)_{!*}(L[\dim(X_w)]). 
\end{equation*}
It is obvious that the motive $\mathbf{IC}_w(L)$ belongs to $\mscr\perv(X,X^+)$ and $\rat_X(\mathbf{IC}_w(L)) = \mathbf{IC}_w(\rat_X(L))$ in $\perv(X,X^+)$. 
Most of formal properties of intersection cohomology complexes hold true for intersection motives. We summarize some ingredients need in the upcoming sections. The proofs are analogous to the classical setting. 

%\begin{prop} \label{prop: characterizations of intersection motives} 
%Let $j \colon U \longhookrightarrow X$ be a locally closed immersion of stratified ind-varieties, the functor $j_{!*} \colon \mscr\perv(U,U^+) \longrightarrow \mscr\perv(X,X^+)$ is fully faithful. Moreover, for $M \in \mscr\perv(U,U^+)$, the motive $j_{!*}(M)$ is the unique motive (up to isomorphism) $N$ in $\mscr\perv(X)$ has the following properties:
%\begin{enumerate}
    %\item $N$ is supported on $\overline{U}$, namely, $N_{\mid X \setminus %\overline{U}} = 0$ and $N_{\mid U} = M$. 
 %   \item $N$ has no nonzero subobjects or quotient supported on $\overline{U}\setminus U$. 
%\end{enumerate}
%\end{prop}

%\begin{proof}
 % The proof is similar to the classica case. See for instance, \cite[Lemma 3.3.3]{achar-book}.
%\end{proof}

\begin{prop} \label{prop: two ses of intersection motives}
Let $X$ be an ind-variety. Let $i \colon Z \longhookrightarrow X$ be a closed immersion with open complement $j \colon U \longrightarrow X$. Let $M \in \mscr\perv(X)$.
\begin{enumerate} 
    \item If $M$ has no quotient supported on $Z$, then there is a short exact sequence
    \begin{equation*}
        0 \longrightarrow \phnor^0(i_*i^!(M)) \longrightarrow M \longrightarrow j_{!*}j^*(M) \longrightarrow 0
    \end{equation*}
    \item If $M$ has no subobject supported on $Z$, then there is a short exact sequence
    \begin{equation*}
        0 \longrightarrow j_{!*}j^*(M) \longrightarrow M \longrightarrow \phnor^0(i_*i^*(M)) \longrightarrow 0 
    \end{equation*}
\end{enumerate}
\end{prop}

\begin{proof}
   The proof is a formal manipulation of four operations. One can consult \cite[Lemma 3.3.8]{achar-book} for instance. 
\end{proof}
\begin{lem} \label{lem: Verdier dualities exchange motivic local systems}
Let $L \in \mscr\localsystem(X_w)$ be a motivic local system, then $\mathbb{D}(\mathbf{IC}_w(L)) \simeq \mathbf{IC}_w(L^{\vee})(\dim(X_w))$. 
\end{lem}
\begin{proof}
Note that it is not clear whether one can apply the realization functor since it is not clear whether there exists a canonical morphism between $\mathbb{D}(\mathbf{IC}_w(L)), \mathbf{IC}_w(L^{\vee})(\dim(X_w))$. However, as in classical case, intersection motives have universal properties and we can proceed like in, for instance, \cite[Lemma 3.3.13]{achar-book}.
\end{proof}
The following will be particularly useful for use in the proof of (a corollary of) Kazhdan-Lusztig parity vanishing. 
\begin{prop} \label{prop: characterizations of intersection motives}
Let $\iota \colon X^+ = \coprod_{w \in W} X_w \longrightarrow X$ be a Whitney-Nori stratification of an ind-variety $X$. Let $M \in \mscr\perv(X,X^+)$ (so that $\mathbb{D}(M) \in \mscr\perv(X,X^+)$ by lemma  \ref{lem: Verdier duality on stratified motives}). Let $w \in W$ and $L \in \mscr\localsystem(X_w)$, then the following statements are equivalent:
\begin{enumerate}
    \item $M \simeq \mathbf{IC}_w(L)$. 
    \item $M$ is supported on $\overline{X}_w$, $M_{\mid X_w} \simeq L[\dim(X_w)]$ and for each $X_{w'} \in \overline{X}_w \setminus X_w$, one has that
    \begin{equation*}
        \iota_{w'}^*(M) \in \derivednori^b(X_{w'},X_{w'})^{\leq -\dim(X_{w'})-1} \ \ \ \ \textnormal{and} \ \ \ \  \iota_{w'}^!(M) \in \derivednori^b(X_{w'},X_{w'})^{\geq -\dim(X_{w'})+1}
    \end{equation*}
    where by an abuse of notation $\iota_{w'} \colon X_{w'} \longhookrightarrow \overline{X}_w \setminus X_w \longhookrightarrow X$ is the obvious inclusion. 
\end{enumerate}
\end{prop}

\begin{proof}
   The direction $(1) \Rightarrow (2)$ is obvious by definition. The direction $(2) \Rightarrow (1)$ can obtained by repeating the proof of ordinary perverse sheaves (see for instance \cite[Lemma 3.3.11]{achar-book}).
\end{proof}

We can then classify simple objects in $\mscr\perv(X,X^+)$. The following is an analogue of \cite[Theorem 3.3.8]{richarz+scholbach-2020}. 

\begin{theorem} \label{thm: simple objects of perverse motives on ind-schemes}
Let $j \colon U \longrightarrow X$ be a locally closed immersion of ind-varieties. If $M$ is a simple object in $\mscr\perv(U)$, then $j_{!*}(M)$ is a simple object in $\mscr\perv(X)$. Moreover, if there exists a Whitney-Nori stratification $\iota \colon X^+ = \coprod_{w \in W} X_w \longrightarrow X$ consisting of smooth varieties $X_w$, then any simple object in $\mscr\perv(X,X^+)$ is of the form $\mathbf{IC}_w(L)$ with $L$ being a simple object in $\mscr\localsystem(X_w)$ (hence necessarily pure of some weight). 
\end{theorem}

\begin{proof}
We see that objects of the forms $j_{!*}(M)$ and $i_{w*}j_{w,!*}(M)$ are simple. In the case of a given stratification, one might assume that $X$ is a stratified scheme. It suffices to prove that any object in $\mscr\perv(X,X^+)$ admits a filtration by objects of the form $i_{w*}j_{w,!*}(M)$ with $M$ being an object in $\mscr\localsystem(X_w)$. We can proceed by noetherian induction. Assume that the result holds for strict closed subschemes of $X$. Let $j\colon U \longhookrightarrow X$ be an open stratum with closed complement $i \colon X \setminus U = Z \longhookrightarrow X$. Let $M \in \mscr\perv(X,X^+)$ be an object then $j^*(M) \in \mscr\operatorname{Loc}(U)$ by lemma \ref{lem: pullbacks of stratified Nori motives}. Thus, $j_{!*}j^*(M) = \mathbf{IC}(U,j^*M)$ is of the desired form (and simple if $M$ is simple). Let $M' = \operatorname{Coker}(\phnor^0(i_*i^!M) \longrightarrow M)$. Note that $j^*(M') \simeq j^*(M)$ so $M'$ is supported on $Z$ so it admits a filtration of the desired form. By proposition \ref{prop: two ses of intersection motives}, there is a short exact sequence 
\begin{equation*}
    0 \longrightarrow \phnor^0(i_*i^!M) \longrightarrow M \longrightarrow M' \longrightarrow 0
\end{equation*}
Similarly, there is also short exact sequence
\begin{equation*}
    0 \longrightarrow \mathbf{IC}(U,j^*M) \longrightarrow M' \longrightarrow \phnor^0 (i_*i^*M') \longrightarrow 0.
\end{equation*}
By induction, both $\phnor^0(i_*i^!M),\phnor^0 (i_*i^*M')$ admit filtrations by intersection motives of desired forms and hence $M$ must admit a filtration of the desired form as well. 
\end{proof}

\begin{cor} \label{cor: realizations of semisimple objects}
Let $X$ be an ind-scheme with a Whitney-Nori stratification $\iota \colon X^+ = \coprod_{w \in W}X_w \longrightarrow X$. If $M \in \mscr\perv(X,X^+)$ is a semisimple object, then $\rat_X(M) \in \perv(X,X^+)$ is a semisimple object.
\end{cor}

\begin{proof}
   It suffices to assume that $M$ is simple. The object $M$ is necessarily of the form $\mathbf{IC}(X_w,M_w)$ for some $w \in W$ and $M_w \in \mscr \operatorname{Loc}(X_w)$ being simple. By virtue of \cite[Lemma 6.12]{terenzi-2024} (built upon the work \cite{tubach-2025}\cite{jacobsen-2025}), the object $\rat_{X_w}(M_w)$ is semisimple in $\mscr\perv(X)$ and hence also semisimple in $\mscr\perv(X,X^+)$ by lemma \ref{lem: stratified perverse motives inside non-stratified perverse motives} and as a consequence 
   \begin{equation*}
   \rat_X(\mathbf{IC}(X_w,M_w)) = \mathbf{IC}(X_w, \rat_{X_w}(M_w))
   \end{equation*}
is semisimple as well. 
\end{proof}

\subsection{(Shifted) Tate motives inside perverse Nori motives}

Let us discuss stratified Tate motives on ind-varieties. For technical reasons, we cannot define Tate motives on ind-varieties but fortunately, we can still define \textit{shifted Tate motives} and this is in analogy with the content of \cite[Section 3]{richarz+scholbach-2020}. Let $\iota \colon X^+ = \coprod_{w \in W}X_w \longrightarrow X$ be a Whitney-Nori stratification. as in the preceding section. 
\begin{defn}  \label{defn: stratified derived Tate motives} 
\begin{enumerate} 
 \item The \textit{category of stratified derived Tate motives} $\mathbf{DT}^b(X,X^+)$ is defined as the full subcategory $\derivednori^b(X,X^+)$ such that 
 \begin{align*}
    \mathbf{DT}^b(X,X^+)  & = \left \{M \in \mathbf{DT}^b(X) \mid \cthnor^n(\iota^*_w(M)) \in \mscr\operatorname{Tate}(X_w) \right \} \\ 
    & = \left \{M \in \mathbf{DT}^b(X) \mid \phnor^n(\iota^*_w(M)) \in \mscr\operatorname{Tate}_p(X_w) \right \}
\end{align*}
\item The \textit{category of shifted, stratified Tate motives} $\mscr\textnormal{Tate}_p(X,X^+)$ is defined as the full subcategory of $\mscr\perv(X,X^+)$ containing all motives $M$ such that $\iota_w^*(M) \in \mscr\textnormal{Tate}_p(X_w)$ for all $w \in W$. The fully faithful embedding 
\begin{equation*}
    \mscr\perv(X,X^+) \longhookrightarrow \derivednori^b(X,X^+)
\end{equation*}
restricts to a fully faithful embedding 
\begin{equation*}
    \mscr\operatorname{Tate}(X,X^+) \longhookrightarrow \mathbf{DT}^b(X,X^+).
\end{equation*}
  \end{enumerate} 
\end{defn}

\begin{lem} \label{lem: characterizations of stratified Tate motives}
The subcategory category $\mscr\textnormal{Tate}_p(X,X^+) \subset \mscr\perv(X,X^+)$ contains motives of the form $\iota_{w,*}(L[\dim(X_w)])$ with $L \in \mscr\textnormal{Tate}(X_w)$ and stable under extensions and direct summands. Moreover, $\mscr\textnormal{Tate}_p(X,X^+) \subset \mscr\perv(X,X^+)$ is the smallest full subcategory of $\mscr\perv(X,X^+)$ having these properties.
\end{lem}

\begin{proof}
The first claim is clear since each $\mscr\operatorname{Tate}_p(X_w) \subset \mscr\localsystem(X_w)$ has the same properties. The second claim then follows. 
\end{proof}

The following lemma is obvious, given lemma \ref{lem: Tate cohomological functors on schemes}.  
\begin{lem} \label{lem: Tate cohomological functors on ind-schemes} 
 The composition 
  \begin{equation*}
      \mathbf{DT}^b(X,X^+) \longrightarrow \derivednori^b(X,X^+) \overset{\phnor^n}{\longrightarrow} \mscr\perv(X,X^+)
  \end{equation*}
  takes values in $\mscr\textnormal{Tate}_p(X,X^+)$ for any $n \in \mathbb{Z}$.
\end{lem}

\begin{proof}
We have to check that for each $w \in W$, $\iota_w^*(\cthnor^n(M)) = \cthnor^n(\iota_w^*(M)) \in \mscr\textnormal{Tate}(X_w)$  and this is clear from the definition. 
\end{proof}

\begin{lem} \label{lem: intersection motives of Tate motives are Tate}
If $L \in \mscr\textnormal{Tate}(X_w) \subset \mscr\localsystem(X_w)$, then $\mathbf{IC}_w(L) \in \mscr\textnormal{Tate}_p(X,X^+)$. 
\end{lem}

\begin{proof}
  Since $\mathbf{IC}_w(L) \in \mscr\perv(X,X^+)$, it remains to show that $\iota_w^*(L) \in \mscr\textnormal{Tate}_p(X_w)$. By proposition \ref{prop: characterizations of intersection motives}, this motive is nothing but $L[\dim(X_w)]$ and hence this is obvious. 
\end{proof}

\subsection{Weight structure (revisited)} 
Let us declare that a motive $M \in \derivednori^b(X,X^+)$ (and hence $\mscr\perv(X,X^+)$) has weight $\leq n$ (resp, $\geq n$ or pure of weight $n$) if the underlying motive $M \in \derivednori^b(X)$ has the corresponding property. 
\begin{cor} \label{prop: weight-exactness of operations on ind-schemes}
Let $f \colon (X,X^+) \longrightarrow (Y,Y^+)$ be a morphism of Whitney-Nori stratified ind-schemes, if the operations $(f^*,f_*,f_!,f^!,\otimes, \underline{\Hom},\boxtimes)$ are well-defined, the they enjoy the same properties as in {\color{blue}section 2.1}. 
\end{cor}

\begin{proof}
  This is obvious, given the case of schemes. 
\end{proof}

\begin{cor}
Let $j \colon (U,U^+) \longrightarrow (X,X^+)$ be a locally closed immersion of Whitney-Nori stratified ind-varieties and $L \in \mscr\perv(U,U^+,n)$ be an object pure of weight $n$, then $j_{!*}(L)$ is pure of weight $n$. 
\end{cor}

\begin{proof}
This is similar to the classical case.  
\end{proof}

\section{Equivariant perverse Nori motives}
In this section, our ultimate goal is to define the (stratified) equivariant perverse Nori motives $\mscr\perv_G(X)$. We propose four definitions of this category (where the fourth definition, to our best knowledge, is not available in the ordinary setting). The functorialities between equivariant perverse Nori motives shall be used implicitly throughout next sections. At the end, we study the such category where algebraic groups are replaced by pro-algebraic groups. 
\subsection{Equivariant Six Operations}
As in remark \ref{rmk: extensions to prestacks}, we can extend $\mathbf{DN}^b$ (and $\mathbf{DN},\daet$ depends on the taste) to prestacks. In this section, we are mainly interested in quotient stacks $X/G$. Let $G$ be algebraic $k$-group acting on a $k$-variety $X$. We also say that $X$ is a $G$-\textit{variety} and hence have at hand the quotient stack $[X/G]$. We define 
\begin{equation*}
    \hbf_G(X) \coloneqq \hbf([X/G]) \coloneqq \lim \left( \begin{tikzcd}[sep=scriptsize]
              \hbf(X) \arrow[r,shift left= 2] \arrow[r,shift right= 2] & \hbf(X \times_{\xfrak} X) \arrow[l,dashed] \arrow[r] \arrow[r,shift right= 4] \arrow[r,shift left= 4]  & ... \arrow[l,shift left = 2,dashed ] \arrow[l, shift right = 2,dashed] 
    \end{tikzcd}\right).
\end{equation*}
if $\hbf \in \left\{\derivednori^b(-),\derivednori(-),\daet(-),\daetct(-,\mathbb{Q}), \derivedcat((-)^{\an},\mathbb{Q}),\derivedcat^b_{\ct}((-)^{\an},\mathbb{Q}), \derivedcat((-)_{\et},\mathbb{Q}_{\ell}),\derivedcat^b_{\ct}((-)_{\et},\mathbb{Q}_{\ell}) \right \}$. The realization functors such as Betti, $\ell$-adic, Nori realization functors extend to 
\begin{equation*}
\begin{split} 
    \betti^*_{[X/G]} \colon \daet_G(X,\mathbb{Q}) & \longrightarrow \derivedcat_G(X^{\an},\mathbb{Q}) \\ 
    \rfrak^{\et}_{[X/G],\ell} \colon \daet_G(X,\mathbb{Q}) & \longrightarrow \derivedcat_G(X_{\et},\mathbb{Q}_{\ell}) \\ 
    \operatorname{Nri}^*_{[X/G]} \colon \daet_G(X,\mathbb{Q}) & \longrightarrow \derivednori_G(X) \\ 
    \end{split}
\end{equation*}
The enhancement developed in \cite{liu+zheng-2017} (see also \cite{khan-2019}\cite{hoyois-2017}\cite{tubach_2025}) allows us to extend all six operations from schemes to algebraic stacks; in particular, to quotient stacks (note that for six functors, one needs unbounded theories such as $\mathbf{DN}(-),\daet(-,\mathbb{Q})$. However, we do not need the full power of six operations on stacks in this article but instead we will work distinctive operations on quotient stacks that make the theory particularly rich and interesting. In what below, operations colored in {\color{red}red} are possibly not well-defined for "compact" theories (such as $\derivednori^b,\daetct$). 
\begin{enumerate}
     \item  Let $G$ be an algebraic group. Let $X$ be a $G$-variety and $H \leq G$ be a subgroup. We denote by $p \colon X/H \longrightarrow X/G$ the projection. The operations 
    \begin{align*}
        \operatorname{Res}^G_H \coloneqq p^* \colon \hbf_G(X) &  \longrightarrow \hbf_H(X) \\ 
       {\color{red} \operatorname{Av}^G_{H*} \coloneqq p_* \colon \hbf_H(X)} & {\color{red} \longrightarrow \hbf_G(X)} \\ 
      {\color{red} \operatorname{Av}^G_{H!} \coloneqq p_{\#} = p_!(\dim(G/H))[2\dim(G/H)] \colon \hbf_H(X) }&  {\color{red} \longrightarrow \hbf_G(X) }
    \end{align*}
    are called \textit{restriction functors}, \textit{right and left averaging functors}, respectively. If $K \subset H \subset G$, then 
    \begin{equation*}
        \res^H_K \circ \res^G_H = \res^G_K.
    \end{equation*}
    \item Let $G$ be an algebraic group and $H \trianglelefteq G$ be a normal subgroup and let $X$ be a $G/H$-variety. We denote by $q \colon X/G \longrightarrow X/(G/H)$ the projection. 
      \begin{align*}
        \operatorname{Infl}^G_{G/H} \coloneqq q^* \colon \hbf_{G/H}(X) &  \longrightarrow \hbf_G(X) \\ 
      {\color{red} \operatorname{Inv}^G_{H*} \coloneqq q_* \colon \hbf_H(X)} & {\color{red}\longrightarrow \hbf_G(X)} \\ 
       {\color{red} \operatorname{Inv}^G_{H!} \coloneqq  q_!(-\dim(G/H))[-2\dim(G/H)] \colon \hbf_H(X)} & {\color{red}\longrightarrow \hbf_G(X)}
    \end{align*}
    are called \textit{inflation functors}, \textit{right and left invariant functors}, respectively. If $K \trianglelefteq H \trianglelefteq G$ be normal subgroups, then 
    \begin{equation*}
        \operatorname{Infl}^G_{G/H} \simeq \operatorname{Infl}^G_{G/K} \circ \operatorname{Infl}^{G/K}_{G/H}. 
    \end{equation*}
\end{enumerate}
The reason for the non-existence of {\color{red}operations colored in red} can be seen as a part of a more general question involving constructiblities of four operations of Artin stacks and they exist only when one allows suitable bounded below and above versions $\hbf^+$, $\hbf^-$ (see \cite[Section 6.8]{achar-book} for some useful comments). For general stacks, the question of constructibility can be subtle and some achievements are obtained in \cite[Theorem 3.8]{tubach_2025}) for mixed Hodge modules. Here we can avoid this problem since we do not need $\operatorname{Inv}^G_{H*}, \operatorname{Inv}^G_{H!}$. Nevertheless, by following the traditional approaches, we show that $\operatorname{Av}_{H*}^G, \operatorname{Av}_{H!}^G$ are indeed well-defined. The restriction functors and inflation functors satisfy additional properties:
 \begin{lem} 
 The following statements hold true:
\begin{enumerate}
    \item [(1)] $\operatorname{Res}^G_H$ commutes with six operations of $G$-equivariant morphisms.  
    \item [(2)] $\infl^G_H$ commutes with six operations of $(G/H)$-equivariant morphisms. 
    \item [(3)]  Restrictions and inflations commute in the following sense: Let $H,K \leq G$ be subgroups with $K \trianglelefteq G$ normal. Let $X$ be a $(G/K)$-variety, then there is a natural isomorphism
    \begin{equation} \label{eq:1}
        \res^G_H \circ \infl^G_{G/K} \simeq \infl^H_{H/(H \cap K)} \circ \res^{G/K}_{H/(H \cap K)}.
    \end{equation}
\end{enumerate}
\end{lem} 
\begin{proof}
These all follow from the functorialities of pullbacks.
\end{proof} 
There are also non-trivial properties listed in the following
\begin{theorem} \label{theorem: properties of special equivariant six operations}
\begin{enumerate}
    \item Let $G$ be an algebraic group and $X$ be a $G$-variety. Suppose that $X$ admits a geometric quotient $p \colon X \longrightarrow X/G$, then 
\begin{equation*}
  p^* \colon \hbf_G(X) \longrightarrow \hbf(X/G)
\end{equation*}
is an equivalence of stable $\infty$-categories. 
\item Let $G$ be an algebraic group and $X$ be a $G$-variety. Let $H \trianglelefteq G$ be a normal subgroup so that $X$ is a principal $H$-variety with quotient $p \colon X \longrightarrow X/H$. The functor
\begin{equation*}
    p^* \circ \operatorname{Infl}^G_{G/H} \colon \hbf_{G/H}(X/H) \longrightarrow \hbf_G(X) 
\end{equation*}
ís an equivalence of stable $\infty$-categories. 
\item Let $G$ be an algebraic group and $H \leq G$ be a subgroup. Let $X$ be a $H$-variety and $i \colon X \longrightarrow G \times^H X$ be unit section, i.e., on points $i(x) = (1,x)$. If $\hbf(-)$ is $\mathbb{Q}$-linear, there is an isomorphism of functors
\begin{equation*}
    i^*[\dim(G/H)] \circ \res^G_H \simeq i^![\dim(G/H)](\dim(G/H)) \circ \res^G_H \colon \hbf(G \times^H X) \longrightarrow \hbf_H(X).
\end{equation*}
Moreover, these functors are equivalences of stable $\infty$-categories.
\item Let $G$ be an algebraic group and $H \leq G$ be a subgroup. Let $X$ be a $H$-variety. Let $\overline{a} \colon G \times^H X \longrightarrow X$ be the morphism induced by the action morphism. There are isomorphisms of functors
\begin{equation*}
\begin{split} 
    \av^G_{H*} \simeq \overline{a}_* \circ (i^* \circ \res^G_H)^{-1} \colon \hbf_H(X) & \longrightarrow \hbf_G(X) \\ 
    \av^G_{H!} \simeq \overline{a}_! \circ (i^* \circ \res^G_H)^{-1} \colon \hbf_H(X) & \longrightarrow \hbf_G(X) 
    \end{split}
\end{equation*}
Consequently, if $\hbf(X)$ is compactly generated and pushforwards preserve constructible motives, then $\operatorname{Av}_{H*}^G,\operatorname{Av}_{H!}^G$ restrict to constructible motives. 
\end{enumerate}
\end{theorem}

\begin{proof}
Part $(1)$ is a special case of part $(2)$, for which we argue as follows: if $X/H$ exists as a scheme, the projection $X/H \longrightarrow (X/G)/(G/H)$ is a smooth cover of $(X/H)/(G/H) = X/G$ and hence they yields the same category thanks to the universal property of limits. Regarding part $(3)$, by relative purity (i.e. $f^! \simeq f^*(d)[2d]$), there are commutative diagrams
\begin{equation*}
    \begin{tikzcd}[sep=large]
              \hbf_H(X) \arrow[rr,"\pr_2^{\dagger}\infl^{H \times H}_{1 \times H}"] \arrow[rd,"\pr_2^{\dagger}\infl^{G \times H}_{1 \times H}",swap] & & \hbf_{H \times H}(X) \\ 
              & \hbf_{G \times H}(G \times X) \arrow[ru,swap,"\widetilde{i}^*\res^{G \times H}_{H \times H}\textnormal{[}-\dim(G/H)\textnormal{]}"] & 
    \end{tikzcd} \ \ \ \ \ \ \ \ \  \begin{tikzcd}[sep=large]
              \hbf_H(X) \arrow[rr,"\pr_2^{\dagger}\infl^{H \times H}_{1 \times H}"] \arrow[rd,"\pr_2^{\dagger}\infl^{G \times H}_{1 \times H}",swap] & & \hbf_{H \times H}(X) \\ 
              & \hbf_{G \times H}(G \times X) \arrow[ru,swap,"\widetilde{i}^*\textnormal{[}\dim(G/H)\textnormal{]}(\dim(G/H))"] & 
    \end{tikzcd}
\end{equation*}
with $\widetilde{i} \colon H \times X \longhookrightarrow G \times X$ the canonical immersion, the rest is formally same as in \cite[Theorem 6.5.10]{achar-book}. Concern the first claim of part $(4)$, it suffices to show that, for instance, $\overline{a}_* \circ (i^* \circ \res^G_H)^{-1}$ is a right adjoint of $\res^G_H$. Equivalently, $(i^* \circ \res^G_H) \circ \overline{a}^* \circ \res^G_H$ is the identity functor. This is clear since $\overline{a} \circ i = \id_X$. The second claim of part $(4)$ is obvious.  
\end{proof} 

\begin{cor} \label{cor: left averaging = right averaging}
Let $G$ be an algebraic group and $X$ be a $G$-variety. Let $H \leq G$ be a closed subgroup of finite index. Then $\av^G_{H!} \simeq \av^G_{H*}$.
\end{cor}

\begin{proof}
This is clear since $G \times^H X = (G/H) \times X$ is then proper over $X$. 
\end{proof}

\begin{prop} \label{prop: equivariant motives of unipotent groups} 
\begin{enumerate}
    \item Let $G$ be an algebraic group and $H \subset G$ be a subgroup such that $G/H$ is unipotent. Let $X$ be a $G$-variety, the functor
\begin{equation*}
    \res^G_H \colon \hbf_G(X) \longrightarrow \hbf_H(X)
\end{equation*}
is fully faithful.
    \item Let $G$ be an algebraic group and let $U \trianglelefteq G$ be a connected, unipotent, normal subgroup. Let $X$ be a $(G/U)$-variety, the functor
    \begin{equation*}
        \infl_{G/U}^G \colon \hbf_{G/U}(X) \longrightarrow \hbf_G(X)
    \end{equation*}
    is an equivalence of stable $\infty$-categories. 
\end{enumerate}
\end{prop}

\begin{proof}
\begin{enumerate}
    \item 
    Let $G$ act on $(G/H) \times X$ by the rule $g \cdot ([g'],x) = ([gg'],gx)$. The projection $\pr_2 \colon (G/H) \times X \longrightarrow X$ is then $G$-equivariant. Let $\overline{a} \colon G \times^H X \longrightarrow X$ be the morphism induced by the action morphism. Let $i \colon X \longrightarrow G \times^H X$ be the unit section. There is a commutative diagram 
    \begin{equation*}
        \begin{tikzcd}[sep=large]
                  G \times^H X \arrow[rd,"\overline{a}",swap] \arrow[rr] & & (G/H) \times X \arrow[dl,"\pr_2"] \\ 
                  & X & 
        \end{tikzcd}
    \end{equation*}
    where the horizontal arrow $(g,x) \longmapsto ([g],gx)$ is an isomorphism. Since $G/H$ is a unipotent, it is isomorphic to an affine space. By the $\mathbb{A}^1$-homotopy axiom, we obtain that $\pr_2^* \colon \hbf(X) \longrightarrow \hbf((G/H) \times X)$ is fully faithful and hence $\overline{a}^*$ must be fully faithful as well. There is a commutative diagram 
    \begin{equation*}
        \begin{tikzcd}[sep=large]
                  \hbf_G(X) \arrow[r,"\overline{a}^*"] \arrow[d,"\res^G_H",swap]  & \hbf_G(G \times^H X) \arrow[d,"\res^G_H"] \\ 
                  \hbf_H(X) & \hbf_H(G \times^H X). \arrow[l,"i^*"] 
        \end{tikzcd}
    \end{equation*}
    Since $a \circ i = \id_X$, we must have that $\res^G_H \colon \hbf_G(X) \longrightarrow \hbf_H(X)$ is fully faithful because $i^* \circ \res^G_H$ is an equivalence thanks to {\color{blue}induction equivalence}. 
    \item By \ref{eq:1}, there is an isomorphism of functors $\res^G_U \circ \infl^G_{G/U} \simeq \infl^{U}_1 \circ \res^{G/U}_1 \simeq \id$. This implies that $\res^G_U$ is essentially surjective and thanks to the preceding that $\res^G_U$ is fully faithful and hence an equivalence. Consequently, $\infl^G_{G/U}$ must be an equivalence as well.
\end{enumerate}
\end{proof}

\begin{rmk}
\begin{enumerate} 
\item A purely topological method can be used in the case of derived categories $\derivedcat^b(-)$ since in this case their equivariant versions can be defined in terms of $n$-acyclic resolutions. This approach is taken in \cite{bernstein+lunts-1994} (see also \cite{achar-book}). 
\item In some sense, the proposition above says that if we are only interested in affine algebraic groups, then it suffices to study equivariant categories $\hbf_G(X)$ only for $G$ being \textit{reductive} since we can always consider $G/R_u(G)$ with $R_u(G)$ being the unipotent radical of $G$. 
\end{enumerate}
\end{rmk}

Let us provide an immediate application of the theory, namely, the Fourier-Laumon transform on derived Nori motives. 

\begin{ex} Let $S$ be a $k$-variety. Let $V$ be a vector bundle of rank $r$ over $S$ whose dual bundle is denoted by $V^{\vee}$. There are natural $\gm_m$-actions on $V$ and $V^{\vee}$ given by scaling.The Fourier-Laumon transform is an equivalence of categories
\begin{equation*}
    \operatorname{Four} \colon \derivednori_{\gm_m}(V) \longrightarrow \derivednori_{\gm_m}(V^{\vee}). 
\end{equation*}
To define the Fourier-Laumon transform, let us consider the diagram
\begin{equation*}
    \begin{tikzcd}[sep=large]
      & V \times_S V^{\vee} \arrow[ld,"\pr_1",swap] \arrow[rd,"\pr_2"] \arrow[rr,"m"] & & \mathbb{A}^1_S \\
      V & & V^{\vee} & 
    \end{tikzcd}
\end{equation*}
where $\pr_1,\pr_2$ are canonical projections, $m$ is the standard pairing. Let $\mu \colon \gm_{m,S} \times_S \gm_{m,S} \longrightarrow \gm_{m,S}$ be the multiplication morphism. Let $V \times_S V^{\vee}$ be endowed with the $\gm_{m,S}^2$-action given by the action of $\gm_{m,S}$ on each factor. The morphism $m \colon V \times_S V^{\vee} \longrightarrow \mathbb{A}_S^1$ is then $\mu$-equivariant. We define three subgroups of $\gm_{m,S}^2$ by formulas
\begin{equation*}
    \begin{split}
        K_1 & = \left \{(1,\lambda) \mid \lambda \in \gm_{m,S} \right \} \subset \gm_{m,S}^2 \\ 
        K_2 & = \left \{(\lambda,1) \mid \lambda \in \gm_{m,S} \right \} \subset \gm_{m,S}^2 \\ 
        K_{\mu} & = \left \{(\lambda,\lambda^{-1}) \mid \lambda \in \gm_{m,S} \right \} \subset \gm_{m,S}^2 
    \end{split}
\end{equation*}
Let $j \colon \gm_{m,S} \longhookrightarrow \mathbb{A}^1_S$ be the canonical open immersion. The \textit{Fourier-Laumon transform} is defined as
    \begin{equation*}
        \begin{split}
            \operatorname{Four}_V \colon \derivednori_{\gm_m}(V) & \longrightarrow \derivednori_{\gm_m}(V^{\vee}) \\ 
            M & \longmapsto \operatorname{Inv}_{K_2,!}\pr_{2,!}\left(\pr_1^* \operatorname{Infl}^{\gm^2_m}_{\gm^2_m/K_1}(M) \otimes m^* \operatorname{Infl}^{\gm^2_m}_{\gm^2_m/K_{\mu}}(\jcal) \right)[r]
        \end{split}
    \end{equation*}
    where $\jcal = j_*(\mathds{1}_{\gm_m})[1](1) \in \derivednori_{\gm_m}(\mathbb{A}^1_S)$ is the \textit{Fourier kernel}. The formation of $\operatorname{Four}$ is universal and hence also applicable to other settings such as derived categories $\mathbf{D}(X^{\an},\mathbb{Q})$. A crucial point is that for $\mathbf{D}(X^{\an},\mathbb{Q})$, $\operatorname{Four}_V$ is well-defined on constructible part (see for instance \cite[Lemma 6.9.8]{achar-book}), namely, on the category $\lim \ \derivedcat^b_{\ct}(\gm_m^n \times X^{\an},\mathbb{Q})$. Thiago Landim (in his unpublished thesis) defined the bounded, constructible part $\mathbf{DN}_{\gm_m}^b(-)$ of $\mathbf{DN}_{\gm_m}(-)$ using a cohomological intepretation and showed that $\operatorname{Four}_V$ is well-defined on $\mathbf{DN}_{\gm_m}^b(-)$. We can cheat a little by declaring a motive in $\mathbf{DN}_{\gm_m}(-)$ to be bounded constructible if its realization is bounded constructible, then the Fourier-Laumon transform is still well-defined on bounded constructible motives. By \cite{laumon-2003}, the Fourier-Laumon transform $\operatorname{Four}_V \colon \derivednori^b_{\gm_m}(V) \longrightarrow \derivednori^b_{\gm_m}(V^{\vee})$ is an equivalence of categories. More precisely, $\operatorname{Four}_{V'} \circ \operatorname{Four}_V(-) \simeq (-)(-r)$. The conservativity of the realization shows that the same holds for Nori motives. 
\end{ex}

\subsection{Equivariant Perverse Nori Motives}
In this section, we freely use the materials on (derived) equivariant Nori motives given in the appendix. Let $G$ be an algebraic group acting on a $k$-variety $X$, the easiest definition of an equivariant Nori motives $M \in \mscr\perv_G(X)$ is a motive in $\mscr\perv(X)$ such that there exists an isomorphism $a^{\dagger}(M) \simeq \pr_2^{\dagger}(M)$, where 
\begin{equation*}
    a,\pr_2 \colon G \times_k X \longrightarrow X
\end{equation*}
denote the action, the projection, respectively. It turns out that there are at least three more candidates for the definition of this category (one more in comparison to the classical theory) and the content of this section is to show that they are indeed equivalent. Let $G$ be a connected algebraic group \footnote{The connectedness is important here, as remarked in \cite[Appendix A]{baumann+riche-2018}.} and $X$ be a $G$-variety. Let
\begin{equation*}
    a \colon G \times_k X \longrightarrow X \ \ \ \ \textnormal{and} \ \ \ \ p \colon G \times_k X \longrightarrow X
\end{equation*}
be the action morphism and the projection, respectively. 
\begin{prop} \label{prop: t-structures on equivariant categories} 
Let $\derivednori_G^b(X) = \lim \derivednori^b(G^n \times_k X)$ be the corresponding category of equivariant derived Nori motives developed abstractly in the previous section. There are two $t$-structures on $\derivednori_G^b(X)$, called the motivic constructible $t$-structure and motivic perverse $t$-structure. Moreover, 
\begin{enumerate}
    \item The realization $\derivednori_G^b(X) \longrightarrow \derivedcat_G^b(X^{\an},\mathbb{Q})$ is $t$-exact for both $t$-structures. 
    \item The restriction $\derivednori_G^b(X) \longrightarrow \derivednori^b(X)$ is $t$-exact for both $t$-structures. 
\end{enumerate}
and consequently, these functors restrict to t-exact functors on constructible motives. 
\end{prop}
\begin{proof}
The existence of t-structures on the limit category is is an application of \cite[Lemma 3.2.18]{richarz+scholbach-2020}. 
\end{proof}

Follow \cite[Appendix]{baumann+riche-2018}, we propose first three categories of equivariant perverse Nori motives:
\begin{enumerate}
    \item As in \cite{bernstein+lunts-1994}, we define the category $\mscr\perv_G^{\#}(X)$ as the heart of perverse $t$-structure on$\derivednori_G^b(X)$. The forgetful functor $\res^G \colon \derivednori^b_G(X) \longrightarrow \derivednori^b(X)$ is $t$-exact and $(\res^G)^{-1}(\mscr\perv(X)) = \mscr\perv_G^{\#}(X)$. 
    \item We define the category $\mscr\perv_G^{\dagger}(X)$ whose objects are pair $(M,\theta)$ where $M \in \mscr\perv(X)$ and $\theta \colon a^{\dagger}(M) \overset{\sim}{\longrightarrow} p^{\dagger}(M)$ is an isomorphism such that
    \begin{equation*}
        e^*(\theta) = \id_M \ \ \ \ \textnormal{and} \ \ \ \ (m \times \id_X)^*(\theta) = (\pr_{23})^*(\theta) \circ (\id_G \times a)^*(\theta)
    \end{equation*}
    and whose morphisms $(M,\theta) \longrightarrow (M',\theta')$ are morphisms $f \colon M \longrightarrow M'$ in $\mscr\perv(X)$ such that the diagram
    \begin{equation*}
        \begin{tikzcd}[sep=large]
                 a^{\dagger}(M) \arrow[r,"\theta"] \arrow[d,"a^*(f)",swap] & \pr_2^{\dagger}(M) \arrow[d,"p^*(f)"] \\ 
                 a^{\dagger}(M') \arrow[r,"\theta'"] & \pr_2^{\dagger}(f)
        \end{tikzcd}
    \end{equation*}
    commutes. 
     \item We define the category $\mscr\perv_G^{\wr}(X)$ as the full subcategory of $\mscr\perv(X)$ consisting of motives $M$ such that there is an isomorphism $a^{\dagger}(M) \simeq p^{\dagger}(M)$ in $\mscr\perv(X)$; equivalently, an isomorphism $a^*(M) \simeq p^*(M)$ in $\derivednori^b(X)$. 
    \end{enumerate} 
By spelling out the definition of $\mscr\perv^{\#}_G(X)$, we see that
\begin{equation*}
    \mscr\perv_G^{\#}(X) = \lim \left( \begin{tikzcd}[sep=scriptsize]
              \mscr\perv(X) \arrow[r,shift left= 2] \arrow[r,shift right= 2] & \mscr\perv(G \times_k X) \arrow[l,dashed] \arrow[r] \arrow[r,shift right= 4] \arrow[r,shift left= 4]  & \mscr\perv(G \times_k G \times_k X) \arrow[l,shift left = 2,dashed ] \arrow[l, shift right = 2,dashed] \arrow[r,shift left = 1] \arrow[r,shift right = 1] \arrow[r,shift right= 3] \arrow[r,shift left= 3]   & \cdots  \arrow[l,shift left = 2,dashed ] \arrow[l, shift right = 2,dashed] \arrow[l,dashed]
    \end{tikzcd}\right).
\end{equation*}
Since $(\Delta^{\op})^+$ is cofinal in $(\Delta^{\op})$ and $e^{\dagger}[-\dim(G)]$ is exact, we can compute the limit above after adding unit sections. Now because $\mscr\perv^{\#}_G(X)$ is an ordinary category, the higher data $n \geq 2$ does not affect the limit. Hence, $\mscr\perv^{\#}_G(X)$ is equivalent to the category $\mscr\perv^{\dagger}_G(X)$. There are functors
\begin{equation*}
    \mscr\perv_G^{\#}(X) \simeq \mscr\perv_G^{\dagger}(X) \longrightarrow \mscr\perv^{\wr}_G(X)
\end{equation*}
where the right arrow is the forgetful functor and under $\rat_X$, they become equivalences of categories. Let us discuss operations of equivariant perverse sheaves. 
\begin{prop}
Let $? \in \left \{\#,\dagger,\wr,\varnothing \right\}$. 
\begin{enumerate}
    \item Let $G$ be an algebraic group and $X,Y$ be $G$-varieties. Let $f \colon X \longrightarrow Y$ be a smooth, surjective morphism of $G$-varieties. There is a well-defined functor 
     \begin{equation*}
      f^{\dagger} \colon \mscr\perv_G^?(Y) \longrightarrow \mscr\perv^?_G(X)
  \end{equation*}
  compatible with $f^{\dagger}$ of ordinary perverse sheaves. 
  \item Let $G$ be an algebraic group and $X$ be a $G$-variety. Let $H \leq G$ be a subgroup of $G$. There is well-defined functor
  \begin{equation*}
      \operatorname{Res}^G_H \colon \mscr\perv_G^?(X) \longrightarrow \mscr\perv_H^?(X)
  \end{equation*}
     compatible with $\res^G_H$ of ordinary perverse sheaves. 
    \item Let $G$ be an algebraic group. Let $H \trianglelefteq G$ be a normal subgroup. There is well-defined functor
  \begin{equation*}
      \infl^G_{G/H} \colon \mscr\perv^?_{G/H}(X) \longrightarrow \mscr\perv^?_G(X)
  \end{equation*}
  compatible with $\infl^G_{G/H}$ of ordinary perverse sheaves. 
\end{enumerate}
Moreover, the same statements hold true if one replaces analytic perverse sheaves $\perv_G(-)$ with $\ell$-adic perverse sheaves $\perv_{G,\ell}(-)$. 
\end{prop}

\begin{proof}
For $? \in \left \{\#,\dagger,\wr \right \}$, the existences of these operations follow from the explicit description and the commutativities follow from the similar descriptions of analytic perverse sheaves (or $\ell$-adic ones). For $? = \varnothing$, everything follows from the definition of universal factorizations. 
\end{proof}

\begin{prop} \label{prop: properties of equivariant perverse sheaves}
Let $G$ be a smooth algebraic group.
\begin{enumerate} 
\item Let $X,Y$ be $G$-varieties. Let $f \colon X \longrightarrow Y$ be a smooth, surjective $G$-equivariant morphism. The functor
\begin{equation*}
    f^{\dagger} \colon \mscr\perv_G^{\dagger}(Y) \longrightarrow \mscr\perv_G^{\dagger}(X)
\end{equation*}
is faithful. If $f$ has geometrically connected fibers, then $f^{\dagger}$ is fully faithful.
\item Let $H \trianglelefteq G$ be a connected, normal subgroup. Given $G/H$-variety $X$, the functor
\begin{equation*}
    \operatorname{Infl}^G_{G/H} \colon \mscr\perv_{G/H}^{\dagger}(X) \longrightarrow \mscr\perv_G^{\dagger}(X) 
\end{equation*}
is fully faithful. If moreover, $G$ is connected, then $\operatorname{Infl}^G_{G/H}$ is an equivalence of categories. 
\item Let $X$ be a principal $G$-variety and let $p \colon X \longrightarrow X/G$ be the quotient morphism. The functor
\begin{equation*}
    p^{\dagger} \colon \mscr\perv^{\dagger}(X/G) \longrightarrow \mscr\perv_G^{\dagger}(X)
\end{equation*}
is an equivalence of categories. 
\item Let $H \leq G$ be a closed subgroup. Let $X$ be a $H$-variety and let $i \colon X \longrightarrow G \times^H X$ be the morphism $i(x) = (1,x)$. The functor
    \begin{equation*}
        i^*[-\dim(G/H)] \colon \mscr\perv_G^{\dagger}(G \times^H X) \longrightarrow \mscr\perv_H^{\dagger}(X)
    \end{equation*}
    is an equivalence of categories.
\item Assume that $G$ is connected, then the inclusion $\mscr\perv_G^{\dagger}(X) \longhookrightarrow \mscr\perv_G^{\wr}(X)$ is an equivalence of categories. 
\end{enumerate}
\end{prop}
\begin{proof}
Given propositions \ref{prop: pullback of smooth morphism is fully faithful}, \ref{prop: smooth descent}, the proofs of (1), (2), (3), (5) are similar to the ones in classical cases. One can consult \cite[Propositions 6.2.6, 6.2.8, 6.2.10, Proposition 6.2.17]{achar-book}. Part (4) will follow from part 3 in theorem \ref{theorem: properties of special equivariant six operations}.
\end{proof}

\begin{convention}
From now on, we write $\mscr\perv_G(X)$ to indicate the category $\mscr\perv_G^{\#}(X)$. This should cause no confusion thanks to the previous theorem. 
\end{convention}

\begin{ex} 
This is an application of corollary \ref{cor: left averaging = right averaging}. Let $G$ be a finite group acting on a variety $X$, then there is an equivalence of categories
\begin{equation*}
    \derivedcat^b(\perv_{G^{\an}}(X^{\an})) \simeq \derivedcat^b_{G^{\an}}(X^{\an}). 
\end{equation*}
(see for instance \cite[Theorem 6.6.12]{achar-book}). Let us show that we can obtain a similar equivalence for Nori motives
\begin{equation*}
     \derivedcat^b(\mscr\perv_G^{\#}(X)) \simeq \derivednori^b_G(X). 
\end{equation*}
Both are applications of Beilinson's realization functor. In \cite{beilinson-1987-2}, Beilinson constructs a realization functor $\derivedcat^b(\ascr) \longrightarrow \tscr$ from the derived category of the heart $\ascr$ of a bounded $t$-structure on a triangulated category $\tscr$ and proves that this realization is an equivalence of and only if any morphism of the form $M \longrightarrow N[n]$ with $M,N \in \ascr$ is effaceable, namely, there exists an epimorphism $p \colon M' \longrightarrow M$ and a monomorphism $i \colon N \longrightarrow N'$ such that $i[n] \circ f \circ p \colon M' \longrightarrow N'[n]$ is zero. 

Now by corollary \ref{cor: left averaging = right averaging}, left average functor equals right average functors and by theorem \ref{theorem: properties of special equivariant six operations} they are $t$-exact (since they are combinations of $t$-exact functors). Given $M \in \mscr\perv_G(X)$, the unit morphism $\av_{1!}^G\res^G_1(M) \longrightarrow M$ is surjective (note that $\av_{!}^G\res^G(M) \in \mscr\perv_G(X)$ by the previous observation) in $\mscr\perv_G(X)$; indeed, this follows from the fact that $\rat^G_X \colon \mscr\perv_G(X) \longrightarrow \perv_G(X)$ is exact and the same holds for ordinary equivariant perverse sheaves (see for instance \cite[Theorem 6.6.12]{achar-book}). Let $M,N \in \mscr\perv_G(X)$ and consider a morphism $f \colon M \longrightarrow N[n]$ in $\derivednori_G^b(X)$, then $\res^G_1(f)$ is effaceable in $\derivednori^b(X) = \derivedcat^b(\mscr\perv(X))$ (see for instance \cite[Lemma A.5.17]{achar-book}), the rest then is same as \cite[Theorem 6.6.12]{achar-book}.
\end{ex} 
\subsection{Equivariant motivic local systems} In this subsection, we develop the notion of equivariant motivic local systems. Equivariant local systems are to equivariant motivic perverse sheaves what local systems are to motivic perverse sheaves. This also leads to the notion of equivariant motivic fundamental group, which already appears in the topological setting.

\begin{defn}
  Let $G$ be a connected algebraic groups over $k$. Let $X$ be a smooth $k$-variety on which $G$ acts on and let $a,p \colon G \times_k X \longrightarrow X$ be the action and projection, respectively . We define a full subcategory of $\mscr\perv_G(X)$, called the \textit{category of shifted equivariant motivic local systems}
\begin{equation*}
\begin{split} 
    \mscr\operatorname{Loc}_{G,p}(X) & = \left \{L \in \mscr\perv_G(X) \mid \res^G_1(L) \in \mscr\operatorname{Loc}_p(X) \right \} \\ 
    & = \left \{L \in \mscr\localsystem_p(X) \mid a^{\dagger}(L) \simeq p^{\dagger}(L) \right \}
    \end{split} 
\end{equation*}
The \textit{category of equivariant motivic local systems} is defined as $\mscr\operatorname{Loc}_G(X) \coloneqq \mscr\operatorname{Loc}_{G,p}(X)[\dim(X)]$. 
\end{defn}
We have some important remarks:
\begin{rmk}
\begin{enumerate} 
    \item The inclusion $\mscr\localsystem_G(X) \longhookrightarrow \mscr\localsystem(X)$ is fully faithful.
    \item In this definition, we do not require $X$ to be a principal $G$-variety (as in the topological setting, see \cite[Proposition 6.2.13]{achar-book}), this means that in general there are equivariant perverse Nori motives that are \textit{not} shifted equivariant motivic local systems. However, if $X$ is a principal $G$-variety, then since $\localsystem_{G,p}(X) = \perv_G(X)$, we indeed have that $\mscr\localsystem_{G,p}(X) = \mscr\perv_G(X)$. 
    \item There is a commutative diagram 
    \begin{equation*}
        \begin{tikzcd}[sep=large]
          \mscr\localsystem_G(X) \arrow[r] \arrow[d] & \mscr\localsystem(X) \arrow[d]  \\ 
          \localsystem_G(X) \arrow[r] & \localsystem(X)
        \end{tikzcd}
    \end{equation*}
    but this diagram is \textit{not} cartesian. This results in some new phenomena. For instance, if $X^{\an}$ is simply connected then $\localsystem_G(X) = \localsystem(X)$ but still $\mscr\localsystem_G(X) \neq \mscr\localsystem(X)$. The reason behind this is that equivariance is a property of motives, not of their realizations. An isomorphism $\rat(a^*(L)) \simeq \rat(p^*(L))$ does not lift to an isomorphism $a^*(L) \simeq p^*(L)$ for $L \in \mscr\localsystem(X)$. 
\end{enumerate}
\end{rmk}

\begin{lem}
The category $\mscr\localsystem_{G,p}(X)$ can be described as 
\begin{equation*}
    \mscr\operatorname{Loc}_{G,p}(X) = \left \{L \in \mscr\perv_G(X) \mid L \ \textnormal{is strongly dualisable} \right \}. 
\end{equation*}
\end{lem}

\begin{proof}
Indeed, we have the same description for the non-equivariant case and since $\res^G_1$ is conservative, exact and commutes with $\otimes, \underline{\Hom}$  so it reflects strong duals and we win.   
\end{proof}

%An object $L \in \mscr\localsystem_G(M)$ is simple if and only if $\res^G(L) \in \mscr\localsystem(X)$ is simple and hence necessarily pure of some weight. 

\begin{theorem}
Let $G$ a connected algebraic group and $X$ be a principal $G$-variety of dimension $d$. The category $\mscr\operatorname{Loc}_{G,p}(X)$ endowed with
\begin{equation*}
\begin{split} 
    (-) \otimes (-) \colon \mscr\operatorname{Loc}_{G,p}(X) \times  \mscr\operatorname{Loc}_{G,p}(X) & \longrightarrow  \mscr\operatorname{Loc}_{G,p}(X) \\ 
   (M, N) & \longmapsto (M[-d] \otimes N[-d])[d]
   \end{split} 
\end{equation*}
is a neutral Tannakian category over $\mathbb{Q}$. 
\end{theorem}

\begin{proof}
The proof is as same as the proof of \cite[Theorem 6.3]{terenzi-2024}.     
\end{proof}

\begin{defn}
    Let $x \in X(k)$ be a $k$-point. This defines a fiber functor 
    \begin{equation*}
        x^* \colon \mscr\localsystem_{G}(X) \longrightarrow \localsystem(X) \longrightarrow \operatorname{Vect}_{\mathbb{Q}}
    \end{equation*}
    and the Tannakian dual of $\mscr\operatorname{Loc}_{G}(X)$ is denoted by $\gscr_G^{\mot}(X,x)$, called the \textit{motivic equivariant fundamental group} of $(X,x)$. In particular, if $X = \Spec(k)$, then the action is trivial and $\gscr^{\mot}_G(X,x) = \gscr^{\mot}(k)$, the motivic Galois group. 
\end{defn}
\begin{rmk} 
Let $G$ be a connected algebraic group and $X$ be a principal $G$-variety, then every $G$-equivariant perverse sheaf on $X$ is a shifted (by $\dim(X)$) local system. Moreover, there is an equivalence of categories (see \cite[Proposition 6.2.13]{achar-book})
\begin{equation*}
    \localsystem_G(X) \simeq k[G^x/(G^x)^{\circ}]-\operatorname{Mod}^{\textnormal{fg}}
\end{equation*}
where $G^x \subset G$ is a stabilizer of some point $x \in G$ and $(G^x)^{\circ}$ is its identity component. This reminds us of the equivalence $\localsystem^{\textnormal{ft}}(X) \simeq \mathbb{Q}[\pi_1(X)]-\operatorname{Mod}^{\textnormal{fg}}$ in the non-equivariant setting. For this reason, we call $\pi_1^G(X,x) = G^x/(G^x)^{\circ}$ the \textit{equivariant fundamental group} of $X$. The group $\gscr_G^{\mot}(X,x)$ can be regarded as a motivic avatar of $\pi_1^G(X,x)$. 
\end{rmk} 

\begin{rmk} Now let $X$ be a $k$-variety with a $k$-point $x \in X(k)$. Thanks to Grothendieck, there is a split exact sequence
\begin{equation*}
    1 \longrightarrow \pi_1^{\et}(\overline{X},\overline{x}) \longrightarrow \pi_1^{\et}(X,x) \longrightarrow \operatorname{Gal}(\overline{k}/k) \longrightarrow 1. 
\end{equation*}
In the works of Terenzi \cite{terenzi-2025} and Jacobsen \cite{jacobsen-2025}, there is a motivic avatar of the sequence above 
\begin{equation*}
    1 \longrightarrow \pi_1^{\textnormal{mot}}(X,x) \longrightarrow \gscr(X,x) \longrightarrow \gscr^{\textnormal{mot}}(k) \longrightarrow 1, 
\end{equation*}
where $\gscr^{\mot}(k) = \gscr_1(k,k)$, $\gscr(X,x) = \gscr_1(X,x)$ as above and $\pi_1^{\textnormal{mot}}(X,x)$ is the motivic Galois group arising from local systems of geometric origin. A fundamental theorem in \cite{jacobsen-2025} shows that $\pi_1^{\textnormal{mot}}(X,x)$ is the Malcev completion of the topological fundamental group $\pi_1(X,x)$. One might pose a similar question: how can one relate $\Ker(\gscr_G(X,x) \longrightarrow \gscr_G^{\mot}(k))$ with $\pi_1^G(X,x) = G^x/(G^x)^{\circ}$?
\end{rmk}
\subsection{Equivariant perverse Nori motives on (stratified) ind-schemes} In this subsection, we investigate equivariant perverse Nori motives on stratified ind-varieties. We also define equivariant intersection motives and show that those of simple equivariant local systems constitute to all simple objects of equivariant perverse Nori motives. We first study the derived version and recover the heart via the perverse $t$-structure. Let $G = \lim_{n \in \mathbb{N}} G_n$ be a pro-algebraic group acting on an ind-scheme $X = \colim_{i \in I} X_i$. By remark \ref{rmk: extensions to prestacks}, the categories $\derivednori^b(G^n \times_k X)$ are well-defined. 

\begin{defn} \label{defn: equivariant derived Nori motives for pro-algebraic groups}
  Let $G = \lim_{n \in \mathbb{N}} G_n$ be a pro-algebraic group acting on an ind-scheme $X$, we can still define the \textit{category of equivariant derived Nori motives} 
  \begin{equation*}
    \derivednori_G^b(X) \coloneqq   \derivednori^b([X/G]) \coloneqq \lim \left( \begin{tikzcd}[sep=scriptsize]
                \derivednori^b(X) \arrow[r,shift left= 2] \arrow[r,shift right= 2] &   \derivednori^b(G \times_k X) \arrow[l,dashed] \arrow[r] \arrow[r,shift right= 4] \arrow[r,shift left= 4]  & ... \arrow[l,shift left = 2,dashed ] \arrow[l, shift right = 2,dashed] 
    \end{tikzcd}\right).
\end{equation*}
When $X$ is an stratified ind-variety $\iota \colon X^+ = \coprod_{w \in W} X_w \longrightarrow X $, we define the \textit{category of stratified equivariant derived Nori motives} $\derivednori^b_G(X,X^+)$ as the homotopy pullback
  \begin{equation*}
      \derivednori^b_G(X,X^+) \coloneqq \derivednori^b(X,X^+) \times_{\derivednori^b(X)} \derivednori^b_G(X).
  \end{equation*}
 Clearly, $\derivednori^b_G(X,X^+)$ is the full subcategory of $\derivednori^b_G(X)$ whose images under the forgetful functor $\derivednori^b_G(X) \longrightarrow \derivednori^b(X)$ are in $\derivednori^b(X,X^+)$ and there is a restriction functor
 \begin{equation*}
     \res^G \colon \derivednori^b_G(X,X^+) \longrightarrow \derivednori^b(X,X^+). 
 \end{equation*}
\end{defn}

Let us assume the following situation to ensure that there is a well-behave heart on $\derivednori^b_G(X,X^+)$. 

\begin{hypothesis} \label{hyp: action of pro-group}
\begin{enumerate}
    \item We assume that the index set $I = \mathbb{N}$ each $X_n$ is induced by a corresponding $G_n$-action (in fact, one can do more general setting but this is what usually happens in practice). 
    \item We assume that each $G_m \longrightarrow G_n$ is surjective for $n \leq m$ and $\Ker(G_m \longrightarrow G_n)$ is a smooth, connected, unipotent group. 
    \item We assume that each $X_w$ is stable under $G$-action. By taking suitable closure $\overline{X}_w$, each $X_i$ can be stratified $X_i^+ = X^+ \times_X X_i = \coprod_{w \in W_i} X_w \longrightarrow X_i$ for some finite set $W_i \subset X$.
\end{enumerate}
\end{hypothesis}

%By \cite[Lemma A.3.2]{richarz+scholbach-2020}, for any pro-algebraic group acting on a stratified ind-scheme $\iota \colon X^+ = \coprod_{w \in W} X_w \longrightarrow X $, there exists a presentation $X = \colim_{i \in I} X_i$

\begin{prop} \label{prop: equivariant perverse Nori motives on ind-schemes}
 Let $G = \lim_{n \in \mathbb{N}} G_n$ be a pro-algebraic group. Let $X = \colim_{n \in \mathbb{N}} X_n$ be an ind-scheme on which $G$ acts. Assume hypothesis \ref{hyp: action of pro-group} holds true then
 \begin{equation*}
    \derivednori^b_{G}(X,X^+) =  \underset{i \in I}{\colim} \  \derivednori^b_{G}(X_n,X_n^+)  =  \underset{i \in I}{\colim} \  \lim_{m \geq n}  \derivednori^b_{G_m}(X_n,X^+_n)
\end{equation*}
and carries a $t$-structure whose heart is 
\begin{equation*}
   \mscr\perv_G(X,X^+) =  \underset{i \in \mathbb{N}}{\colim} \  \mscr\perv_G(X_n,X_n^+)  =  \underset{i \in \mathbb{N}}{\colim} \  \lim_{m \geq n}  \mscr\perv_{G_m}(X_n,X_n^+),
\end{equation*}
called the category of $G$-equivariant stratified Nori motives. 
\end{prop} 

\begin{proof}
Since each $\Ker(G_{n+1} \longrightarrow G_n)$ is smooth, connected, unipotent, its underlying scheme is isomorphic to some $\mathbb{A}_k^{m}$ and hence $\derivednori^b_{G_j}(X_i)$ does not depend on $j \geq i$ thanks to proposition \ref{prop: equivariant motives of unipotent groups}, thus these categories are equivalent to $\derivednori^b_G(X_i)$. The pushforwards along the closed immersions 
\begin{equation*} 
\derivednori^b_{G}(X_j)  \longrightarrow \derivednori^b_{G}(X_i)
\end{equation*} 
are then well-defined and thus the category 
\begin{equation*}
    \derivednori^b_{G}(X) = \ \underset{i \in I}{\colim} \  \derivednori^b_{G}(X_i)  =  \underset{i \in I}{\colim} \  \lim_{j \geq i} \derivednori^b_{G_j}(X_i)
\end{equation*}
is well-defined as well. Taking fiber product on both side gives us the desired equality. Regard the $t$-structure, the pushforwards $\derivednori^b_{G}(X_j)  \longrightarrow \derivednori^b_{G}(X_i)$ are $t$-exact and fully faithful so the hypotheses of \cite[Lemma 3.2.18]{richarz+scholbach-2020} is fulfilled and the heart is computed in the expected way. 
\end{proof}

\begin{prop} \label{prop: forgetful functors of ind-schemes are fully faithful}
Assume that hypothesis \ref{hyp: action of pro-group} holds true and each $G_n$ is connected, then the restriction of 
\begin{equation*}
    \res^G \colon \derivednori^b_G(X,X^+) \longrightarrow \derivednori^b(X,X^+)
\end{equation*}
to hearts, namely, 
\begin{equation*}
    \res^G \colon \mscr\perv_G(X,X^+) \longrightarrow \mscr\perv(X,X^+)
\end{equation*}
is fully faithful. More precisely, let us denote by $a, p \colon G \times X \longrightarrow X$ the action and the projection, respectively, then
\begin{equation*}
 \mscr\perv_G(X,X^+) = \left \{M \in \mscr\perv(X,X^+) \mid a^*(M) \simeq p^*(M) \right \}. 
\end{equation*}
\end{prop}

\begin{proof}
Indeed, the morphism $\res^G \colon \mscr\perv_G(X,X^+) \longrightarrow \mscr\perv(X,X^+)$ can be rewritten as 
\begin{equation*} 
\underset{i \in \mathbb{N}}{\colim} \  \lim_{m \geq n}  \mscr\perv_{G_m}(X_n,X_n^+) \longrightarrow  \underset{i \in \mathbb{N}}{\colim} \ \mscr\perv(X_n,X_n^+) 
\end{equation*} 
so we win because each $\mscr\perv_{G_m}(X_n,X_n^+) \longhookrightarrow \mscr\perv(X_n,X_n^+)$ is fully faithful (as $G_m$'s are connected) and a colimit of fully faithful functors remains fully faithful. The second statement then follows as well because the same is true for each $\mscr\perv_{G_m}(X_n,X_n^+) = \left \{M \in \mscr\perv(X_n,X_n^+) \mid a^*(M) \simeq p^*(M) \right \} \longhookrightarrow \mscr\perv(X_n,X_n^+)$. 
\end{proof}

\begin{ex}
If $G$ is a connected unipotent algebraic group acting transitively on a smooth variety $X$, then $\mscr\perv_G(X,X) = \mscr\localsystem_G(X)$ thanks to proposition \ref{prop: forgetful functors of ind-schemes are fully faithful}. However, this holds true for any connected algebraic group (but the action is needed to be transitive so that we can consider $\mscr\perv_G(X,X)$. 
\end{ex}

\begin{prop}
With the hypotheses of proposition \ref{prop: equivariant perverse Nori motives on ind-schemes}, there exist realization functors
\begin{equation*}
\begin{split} 
    \rat_{(X,X^+)} \colon \derivednori^b_G(X,X^+) \longrightarrow \derivedcat^b_{G^{\an}}(X^{\an},X^{+,\an},\mathbb{Q}) \\
     \rfrak^{\et}_{(X,X^+)} \colon \derivednori^b_G(X,X^+) \longrightarrow \derivedcat^b_{G,\ell}(X_{\et},X^{+}_{\et},\mathbb{Q}_{\ell}) 
    \end{split} 
\end{equation*}
where the right hand sides are defined in \cite{bernstein+lunts-1994}\cite{baumann+riche-2018} or alternatively we can re-proceed like the case $\derivednori^b(-)$. These functors restrict to functors
\begin{equation*}
\begin{split} 
    \mscr\perv_G(X,X^+) \longrightarrow \perv_{G^{\an}}(X^{\an},X^{+,\an},\mathbb{Q}) \\
     \mscr\perv_G(X,X^+) \longrightarrow \perv_{G,\ell}(X_{\et},X^{+}_{\et},\mathbb{Q}_{\ell}) 
     \end{split} 
\end{equation*}
which are compatible with forgetful functors on both sides. 
\end{prop}

\begin{proof}
   This is simply functoriality since both sides are defined in the same way.
\end{proof}

\begin{prop} \label{prop: galois descent of equivariant, stratified motives} 
Let $X$ be an ind-variety over $k$ with a Whitney-Nori stratification. Let $e \colon \Spec(F) \longrightarrow \Spec(k)$ be a field extension of subfields of $\mathbb{C}$, then there is a canonical $t$-exact functor
\begin{equation*}
    \derivednori^b_G(X,X^+) \longrightarrow \derivednori^b_{G_F}(X_F,X_F^+). 
\end{equation*}
Hence an exact functor
\begin{equation*}
    \mscr\perv_G(X,X^+) \longrightarrow \mscr\perv_{G_F}(X_F,X_F^+). 
\end{equation*}
Moreover, if $e$ is a Galois extension, then there is an equivalence of categories
\begin{equation*}
    \mscr\perv_{G_k}(X_k,X_k^+) \simeq \mscr\perv_{G_F}(X_F,X^+_F)^{\Gal(F/k)}.
\end{equation*}
\end{prop}

\begin{proof}
The morphism $ \derivednori^b_G(X,X^+) \longrightarrow \derivednori^b_{G_F}(X_F,X_F^+)$ exists thanks to the universal property and the existence of $\derivednori^b(X,X^+) \longrightarrow \derivednori^b(X_F,X_F^+)$ in proposition \ref{prop: galois descent of stratified motives}. To show Galois descent, it suffices to check that $e_* \colon \mscr\perv(X,X^+) \longrightarrow \mscr\perv(X_F,X_F^+)$ restricts to $e_* \colon \mscr\perv_G(X,X^+) \longrightarrow \mscr\perv_{G_F}(X_F,X_F^+)$ but this is clear since $e$ is \'etale so it commutes with $a^{\dagger},p^{\dagger}$. 
\end{proof}

\subsection{Equivariant Functoriality} This is an equivariant version of {\color{blue}section 3.5}. Let $G$ be a pro-algebraic group acting on two stratified ind-varieties $X^+ = \coprod_{w \in W}X_w \longrightarrow X$ and $Y^+ = \coprod_{r \in R}Y_r \longrightarrow Y_r$. A stratified $G$-equivariant morphism $f \colon X \longrightarrow Y$ is a $G$-equivariant morphism so that the restriction to each stratum is a $G$-equivariant morphism of schemes. 

\begin{lem} \label{lem: equivariant, stratified pullbacks} 
Let $G$ be a pro-algebraic group. Let $f \colon X \longrightarrow Y$ be a stratified $G$-equivariant morphism of stratified ind-varieites, then the functor $f_G^* \colon \derivednori^b(Y) \longrightarrow \derivednori^b_G(X)$ restricts to a functor
\begin{equation*}
    f^*_G \colon \derivednori^b_G(Y,Y^+) \longrightarrow \derivednori^b_G(X,X^+)
\end{equation*}
and 
\end{lem}

\begin{proof}
This simply follows from the definition because $f^* \colon \derivednori^b(Y,Y^+) \longrightarrow \derivednori^b(X,X^+)$ (lemma \ref{lem: pullbacks of stratified Nori motives}) and $f^*_G \colon \derivednori^b_G(Y) \longrightarrow \derivednori^b(X)$ are well-defined. 
\end{proof}
\begin{lem}
Let $i \colon Z \longhookrightarrow X$ be a quasi-finite, stratified $G$-equivariant morphism of stratified ind-varieties, then the functors $i^*,i_!$ are right perverse $t$-exact, the functors $i^!,i_*$ are left perverse $t$-exact. 
\end{lem}

\begin{proof}
Similar to the above, this follows from lemma \ref{lem: pushforwards of quasi-finite morphisms of stratified Nori motives}.   
\end{proof}

\begin{prop} \label{prop: pushforwards of equivariant stratified motives}
Let $f \colon (X,X^+) \longrightarrow (Y,Y^+)$ be a stratified semismall, equivariant morphism and locally trivial fibration. Assume that strata of $X^+,Y^+$ are smooth, if $M \in \derivednori^b_G(X,X^+)$, then $f_*(M) \in \derivednori^b_G(Y,Y^+)$. Moreover, if $M \in \mscr\perv_G(X,X^+)$, then $f_*(M) \in \mscr\perv_G(Y,Y^+)$.
\end{prop}

\begin{proof}
This is a consequence of proposition \ref{prop: pushforwards of stratified motives}.  
\end{proof}

\begin{lem} \label{lemma: equivariant, stratified tensors}
The category $\derivednori^b_G(X,X^+)$ is stable under $(-) \otimes (-)$ and $\underline{\Hom}(-,-)$ (consequently, by duality). 
\end{lem}

\begin{proof}
Let $M,N \in \derivednori^b_G(X,X^+) \subset \derivednori^b_G(X)$, then $\res^G(M \otimes N) = \res^G(M) \otimes \res^G(N)$, but $\res^G(M),\res^G(N)\in \derivedcat^b(X,X^+)$ and $\derivednori^b(X,X^+)$ is closed under tensor product by lemma \ref{lem: tensor products and internal homs of stratified motives}. 
\end{proof}

\begin{cor} \label{cor: box products of equivariant, stratified motives} 
Let $G$ be a pro-algebraic group acting on stratified ind-varieties $X^+ = \coprod_{w \in W}X_w \longrightarrow X, Y^+ = \coprod_{r \in R}Y_r \longrightarrow Y$. Let $X^+ \times Y^+ = \coprod_{(w,r) \in W \times R}(X_w \times Y_r) \longrightarrow X \times Y$ be the product stratification endowed with the diagonal action. If $M \in \derivednori^b_G(X,X^+)$ and $N \in \derivednori^b_G(Y,Y^+)$ then $M \boxtimes N \in \derivednori^b_G(X \times Y, X^+ \times Y^+)$. 
\end{cor}

\begin{proof}
  This is an immediate consequence of lemmas \ref{lem: equivariant, stratified pullbacks} and \ref{lemma: equivariant, stratified tensors} since the projections $(X \times Y, X^+ \times Y^+) \longrightarrow (X,X^+),(Y,Y^+)$ are stratified $G$-equivariant morphisms. 
\end{proof}

\begin{cor} \label{cor: box products of products of groups}
Let $G,H$ be pro-algebraic groups acting on stratified ind-varieties $X^+ = \coprod_{w \in W}X_w \longrightarrow X, Y^+ = \coprod_{r \in R}Y_r \longrightarrow Y$, respectively. Let $X^+ \times Y^+ = \coprod_{(w,r) \in W \times R}(X_w \times Y_r) \longrightarrow X \times Y$ be the product stratification endowed with the product action $G \times H$. If $M \in \derivednori^b_G(X,X^+)$ and $N \in \derivednori^b_H(Y,Y^+)$ then $M \boxtimes N \in \derivednori^b_{G \times H}(X \times Y, X^+ \times Y^+)$. 
\end{cor}

\begin{proof}
We note that the action of $G \times H$ on $X \times Y$ still satisfies hypothesis \ref{hyp: action of pro-group}  so $ \derivednori^b_{G \times H}(X \times Y, X^+ \times Y^+)$ is well-defined. We let $G \times H$ act on $X$ by $(g,h)\cdot x = g \cdot x$ and on $Y$ by $(g,h) \cdot y = h \cdot y$. These actions again satisfy hypothesis \ref{hyp: action of pro-group}. The result is then an obvious consequence of corollary \ref{cor: box products of equivariant, stratified motives}. 
\end{proof}

\subsection{Equivariant Intersection Motives} Let us show that in the equivariant setting, one can also characterize simple objects, known as equivariant intersection motives. The case of ordinary perverse sheaves is treated in \cite[Section 5.2]{bernstein+lunts-1994}. Let $\iota \colon X^+ = \coprod_{w \in W}X_w \longrightarrow X$ be a stratified ind-scheme so that hypothesis \ref{hyp: action of pro-group} is satisfied. We have factorizations $\iota_w \colon X_w \overset{j_w}{\longrightarrow } \overline{X}_w \overset{i_w}{\longrightarrow} X$ where $j_w$ is an open immersion and $i_w$ is a closed immersion. The operations $j_w,i_w$ are equivariant so they induce equivariant functors
\begin{equation*}
    \begin{split}
        i_{w*} \colon \mscr\perv_G(\overline{X}_w,\overline{X}_w^+) \longrightarrow \mscr\perv_G(X,X^+) \\ 
        j_{w!}, j_{w*} \colon \mscr\perv_G(X_w,X_w) = \mscr\operatorname{Loc}_{G,p}(X_w) \longrightarrow \mscr\perv_G(\overline{X}_w,\overline{X}_w^+) \\ 
        j_{w,!*} \colon \operatorname{Im}(\phnor_G^0 \circ j_{w!} \longrightarrow \phnor_G^0 \circ j_{w*}) \colon \mscr\operatorname{Loc}_{G,p}(X_w) \longrightarrow \mscr\perv_G(\overline{X}_w,\overline{X}^+_w).
    \end{split}
\end{equation*}
In particular, we can define the \textit{equivariant Nori intersection motive} 
\begin{equation*}
    \mathbf{IC}_{w}^G(L) = \mathbf{IC}^G(X_w,L) = (i_w)_*(j_w)_{!*}(L[\dim(X_w)]) \in \mscr\perv_G(X,X^+)
\end{equation*}
for any $L \in \mscr\operatorname{Loc}_G(X_w)$.

%\begin{proof}
 % By proposition \ref{prop: characterizations of equivariant intersection motives}, one has a characterization for $\res^G(L)$. For $(1) \Rightarrow (2)$, we note that the functor $\res^G$ commutes with six operations and respects cohomological degrees. For $(2) \Rightarrow (1)$, proposition \ref{prop: characterizations of equivariant intersection motives} shows that $\res^G(M) = \mathbf{IC}_w(L')$ for some motivic local system $L' \in \mscr\localsystem(X_w)$. Since $M$ is equivariant, we know that $a^{\dagger}\mathbf{IC}_w(L') \simeq p^{\dagger}\mathbf{IC}_w(L')$ and hence $a^{\dagger}(L') \simeq p^{\dagger}(L')$ because $i^G_{w,*},(j_w)_{!*}^G$ are {\color{red}fully faithful}.
% \end{proof}

\begin{prop}  \label{prop: characterizations of equivariant intersection motives}
Let $G = \lim_{n \in \mathbb{N}} G_n$ be a pro-algebraic group acting on a stratified ind-variety $\iota \colon X^+ = \coprod_{w \in W} X_w \longrightarrow X$. Assume that hypothesis \ref{hyp: action of pro-group} is satisfied. Let $M \in \mscr\perv_G(X,X^+)$ (so that $\mathbb{D}(M) \in \mscr\perv(X,X^+)$ by lemma  \ref{lem: Verdier duality on stratified motives}). Let $w \in W$, then the following statements are equivalent:
\begin{enumerate}
    \item $M \simeq \mathbf{IC}_w^G(L)$ for some equivariant motivic local system $L \in \mscr\localsystem_G(X_w)$. 
    \item $\res^G(M) \simeq \mathbf{IC}_w(\res^G(L))$ for some equivariant motivic local system $L \in \mscr\localsystem(X_w)$.  
\end{enumerate}
\end{prop}

\begin{proof}
This is simply because $\res^G$ is fully faithful (here we use proposition \ref{prop: forgetful functors of ind-schemes are fully faithful}) and $\mathbf{IC}_w(\res^G(L)) \simeq \res^G(\mathbf{IC}_w^G(L))$; indeed, this follows from the fact that $\res^G$ is exact (on the derived level) and commutes with six operations. 
\end{proof}

As a consequence, we have the following that is analogous to proposition \ref{prop: characterizations of equivariant intersection motives}. 
\begin{cor} 
Assume the hypotheses of the previous proposition, then the following statements are equivalent:
\begin{enumerate}
    \item $M \simeq \mathbf{IC}_w^G(L)$. 
    \item $\res^G(M)$ is supported on $\overline{X}_w$, $\res^G(M)_{\mid X_w} \simeq \res^G(L)[\dim(X_w)]$ and for each $X_{w'} \in \overline{X}_w \setminus X_w$, one has that
    \begin{equation*}
        \iota_{w'}^*(\res^G(M)) \in \derivednori^b(X_{w'},X_{w'})^{\leq -\dim(X_{w'})-1} \ \ \ \ \textnormal{and} \ \ \ \  \iota_{w'}^!(\res^G(M)) \in \derivednori^b(X_{w'},X_{w'})^{\geq -\dim(X_{w'})+1}
    \end{equation*}
\end{enumerate}
\end{cor}

\begin{prop} \label{prop: simple objects of equivariant perverse motives on ind-schemes}
Let $j \colon U \longrightarrow X$ be a $G$-equivariant locally closed immersion of ind-varieties. If $M$ is a simple object in $\mscr\perv(U)$, then $j_{!*}(M)$ is a simple object in $\mscr\perv_G(X,X^+)$. Under hypothesis \ref{hyp: action of pro-group}, any simple object in $\mscr\perv_G(X,X^+)$ is of the form $\mathbf{IC}_w^G(L)$ with $L$ being a simple object in $\mscr \localsystem_G(X_w)$. Moreover, 
\begin{equation*}
  \begin{split} 
      \res^G \colon \mscr\perv_G(X,X^+) & \longrightarrow \mscr\perv(X,X^+) \\ 
      \mathbf{IC}^G_w(L) & \longmapsto \mathbf{IC}_w(\res^G(L))
    \end{split} 
\end{equation*}
for any $L \in \mscr\localsystem_G(X_w)$. 
\end{prop}

\begin{proof}
  The proof is formally the same as in theorem \ref{thm: simple objects of perverse motives on ind-schemes} with suitable modifications. The equality $\res^G(  \mathbf{IC}^G_w(L)) =  \mathbf{IC}_w(\res^G(L))$ is already explained in the proof of proposition \ref{prop: characterizations of equivariant intersection motives}.
\end{proof}

\begin{prop}
Let $G$ be a pro-algebraic group acting on stratified ind-varieties $X^+ = \coprod_{w \in W}X_w \longrightarrow X, Y^+ = \coprod_{r \in R}Y_r \longrightarrow Y$. Let $X^+ \times Y^+ = \coprod_{(w,r) \in W \times R}(X_w \times Y_r) \longrightarrow X \times Y$ be the product stratification endowed with the diagonal action. There is an isomorphism 
\begin{equation*}
   \mathbf{IC}^G_{w}(M) \boxtimes \mathbf{IC}_r^G(N) \simeq \mathbf{IC}^G_{(w,r)}(M \boxtimes N) 
\end{equation*}
with $M \in \mscr\localsystem_G(X_w), N \in \mscr\localsystem_G(Y_r)$.

\end{prop}

%\begin{proof}
%The left hand side is an object of $\derivednori^b_G(X \times Y,X^+ \times Y^+)$ thanks to {\color{red}corollary}. Now we use proposition \ref{prop: characterizations of equivariant intersection motives}: let us prove that $ \mathbf{IC}^G_{w}(M) \boxtimes \mathbf{IC}_r^G(N) $ satisfies the characterizing property of the right hand side. For instance, $\mathbf{IC}_{w}^G(M)_{\mid X_w} \in \derivednori^b_G(X_w,X_w)^{\leq -\dim(X_w)-1}, \mathbf{IC}_r^G(N)_{\mid Y_r} \in \derivednori^b_G(Y_r,Y_r)^{\leq -\dim(Y_r)-1}$ and thus by the exactness of $\boxtimes$, we see that $\left( \mathbf{IC}_{w}^G(M)_{\mid X_w} \boxtimes \mathbf{IC}_r^G(N)_{\mid Y_r} \right)_{\mid X_w \times Y_r} \simeq (\mathbf{IC}_{w}^G(M)_{\mid X_w}  \boxtimes \mathbf{IC}_r^G(N)_{\mid Y_r}  \in \derivednori^b_G(X \times Y, X^+ \times Y^+)^{\leq -\dim(X_w \times Y_r)}$ for all $(w,r) \in X^+ \times Y^+$. 
%\end{proof}

\begin{proof}
The left hand side is an object of $\derivednori^b_G(X \times Y,X^+ \times Y^+)$ thanks to corollary \ref{cor: box products of equivariant, stratified motives}. Now using proposition \ref{prop: characterizations of equivariant intersection motives}, we just have to prove that this already holds at the non-equivariant level, namely, $  \mathbf{IC}_{w}(M) \boxtimes \mathbf{IC}_r(N) \simeq \mathbf{IC}_{(w,r)}(M \boxtimes N)$. Let us prove that $ \mathbf{IC}_{w}(M) \boxtimes \mathbf{IC}_r(N) $ satisfies the characterizing property of the right hand side. For instance, $\mathbf{IC}_{w}(M)_{\mid X_w} \in \derivednori^b(X_w,X_w)^{\leq -\dim(X_w)-1}, \mathbf{IC}_r(N)_{\mid Y_r} \in \derivednori^b(Y_r,Y_r)^{\leq -\dim(Y_r)-1}$ and thus by the exactness of $\boxtimes$, we see that $\left( \mathbf{IC}_{w}(M)_{\mid X_w} \boxtimes \mathbf{IC}_r(N)_{\mid Y_r} \right)_{\mid X_w \times Y_r} \simeq (\mathbf{IC}_{w}(M)_{\mid X_w}  \boxtimes \mathbf{IC}_r(N)_{\mid Y_r} ) \in \derivednori^b(X \times Y, X^+ \times Y^+)^{\leq -\dim(X_w \times Y_r)}$ for all $(w,r) \in X^+ \times Y^+$. 
\end{proof}

\begin{cor}  \label{cor: intersection motives commute with galois extensions} 
Let $e \colon \Spec(F)  \longhookrightarrow \Spec(k)$ be a Galois extension of subfields of $\mathbb{C}$, then under the functor in proposition \ref{prop: galois descent of equivariant, stratified motives}
\begin{equation*}
\begin{split} 
     \mscr\perv_{G_k}(X_k,X^+_k) & \longrightarrow  \mscr\perv_{G_F}(X_F,X^+_F) \\ 
    \mathbf{IC}^{G_k}_w(L_k) & \longmapsto  e^*\mathbf{IC}_w^{G_k}(L_k) \simeq \mathbf{IC}^{G_F}_w(e^*L_k)
    \end{split} 
\end{equation*}
for any $L_k \in \mscr\localsystem_{G_k}(X_{w,k})$. 
\end{cor}

\begin{proof}
  This is because $e$ is \'etale so $e^* = e^!$ commutes with all involved operations. 
\end{proof}

Finally, we have an equivariant analogue of theorem \ref{thm: simple objects of perverse motives on ind-schemes}, the proof remains formally the same so we leave it to the reader. 

\begin{theorem} \label{thm: simple objects of equivariant perverse motives on ind-schemes}
Any simple object in $\mscr\perv_G(X,X^+)$ is of the form $\mathbf{IC}_w^G(L)$ with $L$ being a simple object in some $\mscr\localsystem_G(X_w)$. 
\end{theorem}

\subsection{(Shifted) equivariant Tate motives inside equivariant perverse Nori motives} We still assume hypothesis \ref{hyp: action of pro-group} in this subsection. 

\begin{defn}
  \begin{enumerate} 
  \item We define the \textit{category of equivariant, stratified derived Tate motives} $\mathbf{DT}^b_G(X,X^+)$ to be the subcategory of $\derivednori^b_G(X,X^+)$ whose images under the restriction functor lie in $\mathbf{DT}^b(X)$. 
    \item The \textit{category of shifted, equivariant, stratified Tate motives} $\mscr\textnormal{Tate}_{G,p}(X,X^+)$ is the full subcategory of $\mscr\perv_G(X,X^+)$ whose images under the forgetful functors are shifted Tate motives, namely, belong to $\mscr\operatorname{Tate}_p(X,X^+)$. We note that there is a fully faithful embedding
    \begin{equation*}
        \mscr\operatorname{Tate}_{G,p}(X,X^+) \longhookrightarrow \mathbf{DT}^b_G(X,X^+).
    \end{equation*}
    Indeed, we need to check that 
    \begin{equation*}
        \mscr\operatorname{Tate}_{G,p}(X,X^+) \longhookrightarrow \mscr\perv_G(X,X^+) \longhookrightarrow \derivednori^b_G(X,X^+)
    \end{equation*}
    takes values in $\mathbf{DT}^b_G(X,X^+)$ and hence it suffices to look at underlying non-equivariant motives (see definition \ref{defn: stratified derived Tate motives}). 
    \end{enumerate} 
\end{defn}

\begin{cor}
The composition 
\begin{equation*}
    \mathbf{DT}^b_G(X,X^+) \longhookrightarrow \derivednori^b_G(X,X^+) \overset{\phnor^n}{\longrightarrow} \mscr\perv_G(X,X^+) 
\end{equation*}
takes values in $\mscr\textnormal{Tate}_{G,p}(X,X^+)$ for any $n \in \mathbb{Z}$. 
\end{cor}

\begin{proof}
  This follows from lemma \ref{lem: Tate cohomological functors on ind-schemes}. 
\end{proof}

\begin{lem} 
If $L \in \mscr\textnormal{Tate}_G(X_w) \subset \mscr\localsystem_G(X_w)$, then $\mathbf{IC}_w^G(L) \in \mscr\textnormal{Tate}_{G,p}(X,X^+)$. 
\end{lem}

\begin{proof}
  Since $\mathbf{IC}_w^G(L) \in \mscr\perv_G(X,X^+)$, it remains to show that $\res^G(\mathbf{IC}_w^G(L)) = \mathbf{IC}_w(\res^G(L))$ belongs to $\mscr\operatorname{Tate}_p(X,X^+)$. It suffices to check that $\res^G(L) \in \mscr\operatorname{Tate}_p(X_w)$ but this is clear from the definition. 
\end{proof}
\subsection{Weights on Equivariant Motives}  Let $n \in \mathbb{Z}$, a motive $M \in \derivednori^b_G(X,X^+)$ has weights $\leq n$ (resp, $\geq n$) if and only if $\res^G(M) \in \derivednori^b(X,X^+)$ weights $\leq n$ (resp, $\geq n$). 
\begin{cor} 
Let $f \colon (X,X^+) \longrightarrow (Y,Y^+)$ be a $G$-equivariant morphism so that the restriction to each stratum is $G$-equivariant, if the operations $(f^*,f_*,f_!,f^!,\otimes, \underline{\Hom},\boxtimes)$ are well-defined, the they enjoy the same properties as in {\color{blue}section 2.1}. 
\end{cor}

\begin{proof}
  This is obvious, given proposition \ref{prop: weight-exactness of operations on ind-schemes}. 
\end{proof}

\begin{cor}
Let $j \colon (U,U^+) \longrightarrow (X,X^+)$ be a $G$-equivariant, locally closed immersion of and $L \in \mscr\perv_G(U,U^+,n)$ be an object pure of weight $n$, then $j_{!*}(L)$ is pure of weight $n$. 
\end{cor}

\begin{proof}
This is similar to the classical case.  
\end{proof}
\section{Derived Nori Motives on Partial Affine Flag Varieties}
In this section, we study the $\lnor^+G$-equivariant (derived) Nori motive on the affine Grassmannian $\mathrm{Gr}_G$ associated with a reductive group $G$. Although our ultimate goal is the motives associated with $\mathrm{Gr}_G$ but due to technical reasons (certain reduction step), it is convenient to us to define the general partial (affine) flag varieties $\flag_{G,\mathbf{f}}$ associated with facets $\mathbf{f}$ in the spherical apartment, in which $\mathrm{Gr}_G = \flag_{G,0}$ ($0$ is the origin) is a special case.
\subsection{Recollections on partial affine flag varieties}.
Let $G$ be a split reductive group over $k$. Let us choose $T \subset B \subset G$, a maximal torus $T$ contained in a Borel subgroup $B$. Let $W = N_G(T)/T$ be the Weyl group, where $N_G(T)$ is a normalizer of $T$ in $G$. 

We denote by $X^*(T), X_*(T)$ the characters and cocharacters, respectively. Let $R \subset X^*(T), R^{\vee} \subset X_*(T)$ be the root and coroot systems of $(G,T)$, respectively. The choice of a Borel subgroup $B \subset G$ leads to a set $R^+ \subset R$ of \textit{positive roots}. Let $\acal = X_*(T) \otimes \mathbb{R}$ be the standard apartment and we identify $X^*(T) \otimes \mathbb{R}$ with $\acal^{\vee}$ via the natural pairing $X^*(T) \times X_*(T)\longrightarrow \mathbb{Z}$. Hence we view elements in $X^*(T) \otimes \mathbb{R}$ as $\mathbb{R}$-linear map on $\acal$. Let $R_{\aff} = (R \times \mathbb{Z})$ be the set of \textit{affine roots}. 

For each $(\alpha,n) \in R_{\aff}$, we consider the affine hyperplane $H_{\alpha,n} = \left \{x \in \acal \mid \left <\alpha,x \right> = n \right \}$. An \textit{alcove} is a connected component of $\acal \setminus \left( \bigcup_{(\alpha,n)\in R_{\aff}} H_{\alpha,n}\right)$. There is a \textit{standard alcove} defined via $B$
\begin{equation*} 
\mathbf{a}_0 = \left \{x \in \acal \mid 0 < \left<\alpha,x \right> < 1 \ \forall \ \alpha \in R^+\right \}
\end{equation*} 
Each closure $\overline{\mathbf{a}}$ of an alcove is decomposed into locally closed subsets $\mathbf{f}$, called \textit{facets}. For instance, if $G = \mathrm{SL}_3$ (the $A_2$-system), $\overline{\mathbf{a}}_0$ is a triangle with three vertices (including the origin), three edges, one (open) triangle. 

For each $(\alpha,n) \in R_{\aff}$, the hyperplane $H_{\alpha,n}$ admits an affine reflection $s_{\alpha,n}(x) = x - (\left<x,\alpha\right> - n)\alpha^{\vee}$. The \textit{affine Weyl group} $W_{\aff}$ is the subgroup of the group of affine transformations $\operatorname{Aff}(\acal)$ generated by $s_{\alpha,n}$, i.e., $W_{\aff} = \left <s_{\alpha,n} \mid (\alpha,n) \in R_{\aff} \right > \subset \operatorname{Aff}(\acal)$. It is well-known that $W_{\aff} \simeq W \ltimes \mathbb{Z}R$ and $W_{\aff}$ is a Coxeter group (see for instance \cite{iwahori+matsumoto-1965}). Indeed, let $\alpha_0$ be the highest root (with respect to the choice of $B$) in the finite root system, then 
\begin{equation*}
    W_{\aff} = \left <s_{\alpha_0,1}, s_{\alpha,0} \mid \alpha \in R \right> \subset \operatorname{Aff}(\acal). 
\end{equation*}
We refer the elements of $\left \{s_{\alpha_0,1},s_{\alpha,0} \mid \alpha \in R \right \}$ as \textit{simple affine reflections}. In particular, $W_{\aff}$ admits a length function 
\begin{equation*}
    \ell \colon W_{\aff} \longrightarrow \mathbb{Z}_{>0} \\ 
\end{equation*}
The \textit{extended affine Weyl group} (also called the Iwahori-Weyl group) is the group $W_{\ext} = W \ltimes X_*(T)$. Despite of not being a Coxeter group, $W_{\ext}$ still admits a length function 
\begin{equation*} 
\ell \colon W_{\ext} \longrightarrow \mathbb{Z}_{\geq 0}.
\end{equation*} 
If $w \in W_{\aff} \subset W_{\ext}$, then $\ell(w)$ is the minimal number of simple reflections needed to represent $w$. Elements in $W_{\ext}$ can have length zero. Let us set $\Omega = \left \{w \in W_{\ext} \mid \ell(w) = 0 \right \}$, then there is an isomorphism
\begin{equation*}
    W_{\ext} \simeq \Omega \ltimes W_{\aff}. 
\end{equation*}
In particular, any element $w$ in $W_{\ext}$ is expressed as a product of the form $sw_1\cdots w_n$ with $s \in \Omega$ and $w_i \in W_{\aff}$. For such an expression with minimal $n$, we have $\ell(w) = n$. For each facet $\mathbf{f}$, let $W_{\mathbf{f}} \subset W_{\ext}$ be the finite subgroup generated by those $s_{\alpha,n}$ with $(\alpha,n)_{\mid \mathbf{f}}=0$. 
\begin{ex} \label{ex: length function} 
For facets $\mathbf{f},\mathbf{f}' \subset \overline{\mathbf{a}}_0$ contained in the closure of the standard alcove, the length function on $W_{\ext}$ induces a length function on the double class $W_{\mathbf{f}'}\setminus W_{\ext} /W_{\mathbf{f}}$. For instance, if $\mathbf{f} = \mathbf{f}' = 0$ is the origin, then $W_{\mathbf{f}'} \setminus W_{\ext} / W_{\mathbf{f}}  = X_*(T)^+$ is the set of positive coroots and the length function is given by $\ell(\mu) = \left<2\rho,\mu\right>$, where $2\rho = \sum_{\alpha \in R^+}\alpha$ is the sum of positive roots. 
\end{ex}
\subsection{Parahoric subgroups}. We are interested in certain closed subgroups of the \textit{loop group}
\begin{equation*}
     \lnor G\colon \Alg_k \longrightarrow \sets  \ \ \ \ R \longmapsto G(R(\!(t)\!))
\end{equation*}
called \textit{parahoric subgroups}. In \cite{bruhat+tits-1984}\cite{bruhat+tits-1987}, Bruhat and Tits show that for each facet $\mathbf{f}$, there is an algebraic $k[\![t]\!]$-group scheme $\gscr$ so that $\gscr \otimes k(\!(t)\!) \simeq G \otimes k(\!(t)\!)$. Let us denote by $\gscr_{\mathbf{f}}$ its connected component of the identity. Let us set $\pscr_{\mathbf{f}}(R) = \gscr_{\mathbf{f}}(R[\![t]\!]) \subset LG$ for any $k$-algebra $R$. The $\pscr_{\mathbf{f}}$ is called \textit{parahoric subgroups} associated with the facet $\mathbf{f}$. It is a closed subgroup of $\lnor G$ and is pro-algebraic group with a presentation of the form $\pscr_{\mathbf{f}} = \lim_{n \in \mathbb{N}} \pscr_{\mathbf{f},n}$ so that $\Ker(\pscr_{\mathbf{f},n+1} \longrightarrow \pscr_{\mathbf{f},n})$ is a vector group (in fact, isomorphic to $\mathbb{A}_k^{\dim(G)}$). 
\begin{ex}
\begin{enumerate}
    \item If $\mathbf{f} = 0$ is the origin, then $\pscr_{\mathbf{f}} = \lnor^+G$ is the \textit{positive loop group}, i.e., $\lnor^+G(R) = G(R[\![t]\!])$ for any $k$-algebra $R$. Clearly, $\lnor^+G = \lim_{n \in\mathbb{N}}\lnor^+_nG$ with $\lnor^+_nG(R) = G(R[t]/(t^{n+1}))$ and each $\Ker(\lnor^+_{n+1}G \longrightarrow \lnor^+_nG)$ is a vector group. 
    \item If $\mathbf{f}= \mathbf{a}_0$ is the standard alcove, then $\pscr_{\mathbf{f}} = \inor$ is the \textit{Iwahori subgroup}, i.e., the inverse image of $B$ under the reduction $\pi \colon \lnor^+G \longrightarrow G, t \longmapsto 0$. If we write $\inor_n = \pi^{-1}_n(B)$ with $\pi_n \colon \lnor^+_nG\longrightarrow G$ the reduction map, then $\inor = \lim_{n \in \mathbb{N}} \inor_n$ and each $\Ker(\inor_{n+1} \longrightarrow \inor_n)$ is a vector group as in the case of the positive loop group. 
\end{enumerate}
\end{ex}
For each facet $\mathbf{f}$, the \'etale sheafification
\begin{equation*}
    \begin{split} 
    \lnor G/\pscr_{\mathbf{f}} \colon \Alg_k & \longrightarrow \sets \\ 
    R & \longmapsto G(R(\!(t)\!))/\pscr_{\mathbf{f}}(R[\![t]\!]).
    \end{split} 
\end{equation*}
is called the \textit{partial affine flag variety}, denoted $\flag_{G,\mathbf{f}}$. It is an ind-projective ind-varieties. The action of $\lnor G$ on $\flag_{G,\mathbf{f}}$ restricts to an action of $\pscr_{\mathbf{f}'}$ on $\flag_{G,\mathbf{f}}$ for any other facet $\mathbf{f}'$. Moreover, there exists a presentation $\flag_{G,\mathbf{f}} = \colim_{n \in \mathbb{N}}\flag_{G,\mathbf{f},n}$ so that $\flag_{G,\mathbf{f},n}$ is a projective variety, stable under the $\pscr_{\mathbf{f}'}$-action and the $\pscr_{\mathbf{f}'}$-action factors through an action of some $\pscr_{\mathbf{f}',n}$. 
\begin{ex}
Two most important cases for us are $\flag_{G,0} = \mathrm{Gr}_G$ the \textit{affine Grassmannian} and $\flag_{G,\mathbf{a}_0} = \flag_G$ the \textit{full affine flag variety}. For instance, if $G = \mathrm{GL}_n$ is some general linear group, there is an identification $\mathrm{Gr}_G(R) = \left \{\textnormal{lattices of} \ R(\!(t)\!)^n \right \}$ and hence there is a presentation $\mathrm{Gr}_G(R) = \colim_{N \in \mathbb{N}}\mathrm{Gr}_{G,N}(R)$ with $\mathrm{Gr}_{G,N}(R) = \left \{L \subset R(\!(t)\!)^n \ \textnormal{lattice} \mid t^N R[\![t]\!] \subset L \subset t^{-N}R[\![t]\!]\right \}$ and these $\mathrm{Gr}_{G,N}$'s are projective.
\end{ex}

\subsection{Whitney-Nori conditions} Let us discuss the (affine) Bruhat decomposition. Let $\mathbf{f},\mathbf{f}' \subset \overline{\mathbf{a}}_0$ be facets contained in the closure of the standard alcove, the decomposition into orbits of $\pscr_{\mathbf{f}'}$-action on $\flag_{G,\mathbf{f}}$ takes the following form
\begin{equation*}
    \flag_{G,\mathbf{f}} = \coprod_{w \in W_{\mathbf{f}'}\setminus W_{\ext} / W_{\mathbf{f}}}\flag_G^{w}(\mathbf{f}',\mathbf{f}).
\end{equation*}
The varieties $\flag_G^{w}(\mathbf{f}',\mathbf{f})$ are \textit{Schubert varieties}. Let us mention some standard examples:
\begin{ex} \label{ex: examples of affine Bruhat decompositions} 
\begin{enumerate} 
\item If $\mathbf{f}' = \mathbf{a}_0$, then $\flag_G^{w}(\mathbf{f}',\mathbf{f}) \simeq \mathbb{A}_k^{\ell(w)}$, where $\ell \colon W_{\mathbf{f}'} \setminus W_{\ext} / W_{\mathbf{f}} \longrightarrow \mathbb{Z}_{\geq 0}$ is the induced length function in example \ref{ex: length function}.
\item If $\mathbf{f} = \mathbf{f}' = 0$, then although $\mathrm{Gr}_G^w = \flag_G^{w}(\mathbf{f}',\mathbf{f})$ is not an affine space, we still have that $\dim(\flag_G^{w}(\mathbf{f}',\mathbf{f})) = \ell(w) = \left <2\rho,w \right>$. Moreover, $\flag_G^{w}(\mathbf{f}',\mathbf{f})$ is an affine bundle over the partial flag variety $G/B$. In particular, it is simply connected in the analytic topology. 
\item The following situation reminds us about the stratification of the partial flag variety $G/P$ into $B$-orbits (where $P$ is a parabolic subgroup). Let $s \in W_{\ext}$ be a \textit{simple affine reflection}, there is a unique facet $\mathbf{f}_s$ of maximal dimension in $\overline{\mathbf{a}}_0$ such that $s(\mathbf{f}_s) = \mathbf{f}_s$. The group $W_{\mathbf{f}_s}$ is then generated by $s$ and the canonical projection $\flag_{G,0} \longrightarrow \flag_{G,\mathbf{f}_s}$ is a fibration in the \'etale topology with fibers $\pscr_{\mathbf{f}_s}/\inor \simeq \mathbb{P}^1_k$. 
\end{enumerate} 
\end{ex}
\begin{prop}
Let $\mathbf{f},\mathbf{f}'$ be facets contained in $\overline{\mathbf{a}}_0$, then the decomposition
\begin{equation*}
    \flag_{G,\mathbf{f}} = \coprod_{w \in W_{\mathbf{f}'}\setminus W_{\ext} / W_{\mathbf{f}}} \flag_G^{w}(\mathbf{f}',\mathbf{f})
\end{equation*}
is a Whitney-Nori stratification. 
\end{prop}

\begin{proof}
 By remark \ref{rmk: testing Whitney-Nori on Betti realization}, it suffices to check this on $\derivedcat^b_{\ct}((-)^{\an},\mathbb{Q})$, namely we have to verify that
\begin{equation*}
    \iota_w^*\iota_{s,*}(M) \in \derivedcat^b_{\localsystem(\flag_G^{w}(\mathbf{f}',\mathbf{f}))}(\flag_G^{w}(\mathbf{f}',\mathbf{f}),\mathbb{Q}) \ \forall \ w,s \in W_{\mathbf{f}'}\setminus W_{\ext} / W_{\mathbf{f}}, M \in \derivedcat^b_{\localsystem(\flag_G^{s}(\mathbf{f}',\mathbf{f}))}(\flag_G^{s}(\mathbf{f}',\mathbf{f}),\mathbb{Q}).
\end{equation*}
However, since $\flag_G^w(\mathbf{f}',\mathbf{f})$ is simply connected (see example \ref{ex: examples of affine Bruhat decompositions}), we can argue as in example \ref{ex: Bruhat decomposition is Whitney-Nori} to reduce to the case $M = \mathds{1}$. The rest is formally identical to the proof of \cite[Theorem 5.1.1]{richarz+scholbach-2020}.
\end{proof}

\begin{cor}
\begin{enumerate} 
\item There is a well-defined category $\derivednori^b(\flag_{G,\mathbf{f}},\pscr_{\mathbf{f}'})$ (the second variable $\pscr_{\mathbf{f}'}$ indicates the stratification into $\pscr_{\mathbf{f}'}$-orbits) with a $t$-structure whose heart is $\mscr\perv(\flag_{G,\mathbf{f}},\pscr_{\mathbf{f}'})$. 
\item The category $\mscr\perv(\flag_{G,\mathbf{f}},\pscr_{\mathbf{f}'})$ is generated by intersection motives of the form $\mathbf{IC}_w(L)$ with $w \in W_{\mathbf{f}'} \setminus W_{\ext} / W_{\mathbf{f}}$ and $L \in \mscr\localsystem(\flag_G^w(\mathbf{f}',\mathbf{f}))$.
\end{enumerate} 
\end{cor}

\begin{proof}
The fist part is proposition \ref{prop: equivalent definitions of stratified motives} and second part is theorem \ref{thm: simple objects of perverse motives on ind-schemes}.
\end{proof}

\begin{cor}
\begin{enumerate} 
\item The affine Bruhat decomposition satisfies hypothesis \ref{hyp: action of pro-group}, hence there is a well-defined category $\derivednori^b_{\pscr_{\mathbf{f}'}}(\flag_{G,\mathbf{f}},\pscr_{\mathbf{f}'})$ with a $t$-structure whose heart is $\mscr\perv_{\pscr_{\mathbf{f}'}}(\flag_{G,\mathbf{f}})$. 
\item The category $\mscr\perv_{\pscr_{\mathbf{f}'}}(\flag_{G,\mathbf{f}},\pscr_{\mathbf{f}'})$ is generated by intersection motives of the form $\mathbf{IC}_w^{\pscr_{\mathbf{f}'}}(L)$ with $w \in W_{\mathbf{f}'} \setminus W_{\ext} / W_{\mathbf{f}}$ and $L \in \mscr\localsystem_{\pscr_{\mathbf{f}'}}(\flag_G^w(\mathbf{f}',\mathbf{f}))$.
\item The forgetful functor
\begin{equation*}
    \mscr\perv_{\pscr_{\mathbf{f}'}}(\flag_{G,\mathbf{f}},\pscr_{\mathbf{f}'}) \longrightarrow \mscr\perv(\flag_{G,\mathbf{f}},\pscr_{\mathbf{f}'})
\end{equation*}
is fully faithful and the essentially image is given by 
\begin{equation*}
    \left \{M \in \mscr\perv(\flag_{G,\mathbf{f}},\pscr_{\mathbf{f}'}) \mid a^*(M) \simeq p^*(M) \right \}, 
\end{equation*}
where $a,p \colon \pscr_{\mathbf{f}'} \times \flag_{G,\mathbf{f}} \longrightarrow \flag_{G,\mathbf{f}}$ denote the action and the projection, respectively. 
\end{enumerate} 
\end{cor}

\begin{proof}
The first part is proposition \ref{prop: equivariant perverse Nori motives on ind-schemes}, the second part is theorem \ref{thm: simple objects of equivariant perverse motives on ind-schemes} and the third part is proposition \ref{prop: forgetful functors of ind-schemes are fully faithful}. 
\end{proof}
\subsection{The convolution product} Let $\mathbf{f}',\mathbf{f},\mathbf{f}^{''} \subset \overline{\mathbf{a}}_0$ be facets. The ind-projective ind-variety 
\begin{equation*}
    \flag_{G,\mathbf{f}} \widetilde{\times} \flag_{G,\mathbf{f}^{''}} \coloneqq (\lnor G \times^{\pscr_{\mathbf{f}}} \lnor G/\pscr_{\mathbf{f}^{''}})^{\et}
\end{equation*}
is called the \textit{convolution flag variety}. Let us consider morphisms
\begin{equation*}
    \begin{tikzcd}[sep=large] 
     \flag_{G,\mathbf{f}} \times  \flag_{G,\mathbf{f}^{''}} & \lnor G \times \flag_{G,\mathbf{f}^{''}} \arrow[r,"q"] \arrow[l,"p",swap] &   \flag_{G,\mathbf{f}} \widetilde{\times} \flag_{G,\mathbf{f}^{''}} \arrow[r,"m"] & \flag_{G,\mathbf{f}^{''}}.
    \end{tikzcd}
\end{equation*}
The functor
\begin{equation*}
     q^* \colon \derivednori^b_{\pscr_{\mathbf{f}'}}(\flag_{G,\mathbf{f}} \widetilde{\times} \flag_{G,\mathbf{f}^{''}},\pscr_{\mathbf{f}'}) \overset{\sim}{\longrightarrow}  \derivednori^b_{\pscr_{\mathbf{f}'} \times \pscr_{\mathbf{f}}}(\lnor G \times \flag_{G,\mathbf{f}^{''}},\pscr_{\mathbf{f}'} \times \pscr_{\mathbf{f}})
\end{equation*}
is an equivalence of categories by means of theorem \ref{theorem: properties of special equivariant six operations}. Let $M \in \derivednori^b_{\pscr_{\mathbf{f}'}}(\flag_{G,\mathbf{f}},\pscr_{\mathbf{f}'}), N \in \derivednori^b_{\pscr_{\mathbf{f}}}(\flag_{G,\mathbf{f}^{''}},\pscr_{\mathbf{f}})$, the motives $M \boxtimes N$ is in $\derivednori^b_{\pscr_{\mathbf{f}'}\times \pscr_{\mathbf{f}}}( \flag_{G,\mathbf{f}} \times  \flag_{G,\mathbf{f}^{''}},\pscr_{\mathbf{f}'}\times \pscr_{\mathbf{f}})$ thanks to corollary \ref{cor: box products of products of groups} and the motive $p^*(M \boxtimes N)$ is in $\derivednori^b_{\pscr_{\mathbf{f}'} \times \pscr_{\mathbf{f}}}(\lnor G \times \flag_{G,\mathbf{f}^{''}},\pscr_{\mathbf{f}'} \times \pscr_{\mathbf{f}})$ and hence descends to a motive $M \widetilde{\boxtimes}N \in  \derivednori^b_{\pscr_{\mathbf{f}'}}(\flag_{G,\mathbf{f}} \widetilde{\times} \flag_{G,\mathbf{f}^{''}},\pscr_{\mathbf{f}'})$ by the formula
\begin{equation*}
    q^*(M \widetilde{\boxtimes}N) \simeq p^*(M \boxtimes N).
\end{equation*}
We define the convolution product 
\begin{equation*}
     (-) \star (-) \colon  \derivednori^b_{\pscr_{\mathbf{f}'}}(\flag_{G,\mathbf{f}},\pscr_{\mathbf{f}'}) \times  \derivednori^b_{\pscr_{\mathbf{f}}}(\flag_{G,\mathbf{f}^{''}},\pscr_{\mathbf{f}}) \longrightarrow  \derivednori^b_{\pscr_{\mathbf{f}'}}(\flag_{G,\mathbf{f}^{''}},\pscr_{\mathbf{f}'})
\end{equation*}
by the formula
\begin{equation*}
    M \star N \coloneqq m_*(M \widetilde{\boxtimes } N) = m_*(q^*)^{-1}p^*(M \boxtimes N) \in \derivednori^b_{\pscr_{\mathbf{f}'}}(\flag_{G,\mathbf{f}^{''}},\pscr_{\mathbf{f}'}). 
\end{equation*}
In particular, if $\mathbf{f}'=\mathbf{f}= \mathbf{f}^{''}=0$, there is a convolution product
\begin{equation*}
    (-) \star (-) \colon \derivednori^b_{\lnor^+G}(\mathrm{Gr}_G,\lnor^+G) \times \derivednori^b_{\lnor^+G}(\mathrm{Gr}_G,\lnor^+G)\longrightarrow \derivednori^b_{\lnor^+G}(\mathrm{Gr}_G,\lnor^+G)
\end{equation*}
of the affine Grassmannian. This special convolution is studied in the upcoming sections.
\begin{prop}
The convolution product is associative. More precisely, let $M_1 \in   \derivednori^b_{\pscr_{\mathbf{f}'}}(\flag_{G,\mathbf{f}},\pscr_{\mathbf{f}'}), M_2 \in   \derivednori^b_{\pscr_{\mathbf{f}}}(\flag_{G,\mathbf{f}},\pscr_{\mathbf{f}}), M_3 \in \derivednori^b_{\pscr_{\mathbf{f}}}(\flag_{G,\mathbf{f}^{''}},\pscr_{\mathbf{f}})$, then there is a canonical isomorphism
\begin{equation*}
    (M_1 \star M_2) \star M_3 \simeq M_1 \star (M_2 \star M_3)
\end{equation*}
in $\derivednori^b_{\pscr_{\mathbf{f}'}}(\flag_{G,\mathbf{f}^{''}},\pscr_{\mathbf{f}'})$. In particular, the convolution product
\begin{equation*}
    (-) \star (-) \colon \derivednori^b_{\lnor^+G}(\mathrm{Gr}_G,\lnor^+G) \times \derivednori^b_{\lnor^+G}(\mathrm{Gr}_G,\lnor^+G)\longrightarrow \derivednori^b_{\lnor^+G}(\mathrm{Gr}_G,\lnor^+G)
\end{equation*}
is associative. 
\end{prop}

\begin{proof}
This is an application of base change theorems. The proof is essentially the same as \cite[Lemma 3.7]{richarz+scholbach-2021} though in $\loccit$, the author writes everything in a stacky style. 
\end{proof}

\begin{rmk}
We can cheat in the case of $\derivednori^b_{\lnor^+G}(\mathrm{Gr}_G,\lnor^+G)$ by checking the associativity under the Betti (resp. $\ell$-adic) realization and uses the established result in the analytic setting (resp. $\ell$-adic).
\end{rmk}

\begin{prop} \label{prop: convolution products preserve derived Tate motives}
The convolution product restricts to a convolution product
\begin{equation*}
     (-) \star (-) \colon  \mathbf{DT}^b_{\pscr_{\mathbf{f}'}}(\flag_{G,\mathbf{f}},\pscr_{\mathbf{f}'}) \times  \mathbf{DT}^b_{\pscr_{\mathbf{f}}}(\flag_{G,\mathbf{f}^{''}},\pscr_{\mathbf{f}}) \longrightarrow  \mathbf{DT}^b_{\pscr_{\mathbf{f}'}}(\flag_{G,\mathbf{f}^{''}},\pscr_{\mathbf{f}'})
\end{equation*}
on Tate motives.
\end{prop}

\begin{proof}
   The operations $\boxtimes, p^*, q^*$ preserve Tate motives so we just have to prove that $m_* \colon \mathbf{DT}^b_{\pscr_{\mathbf{f}'}}(\flag_{G,\mathbf{f}} \widetilde{\times} \flag_{G,\mathbf{f}^{''}},\pscr_{\mathbf{f}'}) \longrightarrow \mathbf{DT}^b_{\pscr_{\mathbf{f}'}}(\flag_{G,\mathbf{f}^{''}},\pscr_{\mathbf{f}'})$ (this is well-defined thanks to \ref{prop: pushforwards of stratified  motives}) preserves Tate motives, namely, it restricts to $m_* \colon \mathbf{DN}^b_{\pscr_{\mathbf{f}'}}(\flag_{G,\mathbf{f}} \widetilde{\times} \flag_{G,\mathbf{f}^{''}},\pscr_{\mathbf{f}'}) \longrightarrow \mathbf{DN}^b_{\pscr_{\mathbf{f}'}}(\flag_{G,\mathbf{f}^{''}},\pscr_{\mathbf{f}'}G)$. It is enough to work at a non-equivariant level and this is \cite[Theorem 3.17]{richarz+scholbach-2021}. 

   %We note that in \cite[Theorem 3.17]{richarz+scholbach-2021}, the author makes a reduction argument from stratifications into $\lnor^+G$-orbits to stratifications into Iwahori orbits and hence a more general definition $\derivednori^b_{P_{\mathbf{f}_1}}(\lnor G/P_{\mathbf{f}_2})$ (where $\mathbf{f}_i$ are facets contained in the closure of the standard alcove and $P_{\mathbf{f}_i}$ are associated parahoric subgroups, if $\mathbf{f}_1=\mathbf{f}_2=0$ then $\derivednori^b_{P_{\mathbf{f}_1}}(\lnor G/P_{\mathbf{f}_2}) = \derivednori^b_{\lnor^+G}(\lnor G/\lnor^+G) = \derivednori_{\lnor^+G}^b(\mathrm{Gr}_G)$) is involved; however, strictly speaking, the only formal properties needed are stable under extensions and direct summands and under the standard conjectures, both settings are the same so formally, arguments can be used interchangeably). 
\end{proof}
\section{The Cases of Affine Grassmannians}
In this section, we apply the tools developed in the preceding ones to affine Grassmannians. We begin with the Kazhdan-Lusztig parity vanishing, which provides a deeper look at the structure of the categories $\mscr\perv(\mathrm{Gr}_G,\lnor^+G), \mscr\perv_{\lnor^+G}(\mathrm{Gr}_G,\lnor^+G)$. 
\subsection{Kazhdan-Lusztig parity vanishing and consequences}
Thanks to the tools that we develop in preceding sections, we have at hands a category of stratified Nori motives $\mscr\perv(\mathrm{Gr}_G,\lnor^+G)$ and its equivariant counterpart $\mscr\perv_{\lnor^+G}(\mathrm{Gr}_G,\lnor^+G)$ embedded (fully faithful) in $\mscr\perv(\mathrm{Gr}_G,\lnor^+G)$ with essential image given by
\begin{equation*}
    \left \{M \in \mscr\perv(\mathrm{Gr}_G,\lnor^+G,n) \mid a^*(M) \simeq p^*(M) \right \}
\end{equation*}
where $a,p \colon \lnor^+G \times \mathrm{Gr}_G \longrightarrow \mathrm{Gr}_G$ denotes the action and the projection, respectively. Now we would like to study these categories more deeply. For each $\lambda \in X_*(T)^+ = W_0 \setminus W_{\ext} / W_0$, let $\mathrm{Gr}^{\lambda}_G = \flag_G^{\lambda}(0,0)$ and $\mathrm{Gr}_G^{\leq \lambda} = \overline{\mathrm{Gr}}_G^{\lambda}$ and $j_{\lambda} \colon \mathrm{Gr}^{\lambda}_G \longhookrightarrow \mathrm{Gr}^{\leq \lambda}_G, i_{\lambda} \colon \overline{\mathrm{Gr}}^{\leq \lambda}_G \longhookrightarrow \mathrm{Gr}_G$ be the associated immersions. We recall the definition of the \textit{Nori's intersection motive} 
\begin{equation*}
    \mathbf{IC}_{\lambda}(L) \coloneqq (i_{\lambda})_*(j_{\lambda})_{!*}(L[\left<2\rho,\lambda \right>]) \in \mscr\perv(\mathrm{Gr}_G,\lnor^+G)
\end{equation*}
where $L \in \mscr\localsystem(\mathrm{Gr}^{\lambda}_G)$. 
\begin{prop}[Kazhdan-Lusztig parity vanishing] \label{thm: Kazhdan-Lusztig parity vanishing}
For $\lambda \in X_*(T)^+$
\begin{equation*} 
  \cthnor^n(\mathbf{IC}_{\lambda}(L)) \neq 0 \ \forall \ n \not\equiv \left<2\rho,\lambda \right> \ (\operatorname{mod}  2)
\end{equation*}
where $L\in \mscr\localsystem(\mathrm{Gr}_G^{\lambda})$ is a motivic local system. Similarly, if $L \in \mscr\localsystem_{\lnor^+G}(\mathrm{Gr}_G^{\lambda})$, then 
\begin{equation*} 
  \cthnor^n_{\lnor^+G}(\mathbf{IC}_{\lambda}^{\lnor^+G}(L)) \neq 0 \ \forall \ n \not\equiv \left<2\rho,\lambda \right> \ (\operatorname{mod}  2)
\end{equation*}
\end{prop}

\begin{proof}
%By the observation above, it suffices to assume that $L$ is indecomposable.
We note that $\rat(\mathbf{IC}_{\lambda}(L)) = \mathbf{IC}_{\lambda}(\rat(L))$ and $\rat(L) \in \localsystem(\mathrm{Gr}_G^{\lambda}) = \operatorname{Vect}_{\mathbb{Q}}$ since $\mathrm{Gr}^{\lambda}_G$ is simply connected. Thus $\rat(\mathbf{IC}_{\lambda}(L)) = \mathbf{IC}_{\lambda}(\mathds{1})^{\oplus m}$ for some $m$. At this step we are beneficial from the corresponding statements in \cite{gaitsgory-2001}\cite{haines-2004} (see also, for instance, \cite{baumann+riche-2018}). The result then follows from the compatibility $\rat_{\mathrm{Gr}_G} \circ \cthnor^n =  \cthnor^n \circ \rat_{\mathrm{Gr}_G}$ (see lemma \ref{lem: motivic perverse and constructible sheaves on ind-schemes}) and the conservativity of $\rat_{\mathrm{Gr}_G}$. The equivariant case follows from the non-equivariant one thanks to conservativity. 
\end{proof}

\begin{prop} \label{prop: a corollary of parity vanishing}
Let $\lambda,\mu \in X_*(T)^+$ and $M \in \mscr\localsystem(\mathrm{Gr}^{\lambda}_G), N \in \mscr\localsystem(\mathrm{Gr}^{\mu}_G)$ be motivic local systems, then 
\begin{equation*}
    \Hom_{\derivednori^b(\mathrm{Gr}_G,\lnor^+G)}(\mathbf{IC}_{\lambda}(M),\mathbf{IC}_{\mu}(N)[+1]) = \begin{cases} 
    \Hom_{\derivednori^b(\mathrm{Gr}^{\lambda}_G)}(M,N[+1]) & \lambda = \mu  \\ 
    0 & \lambda \neq \mu\\ 
    \end{cases}
\end{equation*}
%Similarly, if $M \in \mscr\localsystem_{\lnor^+G}(\mathrm{Gr}^{\lambda}_G)$ and $N \in \mscr\localsystem_{\lnor^+G}(\mathrm{Gr}_G^{\mu})$ are equivariant motivic local systems, then 
%\begin{equation*}
 %   \Hom_{\derivednori^b_{\lnor^+G}(\mathrm{Gr}_G,\lnor^+G)}(\mathbf{IC}_{\lambda}^{\lnor^+G}(M),\mathbf{IC}_{\mu}^{\lnor^+G}(N)[+1]) = \begin{cases} 
 %   \Hom_{\derivednori^b_{\lnor^+G}(\mathrm{Gr}^{\lambda}_G)}(M,N[+1]) & \lambda = \mu  \\ 
 %   0 & \lambda \neq \mu\\ 
 %   \end{cases}
%\end{equation*}
\end{prop}

\begin{proof}
The proof of the equivariant is identical to the proof of the non-equivariant case so we just focus on the first one. We follow the argument in \cite[Proof of Proposition 4.4]{baumann+riche-2018}\cite[Corollary 5.5]{richarz+scholbach-2021} (see also \cite{gaitsgory-2001}). 
\begin{enumerate} 
\item \textbf{First case $\lambda = \mu$}.  If $\lambda = \mu$, let us consider the diagram
\begin{equation*}
\begin{tikzcd}[sep=large]
    \mathrm{Gr}^{\lambda}_G \arrow[r,"j"] \arrow[rd,"j_{\lambda}",swap] & \overline{\mathrm{Gr}^{\lambda}_G} \arrow[d,"i_{\lambda}"] & \overline{\mathrm{Gr}^{\lambda}_G}  \setminus \mathrm{Gr}^{\lambda}_G \arrow[l,"i",swap] \arrow[ld] \\ 
    & \mathrm{Gr}_G & 
    \end{tikzcd}
\end{equation*}
There is a localization sequence
\begin{equation*}
\begin{split} 
    i_{\lambda!}j_!j^*i_{\lambda}^*\mathbf{IC}_{\lambda}(M) \longrightarrow  i_{\lambda!}i_{\lambda}^*\mathbf{IC}_{\lambda}(M) \longrightarrow  i_{\lambda!}i_*i^*i_{\lambda}^*\mathbf{IC}_{\lambda}(M) \longrightarrow +1
    \end{split} 
\end{equation*}
resulting in long exact sequences (by applying $\Hom_{\derivednori^b(\mathrm{Gr}_G,\lnor^+G)}((-), \mathbf{IC}_{\lambda}(N)[+1])$)
\begin{equation*}
    ... \longrightarrow \Hom(j_!j^*i_{\lambda}^*\mathbf{IC}_{\lambda}(M),i_{\lambda}^!\mathbf{IC}_{\lambda}(N)[+1]) \longrightarrow \Hom(i_{\lambda}^!\mathbf{IC}_{\lambda}(M),i_{\lambda}^!\mathbf{IC}_{\lambda}(N)[+1]) \longrightarrow \Hom(i_*i^*i_{\lambda}^*\mathbf{IC}_{\lambda}(M),i_{\lambda}^!\mathbf{IC}_{\lambda}(N)[+1]) \longrightarrow \cdots 
\end{equation*}
In this sequence, let us claim that
\begin{equation*}
    \begin{cases}
        \Hom(j_!j^*i_{\lambda}^*\mathbf{IC}_{\lambda}(M),i_{\lambda}^!\mathbf{IC}_{\lambda}(N)[+1])& = \Hom_{\derivednori^b(\mathrm{Gr}^{\lambda}_G)}(M,N[+1])  \ \ \ \ (1) \\ 
        \Hom(i_{\lambda}^!\mathbf{IC}_{\lambda}(M),i_{\lambda}^!\mathbf{IC}_{\lambda}(N)[+1]) & = \Hom(\mathbf{IC}_{\lambda}(M),\mathbf{IC}_{\lambda}(N)[+1]) \ \ \ \ (2) \\ 
        \Hom(i_*i^*i_{\lambda}^*\mathbf{IC}_{\lambda}(M),i_{\lambda}^!\mathbf{IC}_{\lambda}(N)[+1]) & = 0 \ \ \ \ (3) \\ 
        \Hom(i_*i^*i_{\lambda}^*\mathbf{IC}_{\lambda}(M),i_{\lambda}^!\mathbf{IC}_{\lambda}(N)[+2]) & = 0 \ \ \ \  (4)
    \end{cases}
\end{equation*}
Indeed, $(1),(2)$ are trivial by adjunctions and recollectements. Let us prove $(3),(4)$. or $\nu < \lambda$, we denote by $j_{\nu} \colon \mathrm{Gr}_G^{\nu} \longhookrightarrow \overline{\mathrm{Gr}}^{\lambda}_G \setminus \mathrm{Gr}_G^{\lambda}= \coprod_{\nu < \lambda} \mathrm{Gr}^{\nu}_G$ the obvious inclusion, we have 
\begin{equation*}
i^!i_{\lambda}^!(\mathbf{IC}_{\lambda}(M)) = \bigoplus_{\nu < \lambda} j_{\nu}^!i_{\lambda}^!(\mathbf{IC}_{\lambda}(M)) \ \ \ \ \text{and} \ \ \ \ 
i^*i_{\lambda}^*(\mathbf{IC}_{\lambda}(N)) = \bigoplus_{\nu < \lambda} j_{\nu}^*i_{\lambda}^*(\mathbf{IC}_{\lambda}(N)) 
\end{equation*}
For $(3)$, we have 
\begin{equation*}
    \Hom((i_{\lambda}i)^*\mathbf{IC}_{\lambda}(M),(i_{\lambda}i)^!\mathbf{IC}_{\lambda}(N)[+1]) = \bigoplus_{\nu < \lambda} \Hom(j_{\nu}^*i_{\lambda}^*(\mathbf{IC}_{\lambda}(M)), j_{\nu}^!i_{\lambda}^!(\mathbf{IC}_{\lambda}(N))[+1]).
\end{equation*}
By proposition \ref{prop: characterizations of intersection motives}, $j_{\nu}^*i_{\lambda}^*(\mathbf{IC}_{\lambda}(M))$ lives in perverse degree $\leq -\left<2\rho,\nu\right>-1$ and $j_{\nu}^!i_{\lambda}^!(\mathbf{IC}_{\lambda}(N))[+1]$ lives in perverse degree $\geq -\left<2 \rho, \nu\right>$ so the hom group must be zero. For $(4)$, we have the same argument but now the degree increases by one so let us show that $\phnor^{-\left<2\rho,\nu\right>-1}((j_{\nu}i_{\lambda})^*\mathbf{IC}_{\lambda}(M)) = 0$. Since $\mathrm{Gr}_G^{\nu}$ is smooth of dimension $\left<2\rho,\nu\right>$ one has that 
\begin{equation*}
\begin{split} 
    \phnor^{-\left<2\rho,\nu\right>-1}((j_{\nu}i_{\lambda})^*\mathbf{IC}_{\lambda}(M)) & = \cthnor^{-\left<2\rho,\nu\right>-1}((j_{\nu}i_{\lambda})^*\mathbf{IC}_{\lambda}(M))[\left<2\rho,\nu\right>] \\ 
    & = (j_{\nu}i_{\lambda})^*\cthnor^{-\left<2\rho,\nu\right>-1}(\mathbf{IC}_{\lambda}(M))[\left<2\rho,\nu\right>]
    \end{split} 
\end{equation*}
By Kazhdan-Lusztig parity vanishing \ref{thm: Kazhdan-Lusztig parity vanishing} and the fact that $\left <2\rho,\nu \right> \equiv \left<2\rho,\lambda \right> (\operatorname{mod} 2)$, this motive vanishes and hence we win this case.
\item \textbf{Second case neither $\lambda \leq \mu$ nor $\mu \leq \lambda$.} The proof is identical to the second case of \cite[Proof of proposition 4.4]{baumann+riche-2018}, in which there is no need to use Kazhdan-Lusztig parity vanishing. 
\item \textbf{Third case: $\lambda \neq \mu$ and $\lambda \leq \mu$ or $\mu \leq \lambda$}. Since Verdier duality is an anti-autoequivalence mapping $\mathbf{IC}_{\lambda}(M)$ to $\mathbf{IC}_{\lambda}(M^{\vee})(\left <2\rho,\lambda \right>)$ and $\mathbf{IC}_{\mu}(N)$ to $\mathbf{IC}_{\mu}(N^{\vee})(\left <2\rho,\mu \right>)$ (see lemma \ref{lem: Verdier dualities exchange motivic local systems}), we can, without loss of generality, assume that $\mu \leq \lambda$, i.e., $\mathrm{Gr}_G^{\mu} \subset \overline{\mathrm{Gr}_G^{\lambda}}$. Let us consider a triangle
\begin{equation*}
    \mathbf{IC}_{\mu}(N) \longrightarrow (i_{\lambda}j_{\mu})_*(i_{\lambda}j_{\mu})^*\mathbf{IC}_{\mu}(N) \longrightarrow P \longrightarrow +1.
\end{equation*}
Note that $(i_{\lambda}j_{\mu})_*(i_{\lambda}j_{\mu})^*\mathbf{IC}_{\mu}(N) \simeq (i_{\lambda}j_{\mu})_*(N[\left<2\rho,\mu\right>])$ and hence the triangle induces a short exact sequence 
\begin{equation*}
    \Hom(\mathbf{IC}_{\lambda}(M),P) \longrightarrow \Hom(\mathbf{IC}_{\lambda}(M),\mathbf{IC}_{\mu}(N)[+1]) \longrightarrow \Hom(\mathbf{IC}_{\lambda}(M),(i_{\lambda}j_{\mu})_*(N[\left<2\rho,\mu\right>+1]).
\end{equation*}
The first group is zero without using Kazhdan-Lusztig parity vanishing. The third group 
\begin{equation*}
    \Hom(\mathbf{IC}_{\lambda}(M),(i_{\lambda}j_{\mu})_*(N[\left<2\rho,\mu\right>+1]) = \Hom((i_{\lambda}j_{\mu})^*\mathbf{IC}_{\lambda}(M),N[\left<2\rho,\mu\right>+1])
\end{equation*}
is also zero because in general $(i_{\lambda}j_{\mu})^*\mathbf{IC}_{\lambda}(M)$ lives in constructible degree $\leq -\left<2\rho,\mu\right>-1$ but by Kazhdan-Lusztig parity vanishing \ref{thm: Kazhdan-Lusztig parity vanishing} and the fact that $\left <2\rho,\mu \right> \equiv \left<2\rho,\lambda \right> (\operatorname{mod} 2)$, it actually lives in constructible degrees $\leq -\left<2\rho,\mu\right>-2$. 
\end{enumerate} 
\end{proof}

We note that in a Noetherian and Artinian category, any object is a direct sum of indecomposable objects.  Moreover, if the category itself is semisimple, then being indecomposable is equivalent to being simple. 

\begin{cor} \label{cor: indecomposable objects}
Let $M \in \mscr\perv(\mathrm{Gr}_G,\lnor^+G)$, then there exist $\mu_i \in X_*(T)^+$ and indecomposable motives $L_i \in \mscr\localsystem(\mathrm{Gr}_G^{\mu_i})$ with $i = \overline{1,n}$ such that $M \simeq \oplus_{i=1}^n \mathbf{IC}_{\mu_i}(L_i)$. An object $\mathbf{IC}_{\lambda}(L) \in \mscr\perv(\mathrm{Gr}_G,\lnor^+G)$ is simple if and only if $L \in \mscr\localsystem(\mathrm{Gr}_G^{\lambda})$ is simple. 
\end{cor}

\begin{proof}
By the discussion above, any object $M$ is decomposed into indecomposable objects so we may assume that $M$ itself is indecomposable. We proceed by induction on the leng $l(M)$ of $M$. If $l(M)=1$, then $M$ is simple and hence $M = \mathbf{IC}_{\lambda}(L)$ for some $\lambda \in X_*(T)^+$ and $L \in \mscr\localsystem(\mathrm{Gr}^{\lambda}_G,n)$ by theorem \ref{thm: simple objects of perverse motives on ind-schemes}. It is clear that $L$ must be indecomposable as $M$ is, but $\mscr\perv(\mathrm{Gr}^{\lambda}_G,n)$ is semisimple by \ref{lem: weight filtrations of motivic local systems}\cite[Theorem 6.24]{florian+morel-2019}, being indecomposible is equivalent to being simple. Thus, $M$ is of the desired form. If $l(M) \geq 2$, let $M' \subset M$ be a subobject with $l(M')=1$ and hence simple, by induction $M' \simeq \mathbf{IC}_{\mu_1}(L_1)$ and $M/M' =  \mathbf{IC}_{\mu_2}(L_2)$ (this quotient is also indecomposable). Now by virtues of proposition \ref{prop: a corollary of parity vanishing}
\begin{equation*}
    \Hom_{\derivednori^b(\mathrm{Gr}_G,\lnor^+G)}\left(\mathbf{IC}_{\mu_1}(L_1),\mathbf{IC}_{\mu_2}(L_2)[+1] \right) = \operatorname{Ext}^1_{\mscr\perv(\mathrm{Gr}_G^{\mu})}(L_1,L_2) 
\end{equation*}
which contains the class of $M$, so $\mu_1 = \mu_2$ and $M$ is an intersection motive of a motivic local system which is an extension of $L_1[\left<2\rho,\mu \right>],L_2[\left<2\rho,\mu \right>]$ in $\mscr\perv(\mathrm{Gr}_G^{\mu})$ (since shifted motivic local systems form a Serre subcategory of motivic perverse sheaves). 
\end{proof}
For $n \in \mathbb{Z}$, we define $\mscr\perv_{\lnor^+G}(\mathrm{Gr}_G,\lnor^+G,n)$ to be the full subcategory whose objects are $\mscr\perv_{\lnor^+G}(\mathrm{Gr}_G,\lnor^+G) \cap \mscr\perv(\mathrm{Gr}_G,\lnor^+G,n)$ and we define $\mscr\perv_{\lnor^+G}(\mathrm{Gr}_G,\lnor^+G,\pure) = \bigoplus_{n \in \mathbb{Z}} \mscr\perv_{\lnor^+G}(\mathrm{Gr}_G,\lnor^+G,n)$.

\begin{cor} \label{cor: semisimplicity of perverse motives on aff grass}
The categories $\mscr\perv(\mathrm{Gr}_G,\lnor^+G,n)$ (with $n\in \mathbb{Z}$) and $\mscr\perv(\mathrm{Gr}_G,\lnor^+G,\textnormal{pure})$ are semisimple.
\end{cor}

\begin{proof}
   By lemma above, each object $M \in \mscr\perv(\mathrm{Gr}_G,\lnor^+G,n)$ is of the form $\oplus_{i=1}^n \mathbf{IC}_{\mu_i}(L_i)$ with $L_i \in \mscr\localsystem(\mathrm{Gr}_G^{\mu_i},n)$ being indecomposable. Since each $\mscr\localsystem(\mathrm{Gr}_G^{\mu_i},n) \subset \mscr\perv(\mathrm{Gr}_G^{\mu_i},n)$ is a Serre subcategory and the latter one is semisimple by \cite[Theorem 6.24]{florian+morel-2019}, the objects $L_i$ are necessarily simple in $\mscr\localsystem(\mathrm{Gr}_G^{\mu_i},n)$. Each object $\mathbf{IC}_{\mu_i}(L_i)$ must be simple as well, hence $\mscr\perv(\mathrm{Gr}_G,\lnor^+G,n)$ is semisimple. In particular, $\mscr\perv(\mathrm{Gr}_G,\lnor^+G,\textnormal{pure}) = \oplus_{n \in \mathbb{Z}}\mscr\perv(\mathrm{Gr}_G,\lnor^+G,n)$ is semisimple as well.
\end{proof}

\begin{cor}
A motive $\mscr\perv(\mathrm{Gr}_G,\lnor^+G)$ is in $\mscr\textnormal{Tate}(\mathrm{Gr}_G,\lnor^+G)$ if and only if it is of the form $\oplus_{i=1}^n \mathbf{IC}_{\mu_i}(L_i)$ with $L_i \in \mscr\textnormal{Tate}(\mathrm{Gr}_G^{\mu_i})$. 
\end{cor}

\begin{proof}
  If each $L_i$ is a Tate motive, then so is $\oplus_{i=1}^n\mathbf{IC}_{\mu_i}(L_i)$ thanks to lemma \ref{lem: intersection motives of Tate motives are Tate}. Conversely, if the direct sum is a Tate motive then by lemma \ref{lem: characterizations of stratified Tate motives}, each direct summand $\mathbf{IC}_{\mu_i}(L_i)$ is a Tate motive and hence by definition $L_i[\left<2\rho,\mu_i\right>] = \mathbf{IC}_{\mu_i}(L_i)_{\mid \mathrm{Gr}_G^{\mu_i}}$ is a shifted Tate motive. 
\end{proof}

\begin{cor}
Let $M \in \mscr\perv_{\lnor^+G}(\mathrm{Gr}_G,\lnor^+G)$, then there exist $\mu_i \in X_*(T)^+$ and indecomposable motives $L_i \in \mscr\localsystem_{\lnor^+G}(\mathrm{Gr}_G^{\mu_i})$ with $i = \overline{1,n}$ such that $M \simeq \oplus_{i=1}^n \mathbf{IC}_{\mu_i}^{\lnor^+G}(L_i)$. An object $\mathbf{IC}_{\lambda}^{\lnor^+G}(L) \in \mscr\perv_{\lnor^+G}(\mathrm{Gr}_G,\lnor^+G)$ is simple if and only if $L \in \mscr\localsystem_{\lnor^+G}(\mathrm{Gr}_G^{\lambda})$ is simple. 
\end{cor}

\begin{proof}
 This is a consequence of proposition \ref{prop: characterizations of equivariant intersection motives}. We decompose $\res^{\lnor^+G}(M) = \oplus_{i =1}^n \mathbf{IC}_{\mu_i}(L_i)^{\oplus n_i}$ in $\mscr\perv(\mathrm{Gr}_G,\lnor^+G)$ with $\mu_i$ pairwise distinct, $L_i \in \mscr\localsystem(\mathrm{Gr}^{\mu_i}_G)$, $n_i \in \mathbb{N}$. By looking at supports and use the smooth base change, we see that each $L_i$ carries an equivariant structure.
\end{proof}

\subsection{Convolution product of affine Grassmannians}  
Thanks to the previous section, there is an associative convolution product  (the derived level) 
\begin{equation*}
    (-) \star (-) \colon \derivednori^b_{\lnor^+G}(\mathrm{Gr}_G,\lnor^+G) \times \derivednori^b_{\lnor^+G}(\mathrm{Gr}_G,\lnor^+G)\longrightarrow \derivednori^b_{\lnor^+G}(\mathrm{Gr}_G,\lnor^+G).
\end{equation*}
We want to define convolution products on $\mscr\perv_{\lnor^+G}(\mathrm{Gr}_G,\lnor^+G$ and $\mscr\operatorname{Tate}_{\lnor^+G,p}(\mathrm{Gr},\lnor^+G)$. 
\begin{prop} \label{prop: convolution products restrict to subcategories of motives}
The convolution product 
\begin{equation*} 
(-) \star (-) \colon \derivednori^b_{\lnor^+G}(\mathrm{Gr}_G,\lnor^+G) \times \derivednori^b_{\lnor^+G}(\mathrm{Gr}_G,\lnor^+G)\longrightarrow \derivednori^b_{\lnor^+G}(\mathrm{Gr}_G,\lnor^+G)
\end{equation*} 
is $t$-exact in both variables. In particular, if $M_1,M_2 \in \mscr\perv_{\lnor^+G}(\mathrm{Gr}_G,\lnor^+G)$ then 
\begin{equation*}
    \phnor^n_{\lnor^+G}(M_1 \star M_2) = 0 \ \forall \ n \neq 0
\end{equation*}
and $M_1 \star M_2 = \phnor^0_{\lnor^+G}(M_1 \star M_2) = m_*\phnor^0_{\lnor^+G}(M_1 \widetilde{\boxtimes} M_2)$ defines a convolution product
\begin{align*}
(-) \star (-) \colon \mscr\perv_{\lnor^+G}(\mathrm{Gr}_G,\lnor^+G) \times \mscr\perv_{\lnor^+G}(\mathrm{Gr}_G,\lnor^+G) \longrightarrow \mscr\perv_{\lnor^+G}(\mathrm{Gr}_G,\lnor^+G)
\end{align*}
which is exact in both variables. 
\end{prop}

\begin{proof}
   The convolution product is $t$-exact because $m_*,q^*,p^*,\boxtimes$ are $t$-exact (for $m_*$, see proposition \ref{prop: pushforwards of equivariant stratified motives}). Concern the vanishing of cohomology, one uses the fact that Betti realization is conservative and commutes with cohomology (since it is exact), the vanishing is then translated to \cite[Proposition 6.1]{baumann+riche-2018} where one can investigate the vanishing of pushforwards of proper, semismall, locally trivial maps of analytic varieties. 
\end{proof}

\begin{cor} \label{cor: convolution products preserve Tate motives}
There is a restricted convolution product
\begin{equation*}
    (-) \star (-) \colon \mscr\operatorname{Tate}_{\lnor^+G,p}(\mathrm{Gr}_G,\lnor^+G) \times \mscr\operatorname{Tate}_{\lnor^+G,p}(\mathrm{Gr}_G,\lnor^+G) \longrightarrow \mscr\operatorname{Tate}_{\lnor^+G,p}(\mathrm{Gr}_G,\lnor^+G)
\end{equation*}
on Tate motives. 
\end{cor}

\begin{proof}
This is a consequence of propositions \ref{prop: convolution products preserve derived Tate motives} and \ref{prop: convolution products restrict to subcategories of motives}.
\end{proof}

\begin{prop} \label{prop: convolution products are weight exact} 
The convolution product 
\begin{equation*} 
(-) \star (-) \colon \derivednori^b_{\lnor^+G}(\mathrm{Gr}_G,\lnor^+G) \times \derivednori^b_{\lnor^+G}(\mathrm{Gr}_G,\lnor^+G) \longrightarrow \derivednori^b_{\lnor^+G}(\mathrm{Gr}_G,\lnor^+G)
\end{equation*} 
is weight-exact, i.e., if $A,B \in \derivednori^b_{\lnor^+G}(\mathrm{Gr}_G,\lnor^+G)^{w \leq 0}$, then $A \star B \in \derivednori^b_{\lnor^+G}(\mathrm{Gr}_G,\lnor^+G)^{w \leq 0}$ and likewise for $w \geq 0$. 
\end{prop}
\begin{proof}
  The box product $(-) \boxtimes (-)$ is weight-exact on schemes. The morphism $p \colon \lnor G \times \mathrm{Gr}_G \longrightarrow \mathrm{Gr}_G \times \mathrm{Gr}_G$ is smooth and hence weight exact and $q^*$ is a smooth equivalence, hence respects as well. Thus $(-) \widetilde{\boxtimes} (-)$ has the same weight as $(-) \boxtimes (-)$. 
\end{proof}

\begin{prop} \label{prop: Soegel's bimodules}
The composition
\begin{equation*} 
    \derivednori^b_{\lnor^+G}(\mathrm{Gr}_G,\lnor^+G) \longrightarrow \derivednori^b(\mathrm{Gr}_G,\lnor^+G) \overset{\epsilon_!}{\longrightarrow} \derivednori^b(k) 
\end{equation*}
is a monoidal functor when $\derivednori^b_{\lnor^+G}(\mathrm{Gr}_G,\lnor^+G)$ is endowed with the convolution product and $\derivednori^b(k)$ is endowed with the tensor product. 
\end{prop}

\begin{proof}
   The proof is similar to \cite[Proposition 5.13]{richarz+scholbach-2021} (see also \cite[Lemma 5.2.3]{zhu-2016}\cite[Lemma 2.18]{zhu-2017}). Indeed, this is \cite[Proposition 5.13]{richarz+scholbach-2021} by choosing $\mathbf{f}= 0$ to be the trivial facet and replacing Tate motives by Nori motives. 
\end{proof}

\begin{lem}
Let $e \colon \Spec(F) \longrightarrow \Spec(k)$ be a field extension of subfields of $\mathbb{C}$, then the convolution product commutes with $e^*$, namely, there is a commutative diagram
\begin{equation*}
    \begin{tikzcd}[sep=large]
      \derivednori^b_{\lnor^+G}(\mathrm{Gr}_G,\lnor^+G) \times   \derivednori^b_{\lnor^+G}(\mathrm{Gr}_G,\lnor^+G) \arrow[r,"(-) \star (-)"] \arrow[d,"(e^*\textnormal{,}e^*)",swap] &    \derivednori^b_{\lnor^+G}(\mathrm{Gr}_G,\lnor^+G) \arrow[d,"e^*"] \\ 
        \derivednori^b_{\lnor^+G_F}(\mathrm{Gr}_{G_F},\lnor^+G_F) \times    \derivednori^b_{\lnor^+G_F}(\mathrm{Gr}_{G_F},\lnor^+G_F)\arrow[r,"(-) \star (-)"] &    \derivednori^b_{\lnor^+G_F}(\mathrm{Gr}_{G_F},\lnor^+G_F).
    \end{tikzcd}
\end{equation*}
\end{lem}

\begin{proof}
  This is clear because convolution products are defined in terms pullbacks of smooth morphisms, pushforwards of ind-proper morphisms, box products and $e^*$ hence commutes with all of them. 
\end{proof}

\subsection{Fusion and commutativity of convolution product}
We would like to show that the convolution product is commutative. Suppose we are given a natural morphism $M_1 \star M_2 \longrightarrow M_2 \star M_1$ so that under the Betti realization, this becomes the natural commutativity constraint of the complex convolution product then we win thanks to the conservativity. As in the analytic and $\ell$-cases, this is possible. Let us recall the definition of the Beilinson–Drinfeld Grassmannian. Let $X = \mathbb{A}^1_k$, we define the following Grassmannians via their moduli interpretations
\begin{equation*}
    \begin{split}
        \mathrm{Gr}_{G,X} \colon \Alg_k & \longrightarrow \sets \\ 
        R & \longmapsto \left\{ (\fcal,\nu,x)\;\left\lvert\;
\begin{array}{l}
x \in X(R) \\
\fcal \colon G\text{-bundle on} \ X_R \\
\nu \colon \fcal_{\mid X_R \setminus x} \simeq G \times (X_R \setminus x)
\end{array}
\right.
\right\} / \textnormal{iso}
    \end{split}
\end{equation*}
\begin{equation*}
    \begin{split}
        \mathrm{Gr}_{G,X^2} \colon \Alg_k & \longrightarrow \sets \\ 
        R & \longmapsto \left\{ (\fcal,\nu,x_1,x_2)\;\left\lvert\;
\begin{array}{l}
x_1,x_2 \in X(R) \\
\fcal \colon G\text{-bundle on} \ X_R \\
\nu \colon \fcal_{\mid X_R \setminus (x_1 \cup x_2)} \simeq G \times (X_R \setminus (x_1 \cup x_2))
\end{array}
\right.
\right\} / \textnormal{iso}
    \end{split}
\end{equation*}
(for more details and other moduli interpretations, one can consult \cite[Section 7.4, Chapter I]{baumann+riche-2018}). There is a canonical projection $\mathrm{Gr}_{G,X^2} \longrightarrow X^2$ and a diagonal morphism $\Delta \colon X \longrightarrow X^2$. Now it is clear from the moduli interpretations that
\begin{align*}
    \mathrm{Gr}_{G,X^2} \times_{X^2} \Delta(X) & \simeq \mathrm{Gr}_{G,X} \\ 
    \mathrm{Gr}_{G,X^2} \times_{X^2} (X^2 \setminus \Delta(X)) & \simeq (\mathrm{Gr}_{G,X} \times \mathrm{Gr}_{G,X} \times_{X^2} (X^2 \setminus \Delta(X)). 
\end{align*}
Let us denote by $i \colon \mathrm{Gr}_{G,X}  \longrightarrow \mathrm{Gr}_{G,X^2}$ and $\tau \colon \mathrm{Gr}_{G,X} = \mathrm{Gr}_G \times X  \longrightarrow \mathrm{Gr}_X$ the corresonding embedding and projection, respectively. 
\begin{prop} \label{prop: commutativity constraints}
There is a canonical isomorphism
\begin{equation*}
    M_1 \star M_2 \simeq M_2 \star M_1
\end{equation*}
for $M_1,M_2 \in \mscr\perv_{\lnor^+G}(\mathrm{Gr}_G,\lnor^+G)$. 
\end{prop}
\begin{proof}
  First we see that, there are canonical isomorphisms
\begin{equation*}
    i^*[-1](\tau^{\circ}(M_1) \star \tau^{\circ}(M_2)) \simeq \tau^{\circ}(M_1 \star M_2) \simeq i^![1](\tau^{\circ}(M_1) \star \tau^{\circ}(M_2)).
\end{equation*}
   The morphisms among three terms are constructed canonically, so they are compatible with realizations and hence we can use \cite[Lemma 7.8]{baumann+riche-2018} and the conservativity of the Betti realization. Alternatively, one can copy the argument in $\loccit$. Second, we claim that \begin{equation*}
    j_{!*}((\tau^{\circ}(M_1) \boxtimes \tau^{\circ}(M_2))_{\mid U}) \simeq \tau^{\circ}(M_1) \star \tau^{\circ}(M_2)
\end{equation*}
with $U = X^2 \setminus \Delta(X)$. Indeed, this is similar to \cite[Lemma 7.10, Chapter I]{baumann+riche-2018} or otherwise one can use the conservativity of the Betti realization again.  Finally, to deduce the commutativity constraint, one repeats \cite[Section 7.6, Chapter I]{baumann+riche-2018}.
\end{proof}

\subsection{The unit object}  In this section, we will see how the intersection motives decompose and prove that $\mathbf{IC}_0^{\lnor^+G}(\mathds{1})$ is the unit of the convolution product. 

\begin{prop} \label{prop: unit object} 
Let $\mu \in X_*(T)^+$ and $p_{\mu} \colon \mathrm{Gr}^{\mu}_G \longrightarrow \Spec(k)$ be the structural morphism. Let $L \in \mscr\localsystem(k) = \mscr\localsystem_{\lnor^+G}(k)$, there is a canonical isomorphism 
\begin{equation*}
    \inters_0^{\lnor^+G}(L) \star \inters_{\mu}^{\lnor^+G}(\mathds{1}) \simeq \inters_{\mu}^{\lnor^+G}(p_{\mu}^*(L))
\end{equation*}
More generally, let $L' \in \mscr\localsystem_{\lnor^+G}(\mathrm{Gr}^{\mu}_G)$, then there is a canonical isomorphism
\begin{equation*}
    \mathbf{IC}_0^{\lnor^+G}(L) \star \mathbf{IC}_{\mu}^{\lnor^+G}(L') \simeq \mathbf{IC}_{\mu}^{\lnor^+G}(p_{\mu}^*(L) \otimes L').
\end{equation*}
\end{prop}
\begin{proof} 
The proof is as same as the proof is of \cite[Lemma 5.4]{richarz+scholbach-2021}, in the diagram
\begin{equation*}
    \begin{tikzcd}[sep=large] 
             \mathrm{Gr}_G \times \mathrm{Gr}_G & \lnor G \times \mathrm{Gr}_G \arrow[r,"q"] \arrow[l,"p",swap] & \lnor G \times^{\lnor^+G} \mathrm{Gr}_G \arrow[r,"m"] & \mathrm{Gr}_G  
    \end{tikzcd}
\end{equation*}
we have that
\begin{equation*}
    \begin{split}
        \mathbf{IC}_0^{\lnor^+G}(L) \star \mathbf{IC}^{\lnor^+G}_{\mu}(\mathds{1}) & = m_*(q^*)^{-1}p^*( \mathbf{IC}_0^{\lnor^+G}(L) \boxtimes \mathbf{IC}^{\lnor^+G}_{\mu}(\mathds{1}))  \\ 
       % & = m_*q_*p^*(\mathbf{IC}_{\mu}^{\lnor^+G}(\mathrm{Gr}^{\leq \mu},p_{\mu}^*(L)) \\ 
        & \simeq p_{\mu}^*(L) \otimes \mathbf{IC}_{\mu}^{\lnor^+G}(\mathds{1}) \\ 
        & \simeq \mathbf{IC}_{\mu}^{\lnor^+G}(p_{\mu}^*(L)), 
    \end{split}
\end{equation*}
where the second isomophism follows from the fact that $m,q$ are isomorphisms over $qp^{-1}(\mathrm{Gr}_G^{\leq 0} \times \mathrm{Gr}_G^{\leq \mu})$ and the last isomorphism follows from proposition \ref{prop: characterizations of equivariant intersection motives}. 
 \end{proof} 

\begin{lem} \label{lem: triangular decomposition of intersection motives} 
Let $\mu,\lambda \in X_*(T)$, then there is a decomposition 
\begin{equation*}
    \mathbf{IC}_{\mu}^{\lnor^+G}(\mathds{1}) \star \mathbf{IC}_{\lambda}^{\lnor^+G}(\mathds{1}) = \mathbf{IC}_{\mu+\lambda}^{\lnor^+G}(\mathds{1}) \oplus \bigoplus_{\nu < \mu + \lambda} \mathbf{IC}_{\nu}^{\lnor^+G}(\mathds{1}(-\left<\rho,\mu+\lambda - \nu \right>))^{\oplus n_{\nu}}
\end{equation*}
for integers $n_{\mu} \in \mathbb{N}$. In particular, if $\mathrm{Gr}^{\mu}_G,\mathrm{Gr}^{\lambda}_G$ are singletons, then
\begin{equation*}
    \mathbf{IC}_{\mu}^{\lnor^+G}(\mathds{1}) \star \mathbf{IC}_{\lambda}^{\lnor^+G}(\mathds{1}) = \mathbf{IC}_{\mu+\lambda}^{\lnor^+G}(\mathds{1})
\end{equation*}
\end{lem}

\begin{proof}
   Note that by general theory (see corollary \ref{cor: indecomposable objects}), there exist indecomposable motives $L_i$ such that \begin{equation*} 
   \mathbf{IC}_{\mu}^{\lnor^+G}(\mathds{1}) \star \mathbf{IC}_{\lambda}^{\lnor^+G}(\mathds{1}) = \oplus_{\nu \in X_*(T)^+}\mathbf{IC}_{\nu}^{\lnor^+G}(L_i)^{\oplus n_{\nu}}
   \end{equation*} 
   (a finite sum). Under the Betti realization, only those $\nu \leq \mu + \lambda$ survive and $n_{\mu + \lambda} = 1$. The motives $\mathbf{IC}_{\mu}^{\lnor^+G}(\mathds{1}), \mathbf{IC}_{\lambda}^{\lnor^+G}(\mathds{1})$ are pure of weights $\left<2\rho,\mu\right>,\left<2\rho,\lambda\right>$, respectively, thanks to proposition \ref{prop: convolution products are weight exact}. Consequently, by proposition \ref{prop: convolution products restrict to subcategories of motives}, $\mathbf{IC}_{\mu}^{\lnor^+G}(\mathds{1}) \star \mathbf{IC}_{\lambda}^{\lnor^+G}(\mathds{1})$ is pure of weight $\left<2\rho,\mu+\lambda \right>$ and is a shifted Tate motive by corollary \ref{cor: convolution products preserve Tate motives}. Hence $\mathbf{IC}_{\nu}^{\lnor^+G}(L_{\nu})$ and $\mathbf{IC}_{\mu}(\res^{\lnor^+G}(L_{\nu}))$ are also pure of the same weight and they are supported on $\mathrm{Gr}^{\leq \nu}_G$. Consequently, $L_{\nu}[\left<2\rho,\nu\right>]$ itself is also a shifted Tate motive pure of weight $\left <2\rho,\mu+\lambda \right>$, hence is of the form $\mathds{1}(\left <\rho,\mu+\lambda\right>)$ by virtues of lemma \ref{lem: pure Tate motives}. 
\end{proof}

\section{The Motivic Satake Equivalence}

\subsection{Small motives}

Unfortunately, unlike the case of ordinary perverse sheaves or mixed Tate motives, we \textit{cannot} hope that the category $\mscr\perv_{\lnor^+G}(\mathrm{Gr},\lnor^+G)$ to be a neutral Tannakian category unless we have a theory of \textit{local systems of simply connected spaces}. In the case of ordinary perverse sheaves. 
\begin{defn}
  Let $\scal \subset \mscr\localsystem(k)$ be a Tannakian subcategory containing all Tate twists $\mathds{1}(n)$. The \textit{$\scal$-motivic Satake category} $\mscr\perv^{\scal}_{\lnor^+G} (\mathrm{Gr}_G,\lnor^+G)$ is the smallest full abelian subcategory of $\mscr\perv_{\lnor^+G}(\mathrm{Gr}_G,\lnor^+G)$ so that:
  \begin{enumerate}
      \item It is stable under the convolution product. 
      \item It is stable under direct sums.
      \item It contains all objects of the form $ \mathbf{IC}^{\lnor^+G}_{\mu}(p_{\mu}^*(L))$ for equivariant motivic local systems $L \in \scal$, where $\mu \in X_*(T)^+$ and $p_{\mu}^* \colon \mscr\localsystem(k) = \mscr\localsystem_{\lnor^+G}(k) \longrightarrow \mscr\localsystem_{\lnor^+G}(\mathrm{Gr}^{\mu})$. 
  \end{enumerate}
 Spell out the definition, every object in $\mscr\perv^{\scal}_{\lnor^+G} (\mathrm{Gr}_G,\lnor^+G)$ admits a filtration whose graded pieces are subquotients of objects of the form 
 \begin{equation*}
     \bigoplus_{\mu_{1},...,\mu_{n} \in X_*(T)^+}  \mathbf{IC}^{\lnor^+G}_{\mu_{1}}(p_{\mu_{1}}^*(L_{1})) \star \cdots \star \mathbf{IC}^{\lnor^+G}_{\mu_{n}}(p_{\mu_{n}}^*(L_{n}))
 \end{equation*}
for equivariant motivic local systems $L_{i} \in \scal$ (one can take $L_{i_j}$'s to be indecomposable motives). The convolution product on $\derivednori^b_{\lnor^+G}(\mathrm{Gr}_G,\lnor^+G)$ restricts to convolution products 
\begin{equation*} 
(-) \star (-) \colon \mscr\perv^{\scal}_{\lnor^+G}(\mathrm{Gr}_G,\lnor^+G) \times \mscr\perv^{\scal}_{\lnor^+G}(\mathrm{Gr}_G,\lnor^+G) \longrightarrow \mscr\perv^{\scal}_{\lnor^+G}(\mathrm{Gr}_G,\lnor^+G). 
\end{equation*} 
If $\scal_1 \subset \scal_2 \subset \mscr\localsystem(k)$, then there is unitary monoidal inclusion 
\begin{equation*}
    \mscr\perv_{\lnor^+G}^{\scal_1}(\mathrm{Gr}_G,\lnor^+G) \longhookrightarrow  \mscr\perv_{\lnor^+G}^{\scal_2}(\mathrm{Gr}_G,\lnor^+G).
\end{equation*}
In particular, there is a unitary monoidal inclusion
\begin{equation*}
    \mscr\perv_{\lnor^+G}^{\mscr\operatorname{Tate}_p(k)}(\mathrm{Gr}_G,\lnor^+G) \longhookrightarrow \mscr\perv^{\mscr\localsystem(k)}_{\lnor^+G}(\mathrm{Gr}_G,\lnor^+G). 
\end{equation*}
There is a canonical embedding $ \mscr\perv_{\lnor^+G}^{\mscr\operatorname{Tate}_p(k)}(\mathrm{Gr}_G,\lnor^+G) \longhookrightarrow \mscr\operatorname{Tate}_{\lnor^+G,p}(\mathrm{Gr}_G,\lnor^+G)$ since each $\mathbf{IC}^{\lnor^+G}_{\mu}(p_{\mu}^*(L))$ is an equivariant Tate motive. We do not know whether this is an equivalence or not. 
\end{defn}

\begin{cor}
The object $\inters_0^{\lnor^+G}(\mathds{1})$ is the unit object of the convolution product on $\mscr\perv_{\lnor^+G}^{\scal}(\mathrm{Gr}_G,\lnor^+G)$. 
\end{cor}

\begin{proof}
This is obvious thanks to proposition \ref{prop: unit object}.   
\end{proof}

\begin{lem}
Let $\mu_1,...,\mu_n \in X_*(T)^+$, then there is a decomposition of the form
\begin{equation*}
    \mathbf{IC}^{\lnor^+G}_{\mu_{1}}(p_{\mu_{1}}^*(L_{1})) \star \cdots \star \mathbf{IC}^{\lnor^+G}_{\mu_{n}}(p_{\mu_{n}}^*(L_{n})) = \bigoplus_{\lambda \leq \left|\mu_{\bullet} \right|}  \mathbf{IC}^{\lnor^+G}_{\lambda}(p_{\lambda}^*(L_1 \otimes \cdots \otimes L_n)(-\left<\rho,\left|\mu_{\bullet} \right|-\lambda \right>))^{n_{\lambda}}.
\end{equation*}
\end{lem}

\begin{proof}
  It suffices to show the case $n=2$ and then do an induction. We have
  \begin{align*}
       \mathbf{IC}^{\lnor^+G}_{\mu_1}(p_{\mu_1}^*(L_1)) \star  \mathbf{IC}^{\lnor^+G}_{\mu_2}(p_{\mu_2}^*(L_2)) & = \mathbf{IC}_0^{\lnor^+G}(L_1) \star \mathbf{IC}_{\mu_1}^{\lnor^+G}(\mathds{1}) \star \mathbf{IC}_0^{\lnor^+G}(L_2) \star \mathbf{IC}_{\mu_2}^{\lnor^+G}(\mathds{1})  \\ 
       & \simeq \mathbf{IC}_0^{\lnor^+G}(L_1) \star \mathbf{IC}_0^{\lnor^+G}(L_2) \star \mathbf{IC}_{\mu_1}^{\lnor^+G}(\mathds{1}) \star \mathbf{IC}_{\mu_2}^{\lnor^+G}(\mathds{1}) \\ 
       & \simeq  \mathbf{IC}_0^{\lnor^+G}(L_1) \star \mathbf{IC}_0^{\lnor^+G}(L_2) \star \left( \bigoplus_{\lambda \leq \mu_1 + \mu_2}\mathbf{IC}_{\lambda}^{\lnor^+G}(\mathds{1}(\left< \rho, \mu_1 + \mu_2 \right>))^{\oplus n_{\lambda}} \right) \\ 
       & \simeq \bigoplus_{\lambda \leq \mu_1+\mu_2} \mathbf{IC}_0^{\lnor^+G}(L_1 \otimes L_2) \star \mathbf{IC}_{\lambda}^{\lnor^+G}(\mathds{1}(\left< \rho, \mu_1 + \mu_2 - \lambda \right>))^{\oplus n_{\lambda}} \\ 
       & \simeq \bigoplus_{\lambda \leq \mu_1+\mu_2} \mathbf{IC}_0^{\lnor^+G}(L_1 \otimes L_2) \star \mathbf{IC}_{\lambda}^{\lnor^+G}(\mathds{1}(\left< \rho, \mu_1 + \mu_2 - \nu \right>))^{\oplus n_{\lambda}} \\ 
       & \simeq \bigoplus_{\lambda \leq \mu_1+\mu_2}  \mathbf{IC}_{\lambda}^{\lnor^+G}(p_{\lambda}^*(L_1 \otimes L_2)(\left< \rho, \mu_1 + \mu_2 - \nu \right>))^{\oplus n_{\lambda}}
  \end{align*}
  as desired. 
\end{proof}

\begin{cor} \label{cor: generators of small motives} 
The category $\mscr\perv_{\lnor^+G}^{\scal}(\mathrm{Gr}_G,\lnor^+G)$ is the smallest full subcategory of $\mscr\perv_{\lnor^+G}(\mathrm{Gr}_G,\lnor^+G)$ that is stable under convolution product, direct sums and contains all motives of the form $\mathbf{IC}_{\mu}^{\lnor^+G}(p_{\mu}^*(L))$ with $\mu \in X_*(T)^+,L \in \scal$. 
\end{cor}

%\begin{prop}
%For each $\mu \in X_*(T)^+$, there exists a {\color{red}unitary monoidal functor}
%\begin{equation*}
 %   \mscr\localsystem(k) \overset{p_{\mu}^*}{\longrightarrow} \mscr\localsystem_{\lnor^+G}(\mathrm{Gr}_G^{\mu}) \overset{\mathbf{IC}^{\lnor^+G}_{\mu}}{\longrightarrow} \mscr\perv_{\lnor^+G}^{\scal}(\mathrm{Gr}_G,\lnor^+G).
%\end{equation*}
%The family $\left \{\mathbf{IC}_{\mu}^{\lnor^+G} \circ p_{\mu}^* \right \}_{\mu \in X_*(T)^+}$ forms a conservative family. The functor $\mathbf{IC}_0^{\lnor^+G} \circ p_0^*$ is a unitary monoidal functor. 
%\end{prop}

\begin{prop}
Let $e \colon \Spec(F) \longrightarrow \Spec(k)$ be a field extension of subfields of $\mathbb{C}$. Let $\scal \subset \mscr\localsystem(k)$ be a Tannakian subcategory containing all Tate twists. There is a unital monoidal functor 
\begin{equation*}
     \mscr\perv^{\scal}_{\lnor^+G_k}(\mathrm{Gr}_{G_k},\lnor^+G_k)  \longrightarrow \mscr\perv^{e^*(\scal)}_{\lnor^+G_F}(\mathrm{Gr}_{G_F},\lnor^+G_F).
\end{equation*}
If $e$ is an \'etale morphism, then there is a canonical equivalence of categories
\begin{equation*}
    \mscr\perv^{\scal}_{\lnor^+G_k}(\mathrm{Gr}_{G_k},\lnor^+G_k) \simeq \mscr\perv^{e^*(\scal)}_{\lnor^+G_F}(\mathrm{Gr}_{G_F},\lnor^+G_F)^{\Gal(F/k)}.
\end{equation*}
\end{prop}

\begin{proof}
  By proposition \ref{prop: galois descent of equivariant, stratified motives}, there is an exact functor 
 \begin{equation*}
    e^* \colon \mscr\perv_{\lnor^+G_k}(\mathrm{Gr}_{G_k},\lnor^+G_k)  \longrightarrow \mscr\perv_{\lnor^+G_F}(\mathrm{Gr}_{G_F},\lnor^+G_F).
\end{equation*}
By corollary, it suffices to check that the images of motives of the form $\mathbf{IC}_{\mu}^{\lnor^+G}(p_{\mu}^*(L))$ with $L \in \scal$ lie in $\mscr\perv^{e^*(\scal)}_{\lnor^+G_F}(\mathrm{Gr}_{G_F},\lnor^+G_F)^{\Gal(F/k)}$ but this is obvious. Regarding Galois descent, it is enough to show that
 \begin{equation*}
    e_* \colon  \mscr\perv_{\lnor^+G_F}(\mathrm{Gr}_{G_F},\lnor^+G_F) \longrightarrow  \mscr\perv_{\lnor^+G_k}(\mathrm{Gr}_{G_k},\lnor^+G_k) 
\end{equation*}
restricts to a functor
 \begin{equation*}
    e_* \colon \mscr\perv^{e^*(\scal)}_{\lnor^+G_F}(\mathrm{Gr}_{G_F},\lnor^+G_F) \longrightarrow  \mscr\perv^{\scal}_{\lnor^+G_k}(\mathrm{Gr}_{G_k},\lnor^+G_k).
\end{equation*}
Again, we just need to verify that $e_*(\mathbf{IC}_{\mu}^{\lnor^+G_F}(p_{\mu}^*(e^*(L)))$ lies in the right hand side for any $L \in \scal$. However, by corollary \ref{cor: intersection motives commute with galois extensions}, $e_*(\mathbf{IC}_{\mu}^{\lnor^+G_F}(p_{\mu}^*(e^*(L)))= \bigoplus_{g \in \Gal(F/k)} \mathbf{IC}_{\mu}^{\lnor^+G}(p_{\mu}^*(L))$ so we win.
\end{proof}
\subsection{The Tannakian structure (the fiber functor)}

Note that by the work of Nori, the category $\mscr\perv(k)$ is endowed with a neutral Tannakian structure whose dual group is the motivic Galois group. Let us briefly review the construction (for details see \cite{huber-book}). The category $\mscr\perv(k)$ can be identified with the category $\mathbf{HM}(k)$ (see \cite[Prosition 2.11]{florian+morel-2019}; see also \cite{florian-2017}) built from the quiver $\operatorname{Pair}_k$ whose objects are triplets $(X,Y,n)$ with $X$ a $k$-variety, $Y \subset X$ a closed subvariety and $n \in \mathbb{Z}$. After a technical modification, the natural tensor product
\begin{equation*}
    (X,Y,n) \otimes (X',Y',n') = (X \times X', Y \times X' \cup X \times Y', n+n')
\end{equation*}
becomes a genuine product. The category $\mathbf{HM}(k)$ is a neutral Tannakian category whose fiber functor is 
\begin{equation*}
    \begin{split}
      \hbf^*_{\textnormal{Nori}} \colon   \mathbf{HM}(k) & \longrightarrow \operatorname{Vect}_{\mathbb{Q}} \\ 
        (X,Y,n) & \longmapsto \hnor^n(X(\mathbb{C}),Y(\mathbb{C}),\mathbb{Q}). 
    \end{split}
\end{equation*}
In this section, we show that the $\mscr\perv_{\lnor^+G}(\mathrm{Gr}_G,\lnor^+G)$ is a neutral Tannakian category. 
\begin{defn}
Let $G$ be a connected reductive group over $k$, we define the \textit{motivic fiber functor} to be the composition
\begin{equation*}
  \omega^{\mot} \colon  \mscr\perv^{\scal}_{\lnor^+G} (\mathrm{Gr}_G,\lnor^+G) \overset{\textnormal{forgetful}}{\longrightarrow} \mscr\perv(\mathrm{Gr}_G,\lnor^+G) \overset{\epsilon_!}{\longrightarrow} \derivednori^b(\mscr\perv(k))  \overset{\mathrm{Gr}^{\ct}}{\longrightarrow} \mscr\perv(k) \overset{\hbf^*_{\textnormal{Nori}}}{\longrightarrow} \mathrm{Vect}_{\mathbb{Q}}.
\end{equation*}
where 
\begin{align*}
    \mathrm{Gr}^{\ct}(M) = \bigoplus_{i \in \mathbb{Z}}\cthnor^i(M)
\end{align*}
is the total constructible cohomology functor (viewed as a complex with trivial differentials). 
\end{defn}

\begin{lem} \label{lem: compatibility of fiber functors} 
There is a commutative diagram 
\begin{equation*}
    \begin{tikzcd}[sep=large]
            \mscr\perv^{\scal}_{\lnor^+G} (\mathrm{Gr}_G,\lnor^+G) \arrow[d,"\rat"] \arrow[r,"\omega^{\mot}"] & \operatorname{Vect}_{\mathbb{Q}} \\ 
             \perv_{\lnor^+G}(\mathrm{Gr}_G,\lnor^+G,\mathbb{Q}) \arrow[ru,"\omega^{\betti}",swap] & 
    \end{tikzcd}
\end{equation*}
where $\omega^{\textnormal{Betti}}$ is the ordinary fiber functor of the Satake equivalence in \cite{mirkovic+vilonen-2007}.
\end{lem}

\begin{proof}
  It suffices to show that $\omega^{\betti}$ can be expressed as the following composition 
  \begin{equation*}
  \omega^{\betti} \colon \perv_{\lnor^+G}(\mathrm{Gr}_G,\lnor^+G) \overset{\textnormal{forgetful}}{\longrightarrow} \perv(\mathrm{Gr}_G,\lnor^+G) \overset{\epsilon_!}{\longrightarrow} \derivedcat^b(\mathbb{Q}) \overset{\mathrm{Gr}^{\ct}}{\longrightarrow} \perv(k) = \mathrm{Vect}_{\mathbb{Q}}.
\end{equation*}
but this is obvious since the ordinary fiber functor is simply $\omega^{\betti}(M) = \bigoplus_{n \in \mathbb{Z}}\cthnor^n(\mathrm{Gr}_G,M)$ (here we use hypercohomology).
\end{proof}

\begin{theorem} \label{thm: fiber functor}
The category $\mscr\perv^{\scal}_{\lnor^+G} (\mathrm{Gr}_G,\lnor^+G)$ endowed with the fiber functor 
\begin{equation*} 
\omega^{\mot} = \omega^{\mot}_{\scal} \colon \mscr\perv^{\scal}_{\lnor^+G} (\mathrm{Gr}_G,\lnor^+G) \longrightarrow \mathrm{Vect}_{\mathbb{Q}}
\end{equation*}
is a neutral Tannakian category over $\mathbb{Q}$.
\end{theorem}

\begin{proof}
We divide the proof into several steps, following the definition of a neutral Tannakian category.
\begin{enumerate}
    \item (Monoidal Structure) By lemma \ref{lem: compatibility of fiber functors}, this is obvious since $\rat$ and $\omega^{\betti}$ are monoidal.
    \item ($\mathbb{Q}$-linearity, exactness, faithfulness) These all follow from lemma \ref{lem: compatibility of fiber functors}. 
    \item (Commutativity Constraints) This is the content of proposition \ref{prop: commutativity constraints}. 
    \item (Neutral Object) By proposition \ref{prop: unit object}, we know that \begin{equation*}
    \begin{split} 
    \Hom_{\mscr\perv^{\scal}_{\lnor^+G} (\mathrm{Gr}_G,\lnor^+G)}(\inters_0^{\lnor^+G}(\mathds{1}),\inters_0^{\lnor^+G}(\mathds{1})) & = \Hom_{\mscr\perv(\mathrm{Gr}_G,\lnor^+G)}(\inters_0(\mathds{1}),\inters_0(\mathds{1}))  \\ 
    & \subset \Hom_{\perv(\mathrm{Gr}_G,\lnor^+G)}(\inters_0(\mathds{1}),\inters_0(\mathds{1})) = \mathbb{Q} 
    \end{split} 
    \end{equation*} 
    and this group is nonzero and hence must equal $\mathbb{Q}$ as well.
    \item (Dual Objects) Let $M \in \mscr\perv^{\scal}_{\lnor^+G} (\mathrm{Gr}_G,\lnor^+G)$ be a motive such that $\dim_{\mathbb{Q}}(\omega^{\mot}_{\scal}(M))=1$, we claim that $M$ admits a strong dual. By the faithfulness, $M$ must be indecomposable and hence of the form $\mathbf{IC}_{\mu}^{\lnor^+G}(p_{\mu}^*(L)) = \mathbf{IC}_0^{\lnor^+G}(L) \star \mathbf{IC}_{\mu}^{\lnor^+G}(\mathds{1})$ with $\mu \in X_*(T)^+$ and $L \in \mscr\localsystem(k)$ an indecomposable motivic local system. We note that by proposition \ref{prop: unit object} and monoidality
    \begin{equation*}
    \begin{split} 
        \omega^{\mot}(\mathbf{IC}_{\mu}^{\lnor^+G}(L)) & = \omega^{\mot}(\mathbf{IC}_{0}^{\lnor^+G}(L)) \otimes \omega^{\mot}(\mathbf{IC}_{\mu}^{\lnor^+G}(\mathds{1})) \\ 
        & = \omega^{\betti}(\mathbf{IC}_0(\rat(L))) \otimes \omega^{\betti}(\mathbf{IC}_{\mu}(\mathds{1})) 
        \end{split} 
    \end{equation*}
    and hence
    \begin{equation*}
        \dim  \omega^{\betti}(\mathbf{IC}_0(\rat(L))) = \dim \omega^{\betti}(\mathbf{IC}_{\mu}(\mathds{1})) = 1.
    \end{equation*}
    We see that $L$ is $\otimes$-invertible, i.e., $L \otimes L^{\vee} \simeq \mathds{1}$. Indeed, $\rat(L) \in \localsystem(k)= \operatorname{Vect}_{\mathbb{Q}}$ and $L$ is indecomposable hence $\rat(L) = \mathbb{Q}$ so $\otimes$-invertible. Now by the classical Satake equivalence, one sees that $\rat(\mathbf{IC}_0(L))$ is $\otimes$-invertible, namely, 
    \begin{equation*}
        \mathbf{IC}_0(\rat(L)) \bullet \mathbf{IC}_0(\rat(L^{\vee})) \simeq \mathds{1}.
    \end{equation*}
    This implies that 
      \begin{equation*}
        \mathbf{IC}_0^{\lnor^+G}(L) \bullet \mathbf{IC}_0^{\lnor^+G}(L^{\vee}) \simeq \simeq \mathds{1}.
    \end{equation*}
   thanks to the conservativity of realization. Regarding $\mathbf{IC}_{\mu}^{\lnor^+G}(\mathds{1})$, one has $\dim \omega^{\mot}(\mathbf{IC}_{\mu}(\mathds{1})) = \dim \omega^{\betti}(\mathbf{IC}_{\mu}(\mathbb{Q})) = 1$ so according to \cite[Proposition 5.13]{baumann+riche-2018}, $\mu$ is orthorgonal to all roots in $X_*(T)$ and $\mathrm{Gr}^{\mu}_G$ is a singleton so that $\mathbf{IC}_{\mu}^{\lnor^+G}(\mathds{1}) \star \mathbf{IC}_{-\mu}^{\lnor^+G}(\mathds{1}) = \mathbf{IC}_0^{\lnor^+G}(\mathds{1})$ thanks to lemma \ref{lem: triangular decomposition of intersection motives}.
   % but $L$ is strongly dualisable (see proposition \ref{prop: motivic local systems are independent of realizations}) so $\omega(\mathbf{IC}_0(L))$ is strongly dualisable as well (here we used that $\star$ is mapped to $\otimes$ in proposition \ref{prop: Soegel's bimodules}). Moreover, by lemma \ref{lem: compatibility of fiber functors}, $\omega(\mathbf{IC}_{\mu}(\mathds{1})) = \omega^{\betti}(\mathbf{IC}_{\mu}(\mathds{1})) = 1$ and hence by \cite[Proposition 5.13]{baumann+riche-2018}, this implies $\lambda$ is orthogonal to all roots in $X_*(T)$ and hence $\mathrm{Gr}^{\mu}_G$ is a singleton so that we can take a dual object by virtues of lemma \ref{lem: triangular decomposition of intersection motives} as follows: indeed, $\mathbf{IC}_{\mu}(\mathds{1}) \star \mathbf{IC}_{-\mu}(\mathds{1}) = \mathbf{IC}_0(L)$ for some indecomposable motive $L \in \mscr\localsystem(\mathrm{Gr}_G^0) = \mscr\localsystem(k)$ and again $\mathbf{IC}_{-\mu}(\mathds{1}) \star \mathbf{IC}_0(L^{\vee})$ does the job (\color{red}explain why $\mathbf{IC}_0(L^{\vee}) \otimes \mathbf{IC}_0(L) \simeq \mathds{1}$). 
\end{enumerate}
\end{proof}
\subsection{The dual groups}
In the previous section, we know that the $\scal$-\textit{motivic Satake category} $\mscr\perv_{\lnor^+G}^{\scal}(\mathrm{Gr}_G,\lnor^+G)$ is a neutral Tannakian category. In ths section, we identity its dual group, the the $\scal$-\textit{motivic Langlands dual group}, denoted by $\gscr^{\mot}_{\lnor^+G}(G,\scal)^{\vee}$. Furthermore, we also study its dual groups when restricting to suitable subcategories of motives. Let $\scal \subset \mscr\localsystem(k)$ be a full abelian category containing $\mathds{1}$, closed under tensor product and taking duals. 

\begin{theorem}
Let $k$ be a subfield of $\mathbb{C}$ and $G$ be a reductive group over $k$. Then there are equivalences of Tannakian categories
\begin{equation*} 
(\operatorname{Sat}^{\scal}_{G,k},\star,\omega^{\mot}_{\scal}) \simeq (\operatorname{Rep}_{\mathbb{Q}}^{\textnormal{fd}}(G^{\vee}_{\mathbb{Q}} \times \gscr^{\mot}_{\scal}(k)), \otimes, \textnormal{forgetful}).
\end{equation*}
Moreover, let $\scal_1 \subset \scal_2$ be an inclusion of Tannakian subcategories of $\mscr\localsystem(k)$, there is a canonical commutative diagram 
\begin{equation*}
    \begin{tikzcd}[sep=large]
      (\operatorname{Sat}^{\scal_1}_{G,k},\star,\omega^{\mot}_{\scal}) \arrow[r] \arrow[d] &  \arrow[d] (\operatorname{Rep}_{\mathbb{Q}}^{\textnormal{fd}}(G^{\vee}_{\mathbb{Q}} \times \gscr^{\mot}_{\scal_1}(k)), \otimes, \textnormal{forgetful}) \\ 
      (\operatorname{Sat}^{\scal_2}_{G,k},\star,\omega^{\mot}_{\scal})  \arrow[r]  & (\operatorname{Rep}_{\mathbb{Q}}^{\textnormal{fd}}(G^{\vee}_{\mathbb{Q}} \times \gscr^{\mot}_{\scal_2}(k)), \otimes, \textnormal{forgetful}).
    \end{tikzcd}
\end{equation*}
If $e \colon \Spec(F) \longrightarrow \Spec(k)$ be an extension in $\mathbb{C}$, then there is a canonical commutative diagram 
\begin{equation*}
    \begin{tikzcd}[sep=large]
      (\operatorname{Sat}^{\scal}_{G_k,k},\star,\omega^{\mot}_{\scal}) \arrow[r] \arrow[d] &  \arrow[d] (\operatorname{Rep}_{\mathbb{Q}}^{\textnormal{fd}}(G^{\vee}_{\mathbb{Q}} \times \gscr^{\mot}_{\scal}(k)), \otimes, \textnormal{forgetful}) \\ 
      (\operatorname{Sat}^{\scal}_{G_F,F},\star,\omega^{\mot}_{e^*(\scal)})  \arrow[r]  & (\operatorname{Rep}_{\mathbb{Q}}^{\textnormal{fd}}(G^{\vee}_{\mathbb{Q}} \times \gscr^{\mot}_{e^*(\scal)}(F)), \otimes, \textnormal{forgetful}).
    \end{tikzcd}
\end{equation*}
\end{theorem}

\begin{proof}
Let $R$ be a $\mathbb{Q}$-algebra, then by definition
  \begin{equation*}
     \gscr^{\mot}_{\lnor^+G}(G,\scal)(R) =  \operatorname{Aut}^{\star}(\omega)(R) = \left \{\textnormal{$R$-linear automorphism} \ g_X \colon \omega^{\mot}_{\scal}(X) \otimes_{\mathbb{Q}} R \longrightarrow \omega^{\mot}_{\scal}(X) \otimes_{\mathbb{Q}} R + \textnormal{compatibilities} \right \}. 
  \end{equation*}
  The morphism $\gscr^{\mot}_{\lnor^+G}(G,\scal) \longrightarrow G_{\mathbb{Q}}^{\vee} \times \gscr^{\mot}_{\scal}(k)$ is defined as follows: we write $X = \bigoplus_{i=1}^n \mathbf{IC}^{\lnor^+G}_{\mu_i}(p_{\mu_i}^*(L_i))$ with $L_i \in \mscr\localsystem(k)$ indecomposable motivic local systems, then
  \begin{equation*}
      \begin{split}
          \gscr^{\circ}_{\mathbb{Q}}(R) & \longrightarrow G^{\vee}_{\mathbb{Q}}(R) \times \gscr^{\mot}(k)(R) \\ g_X & \longmapsto \bigoplus_{i=1}^n  (g_{\mathbf{IC}_{\mu}^{\lnor^+G}(\mathds{1})} \otimes g_{L_i}).
      \end{split}
  \end{equation*}
  We claim that this is an isomorphism of functors (hence of corresponding pro-algebraic groups). From the explicit description, this map is clearly injective. To show the surjectivity, it is enough to show that if
  \begin{equation*}
     \mathbf{IC}^{\lnor^+G}_{\mu}(p_{\mu}^*(L_{\mu})) =  \mathbf{IC}^{\lnor^+G}_{\lambda}(p_{\lambda}^*(L_{\lambda}))
  \end{equation*}
  then $\mu = \lambda$ and $L_{\mu} \simeq L_{\lambda}$. Indeed, $\mu=\mu'$ by support and $p_{\mu}^*(L_{\mu}) \simeq p_{\mu}^*(L_{\lambda})$ obtained by restricting both sides to $\mathrm{Gr}_G^{\mu}$ but since $p_{\mu}$ admits an equivariant section $\Spec(k) \longrightarrow \mathrm{Gr}^{\leq \mu}$, this implies $L_{\mu} = L_{\lambda}$. 
\end{proof}

\begin{ex}
As a consequence, we have different versions of the geometric Satake equivalence, depending on the choice we make.
\begin{enumerate}
    \item If $\scal = \mscr\localsystem(k)$, then the dual group becomes $G^{\vee}_{\mathbb{Q}} \times \gscr^{\mot}(k)$. 
    \item If $\scal$ is the category pure motives, the dual group becomes $G^{\vee}_{\mathbb{Q}} \times \gscr^{\pure}(k)$, where $\gscr^{\pure}(k)$ is the pure motivic Galois group, see for instance \cite[Chapter 10]{huber-book}. 
    \item If $\scal$ is the category of Tate motives, the dual group becomes  $G^{\vee}_{\mathbb{Q}} \times \gscr^{\textnormal{Tate}}(k)$, where $\gscr_{\textnormal{Tate}}(k)$ is the motivic Galois group of mixed Tate motives.  If we restrict to pure Tate motives, then it becomes $G^{\vee}_{\mathbb{Q}} \times \mathbb{G}_{m,\mathbb{Q}}$. 
    \item If $\scal$ is the category of Artin-Tate motives, the dual group becomes $G^{\vee}_{\mathbb{Q}} \times \Gal(\overline{k}/k) \times \gscr^{\textnormal{Tate}}(k)$. If we restrict to pure Artin-Tate motives, then it becomes $G^{\vee}_{\mathbb{Q}} \times \Gal(\overline{k}/k) \times \mathbb{G}_{m,\mathbb{Q}}$. 
    \item Let $X/k$ be a connected, smooth variety and $x \in X(k)$. Let $\pi_1^{\mot}(X,x)$ be the motivic fundamental group. Let $\scal$ be the smallest Tannakian closure containing the image of $\mscr\localsystem(X)$ in $\mscr\localsystem(k)$, then the dual group is $G^{\vee}_{\mathbb{Q}} \times \pi_1^{\mot}(X,x)$. 
\end{enumerate}
\end{ex}

\begin{cor}
Let $t$ be an intermediate, then there are isomorphisms of $\mathbb{Z}[t^{\pm 1/2}]$-algebras
\begin{align*}
    K_0(\operatorname{Sat}_{G,k}^{\mscr\operatorname{Tate}_p(k)}) &  \simeq \mathbb{Z}[t^{\pm 1/2}][\left \{ \mathbf{IC}_{\mu}^{\lnor^+G}(\mathds{1}) \mid \mu \in X_*(T)^+ \right \}] \\ 
    & \simeq K_0(\operatorname{Rep}(G^{\vee})) \\ 
    \end{align*} 
and 
\begin{align*} 
    K_0(\operatorname{Sat}_{G,k}^{\mscr\localsystem(k)}) & \simeq  K_0(\operatorname{Sat}_{G,k}^{\mscr\localsystem(k,\pure)}) \\ 
    & \simeq  \mathbb{Z}[t^{\pm 1/2}][\left \{ \mathbf{IC}_{\mu}^{\lnor^+G}(L) \mid \mu \in X_*(T)^+ \right \}] \\ 
    & \simeq K_0(\operatorname{Rep}(G^{\vee})) \times K_0(\operatorname{Rep}(\gscr^{\mot}(k)))
\end{align*}
In particular, one has that
\begin{equation*}
    K_0(\operatorname{Sat}_{1,k}^{\mscr\localsystem(k)}) \simeq \mathbb{Z} \times K_0(\operatorname{Rep}(\gscr^{\mot}(k))).
\end{equation*}
\end{cor}

\begin{proof}
The first isomorphism follows from the classical Satake isomorphism (see for instance, \cite[Theorem 9.4.1]{achar-book}). The two next isomorphisms are obvious thanks to the existence of weight filtrations. 
\end{proof}

\bibliographystyle{alpha}
\bibliography{ref}

%\makeaddress
\end{document}